\documentclass[a4paper,10pt]{article}

\usepackage{ifthen}
\usepackage{amssymb}
\usepackage{amsmath}
\usepackage{mathrsfs}
\usepackage{theorem}
\usepackage[all]{xypic}
\usepackage{color}

	{
		\theorembodyfont{\itshape}

		\newtheorem{theorem}{Theorem}[section]
		\newtheorem{lemma}[theorem]{Lemma}
		\newtheorem{proposition}[theorem]{Proposition}
		\newtheorem{corollary}[theorem]{Corollary}
	}

	{
		\theorembodyfont{\rmfamily}
		\newtheorem{definition}[theorem]{Definition}
		
		\newtheorem{remark}[theorem]{Remark}
	}

	\newenvironment{proof}{
		\goodbreak\par
		\textit{Proof.}%
	}{%
		\nopagebreak
		\hfill{\vrule width 1ex height 1ex depth 0ex}
		\medskip
		\goodbreak
	}

	\newcommand{\sizedescriptor}[2]
	{
		\ifthenelse{\equal{#1}{0}}{}{
		\ifthenelse{\equal{#1}{1}}{\big}{
		\ifthenelse{\equal{#1}{2}}{\Big}{
		\ifthenelse{\equal{#1}{3}}{\bigg}{
		\ifthenelse{\equal{#1}{4}}{\Bigg}{
		#2}}}}}
	}

	\newcommand{\proven}[1]{\underline{#1}\vspace{0.2em}\\}

	\newcommand{\ie}[1][~]{i.e.{#1}}
	
	\newcommand{\eg}[1][~]{e.g.{#1}}
	
	\newcommand{\etc}[1][~]{etc.{#1}}

	\newcommand{\df}[1]{\emph{#1}}
	\newcommand{\ism}{\cong}
	\newcommand{\equ}{\sim}
	\newcommand{\dfeq}{\mathrel{\mathop:}=}

	\newcommand{\apart}{\mathrel{\#}}
	\newcommand{\id}[1][]{\textrm{Id}_{#1}}

	\newcommand{\impl}{\Rightarrow}

	\newcommand{\rstr}[1]{\left.{#1}\right|}
	\newcommand{\cnct}{{{:}{:}}}
	\newcommand{\im}{\text{im}}
	\newcommand{\ball}[3][]{B_{#1}\left(#2, #3\right)}
	
	\newcommand{\insarg}{\text{\textrm{---}}}

	\newcommand{\monus}{\mathbin{\dot{\relbar}}}

	\newcommand{\all}[4][auto]{\forall\, #2 \,{\in}\, #3\,.\sizedescriptor{#1}{\left}({#4}\sizedescriptor{#1}{\right})}
	\newcommand{\some}[4][auto]{\exists\, #2 \,{\in}\, #3\,.\sizedescriptor{#1}{\left}({#4}\sizedescriptor{#1}{\right})}

	\newcommand{\xall}[3]{\forall\, #1 \,{\in}\, #2\,.\,#3}
	\newcommand{\xsome}[3]{\exists\, #1 \,{\in}\, #2\,.\,#3}

	\newcommand{\st}[3][auto]{\sizedescriptor{#1}{\left}\{#2\;\sizedescriptor{#1}{\middle}|\;#3\sizedescriptor{#1}{\right}\}}

	\newcommand{\finseq}[1]{{#1}^*}
	\newcommand{\fin}{\mathcal{F}}
	\newcommand{\inhfin}{\fin^{+}}

	\newcommand{\NN}{\mathbb{N}}
	\newcommand{\ZZ}{\mathbb{Z}}
	\newcommand{\QQ}{\mathbb{Q}}
	\newcommand{\RR}{\mathbb{R}}
	
	\newcommand{\intoo}[3][\RR]{{#1}_{(#2, #3)}}
	\newcommand{\intcc}[3][\RR]{{#1}_{[#2, #3]}}
	
	\newcommand{\intco}[3][\RR]{{#1}_{[#2, #3)}}
	
	\newcommand{\lr}{\underrightarrow{\RR}}
	
	\newcommand{\ur}{\underleftarrow{\RR}}

	\newcommand{\kat}[1]{\mathbf{\underline{#1}}}

	\newcommand{\dom}{\text{dom}}
	\newcommand{\cod}{\text{cod}}
	\newcommand{\zero}{\mathbf{0}}
	\newcommand{\one}{\mathbf{1}}
	\newcommand{\unit}{*}
	\newcommand{\two}{\mathbf{2}}
	\newcommand{\trm}[1][]{\mathord{!}_{#1}}
	\newcommand{\ini}[1][]{\mathord{\text{\textnormal{!`}}}_{#1}}

	\newdir{ >}{{}*!/-8pt/@{>}}
	\newdir{ (}{{}*!/-5pt/@^{(}}
	\newdir{|>}{!/5pt/@{|}*:(1,-.2)@{>}*:(1,+.2)@_{>}}

	\newcommand{\cat}{\kat{C}}
	\newcommand{\Set}{\kat{Set}}
	\newcommand{\Top}{\kat{Top}}
	\newcommand{\Grp}{\kat{Grp}}
	\newcommand{\Pstr}{\kat{Pstr}}
	\newcommand{\Apstr}{\kat{Apstr}}
	\newcommand{\Str}{\kat{Str}}
	\newcommand{\Ccstr}{\kat{Ccstr}}
	\newcommand{\FormSp}{\kat{FormSp}}

	\newcommand{\pst}{\mathcal{P}}
	\newcommand{\opn}{\Sigma}
	\newcommand{\optp}{\mathcal{O}}
	\newcommand{\cltp}{\mathcal{Z}}

	\newcommand{\lnth}{\text{lnth}}

	\newcommand{\lpo}[1][~]{\textbf{LPO}{#1}}
	\newcommand{\llpo}[1][~]{\textbf{LLPO}{#1}}

	\newcommand{\nap}{\approx}
	\definecolor{lightgray}{rgb}{0.75,0.75,0.75}

	\newcommand{\inford}{\sqsubseteq}
	\newcommand{\covby}{\triangleleft}
	\newcommand{\pospred}[2][]{\lozenge_{#2}{#1}}
	\newcommand{\formal}[1]{\mathbf{#1}}
	\newcommand{\formalgtQ}[1][\formal{R}]{{{}_{\formal{Q}}\formal{<}_{#1}}}
	\newcommand{\formalltQ}[1][\formal{R}]{{{}_{#1}\formal{<}_{\formal{Q}}}}
	\newcommand{\embgtQ}[1][\formal{R}]{{{}_{\formal{Q}}lt_{#1}}}
	\newcommand{\embltQ}[1][\formal{R}]{{{}_{#1}lt_{\formal{Q}}}}
	\newcommand{\smR}{\mathbf{R}}
	\newcommand{\lowcl}{\mathop{\downarrow}}
	
	\newcommand{\pos}[1]{{#1}^+}
	\newcommand{\arch}{\mathsf{Arch}}
	\newcommand{\ring}{\mathsf{Ring}}
	\newcommand{\field}{\mathsf{Field}}
	\newcommand{\dc}{\mathscr{D}}
	\newcommand{\lc}{\mathscr{L}}
	\newcommand{\uc}{\mathscr{U}}

\title{Unified Approach to Real Numbers in Various Mathematical Settings}

\author{
Davorin Le\v{s}nik\vspace{1ex}\\
Department of Mathematics\\
Darmstadt University of Technology, Germany\\
\texttt{lesnik@mathematik.tu-darmstadt.de}
}

\begin{document}

	\maketitle
	
	\begin{abstract}
		We provide a setting-independent definition of reals by introducing the notion of a \df{streak}. We show that various standard constructions of reals satisfy our definition. We study the structure of reals by noting that its pieces correspond to reflections on the category of streaks.
	\end{abstract}


	\section{Introduction}\label{Section: introduction}
		
		Real numbers form one of the most important sets in mathematics. There are different definitions of reals, the most common one probably being: \emph{$\RR$ is a Dedekind complete ordered field}. It turns out that in classical mathematics this definition determines $\RR$ up to isomorphism (so we have uniqueness), and one can construct models, satisfying this definition (thus proving existence), such as Cauchy reals (equivalence classes of Cauchy sequences of rationals), Dedekind reals (Dedekind cuts) etc.
		
		If there is little doubt, what reals are classically, the answer is not so clear-cut in various constructive settings. Constructively the definition ``Dedekind complete ordered field'' doesn't work (for example, it does not imply the archimedean property) and different constructions of reals don't necessarily yield isomorphic sets (for example, depending on the constructive setting, there can be more Dedekind reals than Cauchy ones).
		
		It is not apparent that one construction would be inherently better than the other, and one usually uses the one which behaves best in the current setting. This paper is an attempt to make a unifying definition for reals in a wide variety of settings, by making only some very basic assumptions on the background theory.
		
		Here is the definition, explained in very informal terms (the precise formulation is given in Definition~\ref{Definition: reals} and the definitions leading to it). Consider sets, equipped with an archimedean linear order $<$, addition $+$ and, defined at least on positive elements, mutiplication $\cdot$ (we name this a \df{streak}). Clearly the reals, whatever they are, ought to be an example of such a set (of course, natural numbers $\NN$, integers $\ZZ$ and rationals $\QQ$ ought to be examples also). The fact, that the order $<$ is linear and archimedean, essentially amounts to such a set lying on the real line. Hence, $\RR$ can be characterised as the largest among such sets.
		
		The point of such a definition is to avoid speaking about \df{completion}, as different notions of completion (\eg order completion, metric completion) need not coincide constructively. Replacing `complete' with `largest' is a trick which works for general metric spaces as well: the completion of a metric space is the largest metric space, into which the original one can be densely isometrically embedded (for a precise formulation, see \eg[~]\cite{Lesnik_D_2010:_synthetic_topology_and_constructive_metric_spaces}).
		
		Of course, one must formalize, what `largest' means in this context. The tool we'll use for this (as well as several other things) is the \df{universal property} from category theory. In fact, nearly the entirety of this paper is heavily influenced by ideas from category theory. However, no particular prior knowledge of categories is required; I make an effort to translate everything into noncategorical terms (though I still mention the categorical interpretations for readers, familiar with them). What category theory is explicitly needed, is explained in Subsection~\ref{Subsection: categories} and at the beginning of Section~\ref{Section: reflections}.
		
		The paper is more ambitious than just providing a general definition for reals, however (for that, we'd need little more than Definitions~\ref{Definition: streak} and~\ref{Definition: reals}). Additionally we analize various (order, algebraic, topological, limit, completness) structures, typical for reals --- how they can be added one by one, how they fit together, why the reals must necessarily have them.
		
		Here is the exact breakdown of the paper.
		\begin{itemize}
			\item\textbf{Section~\ref{Section: introduction}:} \underline{Introduction}\\
				Overview of the paper, notation and a primer on category theory (as much as is used in the paper).
			\item\textbf{Section~\ref{Section: setting}:} \underline{Setting}\\
				The purpose of this paper is to work with reals in a very general setting, encompasing the various particular ones used in mathematical practice, and in this section such a setting is described. Also, some basic corollaries of our axioms are derived.
			\item\textbf{Section~\ref{Section: streaks}:} \underline{Prestreaks and streaks}\\
				We define the crucial tool we use to study the structure of reals --- streaks (and their morphisms). This is done in stages, with more general strict orders, prestreaks and archimedean prestreaks defined first. Later we consider special kind of streaks --- multiplicative and dense ones.
			\item\textbf{Section~\ref{Section: reflections}:} \underline{Reflective structures}\\
				We observe that various pieces of the structure of reals correspond to (co)reflections on the category of streaks (or related categories).
			\item\textbf{Section~\ref{Section: reals}:} \underline{Real numbers}\\
				We use the preceding theory to formally define the set of reals and observe some immediate properties.
			\item\textbf{Section~\ref{Section: models_of_reals}:} \underline{Models of reals}\\
				We verify that various standard constructions of reals satisfy our definition in their respective settings.
			\item\textbf{Section~\ref{Section: additional}:} \underline{Additional examples}\\
				There are some structures, closely related to reals, but not actually isomorphic to them. Here we discuss how they fit into our theory.
		\end{itemize}
		
		\subsection{Notation}
			
			\begin{itemize}
				\item
					Number sets are denoted by $\NN$ (natural numbers), $\ZZ$ (integers), $\QQ$ (rationals), and $\RR$ (reals). Zero is considered a natural number (so $\NN = \{0, 1, 2, 3,\ldots\}$).
				\item
					Subsets of number sets, obtained by comparison with a certain number, are denoted by the suitable order sign and that number in the index. For example, $\NN_{< 42}$ denotes the set $\st{n \in \NN}{n < 42} = \{0, 1, \ldots, 41\}$ of all natural numbers smaller than $42$, and $\RR_{\geq 0}$ denotes the set $\st{x \in \RR}{x \geq 0}$ of non-negative real numbers. The apartness relation $\apart$ is used in the similar way.
				\item
					Intervals between two numbers are denoted by these two numbers in brackets and in the index. Round, or open, brackets $(\ )$ denote the absence of the boundary in the set, and square, or closed, brackets $[\ ]$ its presence; for example $\intco[\NN]{5}{10} = \st{n \in \NN}{5 \leq n < 10} = \{5, 6, 7, 8, 9\}$ and $\intcc{0}{1} = \st{x \in \RR}{0 \leq x \leq 1}$.
				\item
					Given a map $a\colon N \to A$ where $N$ is a subset of natural numbers, we often write simply $a_k$ instead of $a(k)$ for the value of $a$ at $k \in N$.
				\item
					The set of maps from $A$ to $B$ is written as the exponential $B^A$.
				\item
					The set of finite sequences of elements in $A$ is denoted by $\finseq{A}$.
				\item
					Given sets $A \subseteq X$, $B \subseteq Y$ and a map $f\colon X \to Y$ with the image $\im(f) \subseteq B$, the restriction of $f$ to $A$ and $B$ is denoted by $\rstr{f}_A^B$. When we restrict only the domain or only the codomain, we write $\rstr{f}_A$ and $\rstr{f}^B$, respectively.
				\item
					A one-element set (a singleton) is denoted by $\one$ (and its sole element by $*$).
				\item
					The onto maps are called surjective, and the one-to-one maps injective.
				\item
					The quotient of a set $X$ by an equivalence relation $\equ$ is denoted by $X/_\equ$. Its elements --- the equivalence classes --- are denoted by $[x]$ where $x \in X$ (\ie if $q\colon X \to X/_\equ$ is the quotient map, then $[x] \dfeq q(x)$).
				\item
					The coproduct (disjoint union) is denoted by $+$ in the binary case, and by $\coprod$ in the general case.
			\end{itemize}
		
		\subsection{Categories}\label{Subsection: categories}
		
			This subsection provides a very brief introduction to category theory (the definition and some instances of the universal property). For a more serious introduction to the topic, consider~\cite{lane1998categories}; however, additional knowledge is not required for understanding this paper (though it helps to understand some issues better; I often make additional remarks, how something can be seen through the categorical lens).
			
			Informally, a category is a collection of objects (whatever we are interested in) and some maps between them (called `morphisms' or `arrows'), typically those which in some way preserve whatever structure the objects have. Examples include $\Set$ (the category of sets and maps), $\Top$ (the category of topological spaces and continuous maps), $\Grp$ (the category of groups and group homomorphisms), $\RR\textrm{-}\kat{Vect}$ (the category of real vector spaces and linear maps), $\kat{Pre}$ (the category of preordered sets and monotone maps) etc. (In the paper we define and study the category of \df{streaks} and their morphisms $\Str$.)
			
			The formal definition of a category describes, what exactly is required to call something `morphisms' or `maps'.
			\begin{definition}
				A \df{category} is a pair $\cat = (\cat_0, \cat_1)$ of families of \df{objects} $\cat_0$ and \df{morphisms} (or \df{arrows}) $\cat_1$, together with operations
				\begin{itemize}
					\item
						$\dom, \cod\colon \cat_1 \to \cat_0$, called \df{domain} and \df{codomain},
					\item
						$\id\colon \cat_0 \to \cat_1$, called \df{identity}, and
					\item
						$\circ\colon \cat_2 \to \cat_1$ (where $\cat_2 \dfeq \st{(f, g) \in \cat_1 \times \cat_1}{\dom(f) = \cod(g)}$), called \df{composition},
				\end{itemize}
				such that the following holds for all $X \in \cat_0$, $f, g, h \in \cat_1$:
				\begin{itemize}
					\item
						$\dom(\id[X]) = X = \cod(\id[X])$,
					\item
						$\dom(f \circ g) = \dom(g)$,\ \ $\cod(f \circ g) = \cod(f)$,
					\item
						$f \circ \id[\dom(f)] = f$,\ \ $\id[\cod(f)] \circ f = f$,
					\item
						$(f \circ g) \circ h = f \circ (g \circ h)$\ \ (in the sense that whenever one side is defined, so is the other, and then they are equal).
				\end{itemize}
				When $\dom(f) = X$ and $\cod(f) = Y$, we denote this by writing $f\colon X \to Y$.
				
				When both $\cat_0$ and $\cat_1$ are sets, the category $\cat$ is called \df{small}, otherwise it is \df{large}.
			\end{definition}
			
			Clearly the various examples of structured sets and structure-preserving maps above satisfy this definition. However, the definition of morphisms is abstract; they need not be actual maps. For example, any preorder $(P, \leq)$ (that is, $\leq$ is reflexive and transitive) is an example of a category ($P$ is the set of objects, and $\leq$, viewed as a subset  of $P \times P$, is the set of morphisms, with $\dom$ and $\cod$ being projections), specifically the kind, where we have at most one arrow from one object to another. Conversely, any such category determines a preorder (existence of identities corresponds to reflexivity, and existence of compositions to the transitivity of the preorder). Hence we define that a \df{preorder category} is a category, in which for every two (not necessarily distinct) objects there is at most one morphism from the first to the second one.
			
			An arrow $f\colon X \to Y$ in a category $\cat$ is called an \df{isomorphism} when it has an inverse with regard to composition, \ie when there exists an arrow $g\colon Y \to X$ in $\cat$ such that $g \circ f = \id[X]$ and $f \circ g = \id[Y]$. By the standard argument the inverse of $f$ is unique, and we denote it by $f^{-1}$. Two objects are called \df{isomorphic} when there exists an isomorphims between them (we write $X \ism Y$); this is an equivalence relation on objects.
			
			In $\Set$, isomorphisms are bijections; in $\Top$, the homeomorphisms; in $\kat{Grp}$, the usual group isomorphisms; and so on.
			
			In any category identities are always isomorphisms. The converse in a preorder category is equivalent to the antisymmetry of the preorder; thus we define that a \df{partial order category} is a preorder category, in which the only isomorphisms are identities.
			
			An object $\zero$ in a category is called \df{initial} when for every object $X$ there exists exactly one morphism $\ini[X]\colon \zero \to X$. In $\Set$, the initial object is the empty set.
			
			Note that the definition of a category is self dual in the sense, that if we reverse the direction of all arrows (that is, we switch $\dom$ and $\cod$, and the order of composition) in a category, we get another category (called the \df{dual} or \df{opposite category}). In particular, any categorical notion has its dual notion. While being isomorphic is self-dual, reversing the arrows in the case of an initial objects yields a new notion. An object $\one$ is called \df{terminal} when for every object $X$ there exists exactly one morphism $\trm[X]\colon X \to \one$.
			
			In $\Set$, the terminal objects are precisely the singletons. In particular, a category can have more than one terminal object (or none), but they are all isomorphic: for any two $\one$, $\one'$, both $\trm[\one'] \circ \trm[\one]'$ and $\id[\one]$ are maps $\one \to \one$, and therefore equal (due to uniqueness); similar for $\trm[\one]' \circ \trm[\one']$. Of course, the same applies for initial objects.
			
			Note that in a preorder category an initial object is a smallest element in the preorder, and a terminal object the largest one. This will allow us to formalize the intuition at the beginning of the Introduction: in Section~\ref{Section: streaks} we define \df{streaks} and their morphisms (which together form a preorder category $\Str$), of which the reals are the largest one; thus we define (in Section~\ref{Section: reals}) the reals as the terminal streak (the initial streak is what we start with: the natural numbers). Such a definition determines the reals up to isomorphism (of all relevant structures), which is what we want.
			
			More generally, whenever some objects/morphisms in a category are characterized in a similar way by requiring the existence and uniqueness of some arrow, it is said that they are defined by/possess a particular \df{universal property} (see~\cite{lane1998categories} or simply the Wikipedia article for the exact definition), and this always determines them up to (a canonical choice of) an isomorphism.
			
			Many construction can be given by a universal property, such as categorical products (a product of objects $X$ and $Y$ is denoted by $X \times Y$). Explicitly, an object $X \times Y$, together with maps (``\df{projections}'') $p\colon X \times Y \to X$, $q\colon X \times Y \to Y$, is called the (\df{categorical}) \df{product} when for every object $T$ and all maps $f\colon T \to X$, $g\colon T \to Y$ there exists a unique map $T \to X \times Y$ which we denote by $(f, g)$, such that $p \circ (f, g) = f$ and $q \circ (f, g) = g$. This can be captured by the following diagram.
			
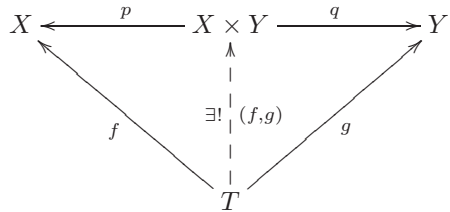
\[\xymatrix@+3em{
				X  &  X \times Y \ar[l]_p \ar[r]^q  &  Y  \\
				&  T \ar[lu]^f \ar[ru]_g \ar@{-->}[u]^{\exists!}_{(f, g)}  &
			}\]
			The dashed arrow represents the morphism of which existence and uniqueness we demand (making this an example of the universal property). The fomulae $p \circ (f, g) = f$ and $q \circ (f, g) = g$ say precisely that any two paths between from one node to another represent the same morphism; we say that this diagram \df{commutes}.
			
			Interpreting this definition in $\Set$, we obtain the usual cartesian product; in $\Top$, the usual topological product; and so on.
			
			This definition can be easily generalized to a product of an arbitrary family of objects $(X_i)_{i \in I}$; the product of such a family is denoted by $\prod_{i \in I} X_i$.
			
			By reversing all the arrows in the definition of the product, we get the definition of (aptly-named) \df{coproduct}, or \df{sum} (denoted $X + Y$). Explicitly, an object $X + Y$, together with morphisms $i\colon X \to X + Y$, $j\colon Y \to X + Y$ is a \df{sum} of $X$, $Y$ when for every object $T$ and all maps $f\colon X \to T$, $g\colon Y \to T$ there exists a unique map $[f, g]\colon X + Y \to T$ which makes the following diagram commute (\ie $[f, g] \circ i = f$, $[f, g] \circ j = g$).
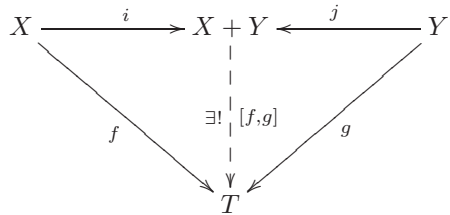
			\[\xymatrix@+3em{
				X \ar[r]^i \ar[dr]_f  &  X + Y \ar@{-->}[d]_{\exists!}^{[f, g]}  &  Y \ar[l]_j \ar[ld]^g  \\
				&  T  &
			}\]
			The coproduct of a family of objects $(X_i)_{i \in I}$ is denoted by $\coprod_{i \in I} X_i$. In $\Set$, coproducts are disjoint unions.
			
			The notation for $\zero$, $\one$, $\times$, $+$ is such not only because of cardinalities of the these objects in $\Set$, but also because the usual laws of arithmetic hold in other categories as well: among others, we have
			\[X + \zero \ism X \ism \zero + X, \qquad X \times \one \ism X \ism \one \times X.\]
			We define $\two \dfeq \one + \one$; in $\Set$, $\two$ is any two-element set.
			
			In partially ordered categories the product and coproduct are the greatest lower bound (infimum) and the least upper bound (supremum) respectively.
			
			The definition of (co)product can be generalized to that of a categorical (\df{co})\df{limit} which encompass many other constructions (such as making quotients); for more on the subject, see~\cite{lane1998categories}.
			
			We will need one more instance of a universal property in this paper, namely the \df{reflections} (and their duals, coreflections). These are defined and discussed in Section~\ref{Section: reflections}. We have a few more notions to define, though, which are not instances of the universal property.
			
			An arrow $m\colon X \to Y$ is called a \df{monomorphism} (or simply \df{mono}) when for all objects $T$ and all $f, g\colon T \to X$, if $m \circ f = m \circ g$, then $f = g$. In $\Set$ monomorphisms are exactly the injective maps. The dual notion is called an \df{epi}(\df{morphism}); in $\Set$ these are surjective maps. In general every isomorphism is both mono and epi, but the converse does not always hold (though it holds in $\Set$): for example, in a preorder category every morphism is both mono and epi, but in general not iso. In fact, this is pretty characteristic for preorder categories: if a category has a terminal (resp.~initial) object and all its morphisms are monos (resp.~epis), then it is a preorder category.
			
			Monos allow us to define subobjects in a general category (corresponding to subsets in $\Set$, subgroups in $\Grp$ \etc[]). We declare monos $m\colon T \to X$, $m'\colon S \to X$ with the same codomain $X$ to be equivalent when there exist (necessarily unique since $m$, $m'$ are monos) morphisms $f\colon S \to T$, $g\colon T \to S$ such that $m \circ f = m'$, $m' \circ g = m$, \ie the following diagram commutes.
			\[\xymatrix@-2em{
				T \ar[rrrrd]^m \ar@/^1ex/[dd]^g  &&&&\\
				&&&&  X  \\
				S \ar[rrrru]_{m'} \ar@/^1ex/[uu]^f  &&&&
			}\]
			In particular this implies that $f$ and $g$ are inverse, therefore isomorphisms.
			
			Intuitively, we identify those monos which have the same image in $X$. We thus define that \df{subobjects} of $X$ are equivalence classes of monos with codomain $X$. In $\Set$ (as well as in many other categories) subobjects have canonical representatives; for a subobject $[m\colon T \to X]$ this is the inclusion $\im(m) \hookrightarrow X$. That is, $[m]$ represents the subset $\im(m) \subseteq X$, and in this sense subsets of $X$ are identified with the categorical subobjects of $X$ in $\Set$.
			
			We can consider not just morphisms within a category, but morphisms between categories themselves, that is, the categories themselves form a category\footnote{More precisely, all small categories form a large category, all large categories form a superlarge category, and so on. The Russel's paradox still applies.}. The morphisms of categories are called \df{functors}. By definition, a functor $F\colon \kat{C} \to \kat{D}$ consists of a map between objects $F_0\colon \kat{C}_0 \to \kat{D}_0$ and a map between arrows $F_1\colon \kat{C}_1 \to \kat{D}_1$ which preserve all the categorical structure, \ie
			\[F_0\big(\dom(f)\big) = \dom_{F_1(f)}, \qquad F_0\big(\cod(f)\big) = \cod_{F_1(f)},\]
			\[F_1\big(\id[X]\big) = \id[F_0(X)], \qquad F_1\big(f \circ g\big) = F_1(f) \circ F_1(g).\]
			We usually drop the subscripts $0$, $1$ when it is clear whether the functor is applied on an object or a morphism, writing simply $F(X)$, $F(f)$.
			
			A functor $F\colon \kat{C} \to \kat{D}$ is called \df{full} when for all objects $X$, $Y$ of $\kat{C}$ and all morphisms $g\colon F(X) \to F(Y)$ in $\kat{D}$ there exists a morphism $f\colon X \to Y$ such that $F(f) = g$ (a form of local surjectivity on morphisms).
			
			A subcategory $\kat{C} \subseteq \kat{D}$ is called \df{full} when the inclusion $\kat{C} \hookrightarrow \kat{D}$ is a full functor. Note that to specify a full subcategory, it is enough to specify its objects.

	\section{Setting}\label{Section: setting}
	
		In this section we discuss the assumptions we make on the background mathematical setting. The purpose is to make them very general, so as to be able to interpret the results (the definition and properties of reals, and such blends of order, algebraic and topological structure more generally) in a wide variety of settings (classical set theory being only one of them). The assumptions are rather standard constructive ones, though we add a less usual one in the form of intrinsic topology. In the second part of the subsection we derive a few basic results which we'll use throughout the paper (without always explicitly calling back to them).
		
		We assume that we have a constructive set theory with (at least) first order logic (the category of \emph{classical} sets $\Set$ is also a special case of this).
		
		We assume that we have reasonable interpretations of the empty set, singletons, products, sums (disjoint unions), solution sets of equations, and quotients (formally, we are assuming to be working in a category with all finite limits and finite colimits).
		
		We assume that images of maps between sets are again sets.
		
		We assume that we have a set of natural numbers $\NN$, subject to Peano axioms (categorically, we are assuming to have a natural numbers object).
		
		We assume that for any set $X$ the collection $X^\NN$ of all sequences in $X$ is again a set (in categorical terms, $\NN$ is an exponentiable object). Consequently we also have the set $\finseq{X} = \coprod_{n \in \NN} X^n$ of all finite sequences in $X$, as these can be viewed as infinite sequences, in which we eventually start repeating an element not from $X$, \ie
		\[\finseq{X} \dfeq \st{a \in \finseq{(X + \one)}}{\xsome{l}{\NN}\all{i}{\NN}{a_i \in X \iff i < l}}.\]
		For any $a \in \finseq{X}$ the witnessing $l$ is clearly unique, and we denote it by $\lnth(a)$ (the \df{length} of the sequence $a$).
		
		For any set $X$ we want to talk about the collection of its subsets $\pst(X)$. However, we want to allow predicative settings when $\pst(X)$ is not in general a set, so we postulate that we also have classes. The only thing we assume for the category of classes is that it contains our category of sets as a full subcategory, it has finite limits, and that for every \emph{set} $X$ we have its its \df{powerclass} $\pst(X)$. (If the powerset axiom is assumed --- that is, $\pst(X)$ is a set for every set $X$ --- then we can take classes to be the same as sets.)
		
		So far, the assumptions have been quite ordinary, but now we will make a less usual one: every set has an \df{intrinsic topology}.
		
		The point is that the reals are useful because of their rich structure (in particular, the blend of various different structures), the topological one no less important than the order and the algebraic one. Even so, the reals are usually defined only through their order and algebra (\eg a "Dedekind complete ordered field") since that is enough, and the topology is tacked on later (in the case of the usual euclidean one, defined via the strict order relation $<$). We'll make the topology an intrinsic part of the definition of reals right from the start.
		
		We assume that for every set $X$ we have classes $\optp(X) \subseteq \pst(X)$ and $\cltp(X) \subseteq \pst(X)$ of \df{open} and \df{closed} subsets of $X$, respectively. We assume that every map $f\colon X \to Y$ between sets is \df{continuous}, that is, the preimage\footnote{Recall that the preimage of a morphism in a category can be formally defined via a pullback, and so exists in our setting.} of an open subset of $Y$ is an open subset of $X$, and a preimage of a closed subset of $Y$ is a closed subset of $X$.
		
		In classical topology $\optp(X)$ is assumed to be closed under finite intersections and arbitrary unions (note that this already includes the usually separately stated condition, that $\emptyset$ and the whole $X$ are open, since these are just the empty (nullary) union and intersection, respectively). We want to be more general than that to include examples where that is not the case, specifically \df{synthetic topology} and \df{Abstract Stone Duality}.\footnote{In both these cases, the indexing sets of the unions, under which the open subsets are closed, are called \df{overt}, which is a notion, dual to compactness. In classical topology this notion is vacant: every topological space is overt.}
		
		Thus we assume that for every set $X$ the topology $\optp(X)$ is closed under finite intersections (in particular, it contains the empty intersection $X \in \optp(X)$) and \emph{countable} unions (in particular, the empty union $\emptyset \in \optp(X)$). This makes $\optp(X)$ a so-called \df{$\sigma$-frame}.
		
		Classically closed subsets are defined simply as complements of the open ones, or equivalently, the subsets, of which the complements are open, or equivalently, the subsets which contain all their adherent points, etc. However, we want our results to hold constructively as well, where these various definitions are not equivalent. We could decide on one of them, but it is not necessary. We will simply percieve $\cltp(X)$ as one more additional structure, that might be connected in some way with $\optp(X)$, but is not necessarily so (giving us an additional degree of freedom).
		
		However, we must then separately specify the unions and intersections, under which closed subsets are preserved. We will want the order relation $\leq$ to be closed, which in a linear order is the negation of $>$ (which we want to be open). Due to the (constructively valid) de Morgan law $\lnot\bigvee_i a_i \iff \bigwedge_i \lnot{a_i}$, we make the assumption on $\cltp(X)$ that it is closed under countable intersections. However, the other de Morgan $\lnot\bigwedge_i a_i \iff \bigvee_i \lnot{a_i}$ constructively doesn't hold, though the following finite version of it does: $\lnot(a \land b) \iff \lnot\lnot(\lnot{a} \lor \lnot{b})$. Hence we assume that for any finite collection of closed subsets of $X$ the double complement of their union is closed also.
		
		In one of the examples of models that we'll consider --- synthetic topology --- the topology $\optp(X)$ is represented by an exponential $\opn^X$ (where $\opn$ is the set of ``open truth values'', \ie (isomorphic to) $\optp(\one)$). In categorical jargon, this exponentiation is a contravariant functor, adjoint to its opposite, and so maps colimits to limits. In more normal language, this means that disjoint unions and quotients have the expected topologies. We want this to hold in our general case also.
		
		Let $i\colon X \to X + Y$ and $j\colon Y \to X + Y$ be the canonical inclusions into the coproduct of $X$ and $Y$. We require that the topology of the disjoint union $X + Y$ is given in the standard way: its subset is open/closed if and only if its restrictions to the summands are, that is,
		\[\optp(X + Y) = \st{U \subseteq X + Y}{i^{-1}(U) \in \optp(X) \land j^{-1}(U) \in \optp(Y)},\]
		and the same for $\cltp(X + Y)$. Note that this implies $\optp(X + Y) \ism \optp(X) \times \optp(Y)$ (and likewise for closed subsets).
		
		Let $\equ$ be an equivalence relation on $X$ and $q\colon X \to X/_\sim$ the quotient map. We require that $X/_\sim$ has the \df{quotient topology}, that is,
		\[\optp(X/_\sim) = \st{U \subseteq X/_\sim}{q^{-1}(U) \in \optp(X)},\]
		and similarly for the closed subsets.
		
		\begin{remark}
			Half of our assumptions (the $\subseteq$ part) for the disjoint unions and quotients follows already from the fact that all maps are continuous. What we are additionally assuming is the reverse inclusion $\supseteq$, \ie that topologies on disjoint unions and quotients aren't \emph{weaker} than expected.
		\end{remark}
		
		This ends our assuptions on the setting. In the remainder of the section we observe some of their immediate consequences.
		
		First note that the usual way of defining the reals without referring to their topology is a special case of what we'll do: if you want to forget about topology, simply choose $\optp(X) = \cltp(X) = \pst(X)$. Everything being open and closed fulfills all the topological conditions in this paper, so it amounts to the same thing as ignoring the topological parts (naturally the reals $\RR$ will then also end up discrete). This is an important point: the assumption of intrinsic topology makes the setting \emph{more} general, not less.
		
		Recall that in constructivism a subset $A \subseteq X$ is called \df{decidable} (in $X$) when for every $x \in X$ the statement $x \in A \lor \lnot(x \in A)$ holds (of course, in classical mathematics every subset is decidable). Relations can be viewed as subsets of a product, and so this term applies for them as well. One can prove that relations $=$, $<$, $\leq$ are decidable on $\NN$, $\ZZ$ and $\QQ$ (though constructively usually not on $\RR$).
		
		Denote the elements of $\two = \one + \one$ by $\two = \{\top, \bot\}$ (``top'' and ``bottom'', or ``true'' and ``false''). Let $A$ be a decidable subset of $X$. Then we can define its \df{characteristic map} $\chi\colon X \to \two$ by
		\[\chi(x) \dfeq \begin{cases} \top & \text{if } x \in A,\\ \bot & \text{if } \lnot(x \in A). \end{cases}\]
		Since the empty subset and the whole set are always open and closed, it follows from the above definition of the disjoint union topology that each summand in a disjoint union is open and closed. In particular $\{\top\}$ is open and closed in $\two$, and then so is $A = \chi^{-1}(\{\top\})$ due to continuity. In summary, decidable implies open and closed.
		
		In particular this means that in classical set theory, where every subset is decidable, the only choice of topologies we can take under our assumptions is the discrete ones. In other words, in $\Set$ the reals we'll obtain will have no topological structure --- as is usual, the topology plays no part in the definition of the set $\RR$. This is to be expected, since sets can be seen as discrete topological spaces. However, if we interpret our definition of $\RR$ in a more ``topologically-minded'' category --- such as sheaves over (small) topological spaces --- we get the reals with the usual euclidean topology.
		
		Subsets, that we explicitly define, will typically be given by their defining property, and it is useful to extend the topological notions to predicates as well. Let $\Phi(x)$ be a predicate on a set $X$ (that is, the variable $x$ can be interpreted as an element of a set $X$). We say that $\Phi(x)$ is an \df{open} (resp.\ \df{closed}) \df{predicate} on $X$ when $\st{x \in X}{\Phi(x)}$ is an open (resp.\ closed) subset of $X$.
		
		This definition often allows us to easily recognize, that a set is open or closed (after we wade through the following very technical lemma); for example, the closure of topologies under countable unions translates to the fact that an open predicate, existentially quantified over a countable set, is again open, and so on.
	
		\begin{lemma}
			Let $f\colon E \to A$ and $p\colon E \to B$ be maps; denote the latter's fibers by $E_b \dfeq p^{-1}(b)$. Note that we have a map $[p, \id[B]]\colon E + B \to B$.
			\begin{enumerate}
				\item
					Suppose all fibers of $p$ are countable in the following sense: there exists a map $s\colon \NN \times B \to E + B$ which preserves fibers (\ie $[p, \id[B]] \circ s = \pi$ where $\pi\colon \NN \times B \to B$ is the second projection)\footnote{In other words, $p$ is countable in the slice category over $B$.} and for all $b \in B$ we have $E_b \subseteq s(\NN \times \{b\})$.
					
					Then if $\Phi(a)$ is an open predicate on $A$,
					\[\xsome{x}{E_b}{\Phi(f(x))}\]
					is an open predicate on $B$. If $\Phi(a)$ is a closed predicate on $A$, then
					\[\xall{x}{E_b}{\Phi(f(x))}\]
					is a closed predicate on $B$.
				\item
					Suppose all fibers of $p$ are finite in the following sense: there exists $m \in \NN$ and a map $s\colon \NN_{< m} \times B \to E + B$ which preserves fibers (\ie $[p, \id[B]] \circ s = \pi$ where $\pi\colon \NN \times B \to B$ is the second projection) and for all $b \in B$ we have $E_b \subseteq s(\NN_{< m} \times \{b\})$.
					
					Then if $\Phi(a)$ is an open predicate on $A$,
					\[\xall{x}{E_b}{\Phi(f(x))}\]
					is an open predicate on $B$. If $\Phi(a)$ is a closed predicate on $A$, then
					\[\lnot\lnot\xsome{x}{E_b}{\Phi(f(x))}\]
					is a closed predicate on $B$.\footnote{Naturally in all of the examples above, $\Phi(f(x))$ can contain $b$ as well, but it is sufficient to write that $f$ is dependent only on $x$ since $b = p(x)$.}
			\end{enumerate}
		\end{lemma}
		\begin{proof}
			\begin{enumerate}
				\item
					Suppose first that $\Phi(a)$ is open, and let $U \dfeq \st{a \in A}{\Phi(a)}$. By assumption $U$ is open in $E$, so $f^{-1}(U)$ is open in $E$, and by the definition of the coproduct topology $f^{-1}(U)$ is open in $E + B$. For any $n \in \NN$ denote $V_n \dfeq s(n,\insarg)^{-1}\big(f^{-1}(U)\big)$; this is an open subset of $B$. We have
					\[\xsome{x}{E_b}{\Phi(f(x))} \iff \xsome{x}{E_b}{f(x) \in U} \iff\]
					\[\iff \xsome{n}{\NN}{f(s(n,b)) \in U} \iff \xsome{n}{\NN}{b \in V_n},\]
					so $\st{b \in B}{\xsome{x}{E_b}{\Phi(f(x))}} = \bigcup_{n \in \NN} V_n$ which is an open subset of $B$.
					
					Suppose now that $\Phi(a)$ is closed, and define $F \dfeq \st{a \in A}{\Phi(a)}$ and for $n \in \NN$ let $G_n \dfeq s(n,\insarg)^{-1}\big(f^{-1}(U) + B\big)$. $F$ is closed in $A$, so $f^{-1}(F)$ is closed in $E$, so $f^{-1}(F) + B$ is closed in $E + B$, and then finally $G_n$ is closed in $B$. Similarly as above we have
					\[\xall{x}{E_b}{\Phi(f(x))} \iff \xall{x}{E_b}{f(x) \in F} \iff\]
					\[\iff \xall{n}{\NN}{s(n,b) \in f^{-1}(G) + B} \iff \xall{n}{\NN}{b \in G_n},\]
					so $\st{b \in B}{\xall{x}{E_b}{\Phi(f(x))}} = \bigcap_{n \in \NN} G_n$ is closed in $B$.
				\item
					Similarly as above. If $U \dfeq \st{a \in A}{\Phi(a)}$ is open, then so are all $V_n \dfeq s(n,\insarg)^{-1}\big(f^{-1}(U) + B\big)$, and $\st{b \in B}{\xall{x}{E_b}{\Phi(f(x))}} = \bigcap_{n \in \NN_{< m}} V_n$ is open in $B$. If $F \dfeq \st{a \in A}{\Phi(a)}$ is closed, then all $G_n \dfeq s(n,\insarg)^{-1}\big(f^{-1}(U)\big)$ are, and $\st{b \in B}{\lnot\lnot\xsome{x}{E_b}{\Phi(f(x))}} = \big(\bigcup_{n \in \NN_{< m}} G_n\big)^{CC}$ is closed in $B$.
			\end{enumerate}
		\end{proof}
		
		Naturally the general version implies the one where all the fibers are the same.
		\begin{corollary}
			Let $f\colon X \times B \to A$ be a map and $\Phi(a)$ a predicate on $A$.\
			\begin{itemize}
				\item
					If $X$ is countable and $\Phi(a)$ open, then $\xsome{x}{X}{\Phi(f(x, b))}$ is an open predicate on $B$.
				\item
					If $X$ is countable and $\Phi(a)$ closed, then $\xall{x}{X}{\Phi(f(x, b))}$ is a closed predicate on $B$.
				\item
					If $X$ is finite and $\Phi(a)$ open, then $\xall{x}{X}{\Phi(f(x, b))}$ is an open predicate on $B$.
				\item
					If $X$ is finite and $\Phi(a)$ closed, then $\lnot\lnot\xsome{x}{X}{\Phi(f(x, b))}$ is a closed predicate on $B$.
			\end{itemize}
		\end{corollary}
		\begin{proof}
			Use the preceding lemma for $E = X \times B$ and $p$ the projection on the second factor.
		\end{proof}

	\section{Prestreaks and streaks}\label{Section: streaks}
	
		In this section we define and study the notion of a \df{streak} which is a convenient way to capture the combination of order, algebraic and topological structure of the reals (allowing us to define $\RR$ in two words --- see Definition~\ref{Definition: reals}), but is general enough that we can use it to classify other number sets $\NN$, $\ZZ$, $\QQ$ as well.
		
		More to the point, it is also modular; it allows us to add new pieces of structure in such a way, that the new structure forms a reflective subcategory of the old one, whence it follows directly from the definition of the reals, that $\RR$ must also possess this structure, and we get direct formulae for it.

		Informally, a streak consists of
		\begin{itemize}
			\item
				the order structure, with both the strict $<$ and nonstrict $\leq$ order relations,
			\item
				the algebraic structure, being a monoid for addition $+$, and its positive elements being a monoid for multiplication $\cdot$ (in other words, we take only the part of the algebraic structure of $\RR$ which preserves the order structure $<$ and $\leq$, \ie addition and multiplication with positive elements, but not subtraction, multiplication with nonpositive elements, or division),
			\item
				the topological structure which must also be in agreement with the order structure ($<$ must be open and $\leq$ closed).
		\end{itemize}
		
		It is convenient to also have a more general notion of a \df{prestreak}. One reason is that the addition of a new structure typically entails constructing a prestreak first, then quotienting it out to get the desired streak. Another reason is that the \df{smooth reals} in synthetic differential geometry (see Subsection~\ref{Subsection: smooth_reals}) are at best a prestreak rather than a streak.
		
		Most of this section is spent on technicalities, proving (in painful detail) that those properties hold which we expected to hold anyway. It is there, so that we have black on white, that our claims are valid, but otherwise this section can be just skimmed through by the reader. What's important, is to keep in mind the following: streaks are (up to isomorphism) subsets of reals containing $0$ and $1$, closed under addition, and positive elements are closed under multiplication (while prestreaks are a generalization, where we don't require antisymmetry of $\leq$ and the archimedean property).
		
		\subsection{Order relations}
		
			We start with the formal treatment of the order relations. We want $\RR$ (and streaks in general) to be linearly ordered, but as it is well known, the condition $\all{x,y}{\RR}{x \leq y \lor y \leq x}$ is constructively too strong (it implies \llpo[]). So we follow the standard constructive way of defining the linear order.
			
			\begin{definition}
				A set $X$ is \df{strictly ordered} by a binary relation $<$ when the following holds for all $a, b, x \in X$.
				\begin{itemize}
					\item
						$\lnot(a < b \land b < a)$ \quad (\df{asymmetry})
					\item
						$a < b \implies a < x \lor x < b$ \quad (\df{cotransitivity})
				\end{itemize}
			\end{definition}
			Observe that a strict order is also irreflexive ($\lnot(a < a)$ for all $a \in X$) --- take $a = b$ in asymmetry condition --- and transitive ($a < b \land b < c \implies a < c$) --- if $a < b$, then by cotransitivity $a < c$ or $c < b$, and the latter, together with $b < c$, contradicts asymmetry.
			
			In a strictly ordered set $(X, <)$ we define the nonstrict order relation by $a \leq b \dfeq \lnot(b < a)$ for $a, b \in X$. This is a preorder on $X$; reflexivity of $\leq$ is irreflexivity of $<$, and transitivity of $\leq$ is the contrapositive of cotransitivity of $<$. Note that asymmetry of $<$ can be restated as $a < b \implies a \leq b$. Also, transitivity of $<$ can be strenghtened (with the same argument) to $a < b \land b \leq c \implies a < c$, as well as $a \leq b \land b < c \implies a < c$.
			
			In classical mathematics we can use the deMorgan law on asymmetry of $<$, obtaining the standard definition of a linear order $a \leq b \lor b \leq a$. Constructively, we make do with just the properties, defined above.
			
			We also define a relation $\apart$ by $a \apart b \dfeq a < b \lor b < a$. This is an apartness relation~\cite{Troelstra_AS_Dalen_D_1988:_constructivism_in_mathematics_volume_2} on $X$: it is irreflexive ($\lnot(a \apart a)$), symmetric ($a \apart b \implies b \apart a$) and cotransitive ($a \apart b \implies a \apart x \lor x \apart b$).
			
			A strict order $(X, <)$ is called \df{tight} when the apartness relation satisfies the \df{tightness} condition
			\[\lnot(a \apart b) \implies a = b\]
			for all $a, b \in X$, or equivalently, when $\leq$ is antisymmetric:
			\[a \leq b \land b \leq a \implies a = b,\]
			thus a partial order.
			
			There is a standard way, how to turn a preorder into a partial order (or an apartness relation into a tight one). Applying this in the case of a strict order $(X, <)$, we define the relation $\nap$ on $X$ by
			\[x \nap y \dfeq x \leq y \land y \leq x, \qquad \text{or equivalently} \qquad x \nap y \dfeq \lnot(x \apart y).\]
			It is easily seen that this is an equivalence relation on $X$ and that $<$ induces a well-defined tight strict order on $X/_\nap$ by $[x] < [y] \dfeq x < y$, therefore a partial order $[x] \leq [y] \iff x \leq y$ and a tight apartness relation $[x] \apart [y] \iff x \apart y$. The relation $\nap$ becomes, of course, the equality on $X/_\nap$.
			
			Recall that for any partial order $(X, \leq)$, a supremum $\sup A$ of a subset $A \subseteq X$ is an element $s \in X$ with the property
			\[\all{x}{X}{s \leq x \iff \xall{a}{A}{a \leq x}}.\]
			The left-to-right implication states $s$ is an upper bound for $A$ (it is equivalent to $\xall{a}{A}{a \leq s}$), the other implication tells it is the least such. A supremum of a set need not exist, but if it does, antisymmetry of $\leq$ implies it is unique; moreover, we have $a \leq b \iff \sup\{a, b\} = b$. We define infimum $\inf A$ analogously.
			
			In the case of a (tight) strict order we have a strengthening of the notion of supremum and infimum. We define $s \in X$ to be the \df{strict supremum} of $A \subseteq X$ when
			\[\all{x}{X}{x < s \iff \xsome{a}{A}{x < a}}\]
			(and analogously the \df{strict infimum}). The left-to-right implication is again equivalent to $s$ being an upper bound, but the other one is a genuine strengthening of the previous condition, so a strict supremum is also a supremum, but not necessarily vice versa, unless we have classical logic. To see this, let $X = \{0, 1\}$ and $A = \{0\} \cup \st{1}{p}$ where $\lnot\lnot p$ holds. We claim $\sup A = 1$. Clearly, $1$ is an upper bound of $A$. Now let $x \in X$ be an upper bound for $A$. It cannot be $x = 0$ since that would imply $\lnot(1 \in A)$, \ie $\lnot\lnot\lnot p = \lnot p = \bot$. So $x = 1$. However, if $1$ is also the strict supremum of $A$, then $0 < 1$ implies $\xsome{a}{A}{0 < a}$, meaning $p$. We obtained $\lnot\lnot p \implies p$ for an arbitrary truth value $p$.
			
			We may calculate (strict) suprema and infima per parts.
			\begin{lemma}
				Let $\{A_i\}_{i \in I}$ be a family of subsets $A_i \subseteq X$, and $<$ a tight strict order on $X$. Then
				\[\sup \bigcup_{i \in I} A_i = \sup \st{\sup A_i}{i \in I}\]
				if all above suprema exist. Analogous formulae hold for infima and their strict versions.
			\end{lemma}
			\begin{proof}
				We prove the formula only in the case of strict suprema. Denote $s \dfeq \sup\st{\sup A_i}{i \in I}$, and assume the suprema in the claim are strict. Take any $a \in \bigcup_{i \in I} A_i$. Then there is an $i \in I$, such that $a \in A_i$, so $a \leq \sup A_i \leq s$. Now take any $x \in X$, $x < s$. That means there exists $i \in I$ for which $x < \sup A_i$, hence there is $a \in A_i$ for which $x < a$. This proves the claim.
			\end{proof}
			
			In particular, this means $\sup\{a, b, c\} = \sup\{\sup\{a, b\}, c\}$ and in general for $n \in \NN_{\geq 2}$,
			\[\sup\{a_0, a_1, \ldots, a_{n-1}\} = \sup\{\sup\{\ldots \sup\{a_0, a_1\}, a_2\}, \ldots, a_{n-1}\}.\]
			Moreover, $\sup\{a\} = a$, so typically when we check something for finite suprema (and infima and their strict versions), we need to verify the condition only for the supremum of the empty set (\ie the least element) and binary suprema.
			
			Actually, in the case of finite suprema and infima, there is no difference between the usual and the strict version.
			\begin{proposition}\label{Proposition: finite_strict_suprema_and_infima}
				Let $<$ be a tight strict order on $X$, $n \in \NN$ and $a_0, \ldots, a_{n-1} \in X$. Suppose $s \dfeq \sup\{a_0, \ldots, a_{n-1}\}$ exists.
				\begin{enumerate}
					\item For $n \geq 1$, it is not true that for all $i \in \NN_{<n}$, $a_i < s$.
					\item The supremum $s$ is also strict.
				\end{enumerate}
				Similarly for infima.
			\end{proposition}
			\begin{proof}
				The claim of the proposition is obvious for $n = 0, 1$. It is sufficient to prove it for $n = 2$. Assume $a_0 < s$, $a_1 < s$. Suppose $a_1 < a_0$. Then also $a_1 \leq a_0$, so $s = a_0$, a contradiction to $a_0 < s$. So $a_0 \leq a_1$, but then $s = a_1$, a contradiction to $a_1 < s$.
				
				As for the second part, take any $x \in X$, $x < s$. By cotransitivity, $x < a_0 \lor a_0 < s$ and $x < a_1 \lor a_1 < s$. It cannot be both $a_0 < s$ and $a_1 < s$, so $x < a_0 \lor x < a_1$ which proves $s$ is strict.
			\end{proof}
			
			As usual, we call $(X, <)$ a \df{lattice} when it has binary (hence inhabited finite) suprema and infima (which are then automatically strict).
		
		\subsection{Prestreaks}
		
			In this subsection we define prestreaks and their morphisms, examine their properties, observe their connection to natural numbers and study the archimedean property.
			
			Prestreaks provide a basic and very general mix of order, algebraic and topologic structure.
			\begin{definition}\label{Definition: prestreaks}
				A structure $(X, <, +, 0, \cdot, 1)$ is called a \df{prestreak} when
				\begin{itemize}
					\item
						$(X, <)$ is a strict order,
					\item
						$(X, +, 0)$ is a commutative monoid,
					\item
						$(X_{> 0}, \cdot, 1)$ is a commutative monoid,
					\item
						\df{multiplication} $\cdot$ distributed over \df{addition} $+$, \ie for all $a, b, c \in X_{> 0}$ we have $(a + b) \cdot c = a \cdot c + b \cdot c$,
					\item
						adding an element preserves and reflects $<$, \ie for all $a, b, x \in X$ we have $a + x < b + x \iff a < b$,
					\item
						multiplying with a (\df{positive}) element preserves and reflects $<$, \ie for all $a, b, x \in X_{> 0}$ we have $a \cdot x < b \cdot x \iff a < b$,
					\item
						the strict order $<$ is an \df{open relation}, \ie $\st{(a, b) \in X \times X}{a < b} \in \optp(X \times X)$,
					\item
						the induced preorder $\leq$ is a \df{closed relation}, \ie $\st{(a, b) \in X \times X}{a \leq b} \in \cltp(X \times X)$.
				\end{itemize}
			\end{definition}
			
			\begin{proposition}
				The following holds in a prestreak $(X, <, +, 0, \cdot, 1)$ for all $a, b, c, d, x \in X$.
				\[a + x \leq b + x \iff a \leq b  \qquad  a < b \land c < d \implies a + c < b + d\]
				\[a < b \land c \leq d \implies a + c < b + d  \qquad  a \leq b \land c \leq d \implies a + c \leq b + d\]
				\[a + b > 0 \implies a > 0 \lor b > 0  \qquad  a + b < 0 \implies a < 0 \lor b < 0\]
				\[a > 0 \land b > 0 \implies a + b > 0  \qquad  a < 0 \land b < 0 \implies a + b < 0\]
				\[a > 0 \land b \geq 0 \implies a + b > 0  \qquad  a < 0 \land b \leq 0 \implies a + b < 0\]
				\[a \geq 0 \land b \geq 0 \implies a + b \geq 0  \qquad  a \leq 0 \land b \leq 0 \implies a + b \leq 0\]
				The following holds for all $a, b, c, d, x \in X_{> 0}$.
				\[a \cdot x \leq b \cdot x \iff a \leq b  \qquad  a < b \land c < d \implies a \cdot c < b \cdot d\]
				\[a < b \land c \leq d \implies a \cdot c < b \cdot d  \qquad  a \leq b \land c \leq d \implies a \cdot c \leq b \cdot d\]
				\[a \cdot b > 1 \implies a > 1 \lor b > 1  \qquad  a \cdot b < 1 \implies a < 1 \lor b < 1\]
				\[a > 1 \land b > 1 \implies a \cdot b > 1  \qquad  a < 1 \land b < 1 \implies a \cdot b < 1\]
				\[a > 1 \land b \geq 1 \implies a \cdot b > 1  \qquad  a < 1 \land b \leq 1 \implies a \cdot b < 1\]
				\[a \geq 1 \land b \geq 1 \implies a \cdot b \geq 1  \qquad  a \leq 1 \land b \leq 1 \implies a \cdot b \leq 1\]
			\end{proposition}
			\begin{proof}
				Follows easily from the definitions.
			\end{proof}
			
			The condition $a > 0 \land b > 0 \implies a + b > 0$ means that there was no problem when we required distributivity $(a + b) \cdot c = a \cdot c + b \cdot c$ in Definition~\ref{Definition: prestreaks}: if $a, b, c > 0$, then all three products $a \cdot c$, $b \cdot c$, $(a + b) \cdot c$ are well defined.
			
			In a prestreak we have $0 < 1$ by definition. In particular $0 \apart 1$, so a prestreak has at least two elements (in fact, it has infinitely many of them --- all the natural numbers, as we shall soon see).
			
			We want prestreaks to form a category, so we need to define the notion of their morphisms. Naturally, for the definition we take maps which preserve all the structure.
			\begin{definition}
				A map between prestreaks $f\colon X \to Y$ is a (\df{prestreak}) \df{morphism} when the following holds for all $a, b \in X$.
				\[a < b \implies f(a) < f(b)\]
				\[f(a + b) = f(a) + f(b)  \qquad  f(0) = 0\]
				\[f(a \cdot b) = f(a) \cdot f(b)  \qquad  f(1) = 1\]
			\end{definition}
			Of course, we require the preservation of multiplication only when it is defined, \ie for $a, b > 0$; since in that case also $f(a), f(b) > 0$, their product is well defined as well.
			
			We don't explicitly ask for the preservation of the topologic structure since this is automatic: recall that we're assuming that all maps are continuous.
			
			We denote the category of prestreaks and their morphisms by $\Pstr$.
			
			The simplest example of a prestreak are natural numbers.
			\begin{proposition}\label{Proposition: N_initial_prestreak}
				The set of natural numbers $\NN$, together with the usual order and operations, is a prestreak. Moreover, for any prestreak $X$ there exists a unique morphism $\ini[X]\colon \NN \to X$; it is given inductively by $\ini[X](0) = 0$ and $\ini[X](n + 1) = \ini[X](n) + 1$.
			\end{proposition}
			\begin{proof}
				That natural numbers satisfy the prestreak conditions is standard; we note only that $<$ and $\leq$ are for natural numbers decidable even constructively, therefore open and closed.
				
				The map $\ini[X]$ is indeed a prestreak morphism; this follows from the inductive definitions of $+$ and $\cdot$ on natural numbers
				\[0 + n \dfeq n  \qquad  s(m) + n \dfeq s(m + n)\]
				\[0 \cdot n \dfeq 0  \qquad  s(m) \cdot n \dfeq m \cdot n + n\]
				(where $s\colon \NN \to \NN$ is the succesor map) and the fact $s(0) = 1$.
				
				Finally, any morphism $f\colon \NN \to X$ must satisfy $f(0) = 0$ and $f(n + 1) = f(n) + 1$, so $\ini[X]$ is indeed a unique one.
			\end{proof}
			
			In categorical language, this proposition states that $\NN$ is an initial prestreak.
			
			Observe that $\ini[X]$ is injective --- if $\ini[X](m) = \ini[X](n)$, then it cannot be $m < n$ (since $<$ is preserved by morphisms) and likewise not $m > n$, so $m = n$. In view of this we will regard $\ini[X]$ as an inclusion of natural numbers into a prestreak $X$, and in this sense $\NN \subseteq X$ for every prestreak $X$. In other words, natural numbers are the smallest prestreak. In particular, every prestreak contains infinitely many elements.
			
			We can always extend multiplication in an arbitrary prestreak $X$ from $X_{> 0} \times X_{> 0} \to X_{> 0}$ to $\big(X_{> 0} \cup \{0\}\big) \times \big(X_{> 0} \cup \{0\}\big) \to \big(X_{> 0} \cup \{0\}\big)$ by defining $0 \cdot a = a \cdot 0 \dfeq 0$ for all $a \in X_{> 0} \cup \{0\}$, and it still remains distributive over addition. When necessary, we will interpret multiplication in this way.
			
			Note that $\big(X_{> 0} \cup \{0\}\big)$ need not be equal to $X_{\geq 0}$ (though of course it is its subset). One reason is that $\leq$ need not be antisymmetric (consider for example the prestreak $\NN + \{0'\}$ in which we define $0 + 0' = 0' + 0 = 0'$, but otherwise $0'$ behaves exactly as $0$). However, even if $\leq$ is a partial order, the equality might fail, as in typical constructive models we have $\RR_{\geq 0} \neq \RR_{> 0} \cup \{0\}$.
			
			That said, we do have the equality $\NN_{> 0} \cup \{0\} = \NN_{\geq 0} = \NN$, so in any prestreak $X$ we can multiply with all natural numbers (via $\NN \subseteq X$). In fact, we can extend the multiplication with natural numbers from $X_{> 0} \cup \{0\}$ to the whole of $X$ via the inductive definition $0 \cdot a \dfeq 0$ and $(n+1) \cdot a \dfeq n \cdot a + a$.
			
			\begin{proposition}\label{Proposition: multiplication_with_natural_numbers}
				The following holds in a prestreak $X$ for all $a, b \in X$ and $n \in \NN$:
				\begin{enumerate}
					\item
						$n \cdot (a + b) = n \cdot a + n \cdot b$,\footnote{We know that distributivity for positive numbers holds by the definition of a prestreak; the point of this statement is that it holds for general $a, b \in X$.}
					\item
						$a < b \land n > 0 \iff n \cdot a < n \cdot b$.
				\end{enumerate}
			\end{proposition}
			\begin{proof}
				Both statements are proved by induction on $n \in \NN$.
				\begin{enumerate}
					\item
						If $n = 0$, then we have $0 = 0 + 0$. If the equation holds for $n$, then
						\[(n + 1) \cdot (a + b) = n \cdot (a + b) + a + b =\]
						\[= n \cdot a + n \cdot b + a + b = (n + 1) \cdot a + (n + 1) \cdot b,\]
						so it holds for $n+1$ as well.
					\item
						Both sides of the stated equivalence are false for $n = 0$. Suppose the equivalence holds for $n$; we want to prove it for $n+1$. Of course $n + 1 > 0$, so we're proving $a < b \iff (n+1) \cdot a < (n+1) \cdot b$. This clearly holds for $n = 0$, so assume hereafter $n > 0$.
						
						For the `only if' direction, if $a < b$, then by induction hypothesis $n \cdot a < n \cdot b$, and then $(n+1) \cdot a = n \cdot a + a < n \cdot b + b = (n+1) \cdot b$. Conversely, assume $(n+1) \cdot a < (n+1) \cdot b$. By cotransitivity $(n+1) \cdot a < n \cdot a + b$ or $n \cdot a + b < (n+1) \cdot b$. Rewriting the first case, we obtain $n \cdot a + a < n \cdot a + b$, so $a < b$. In the second case we have $n \cdot a + b < n \cdot b + b$; cancel $b$, then use the induction hypothesis.
				\end{enumerate}
			\end{proof}
			
			We have seen that natural numbers $\NN$ can be embedded into an arbitrary prestreak. Clearly this doesn't hold for other number sets --- we can't embed, say, the integers $\ZZ$ or the rationals $\QQ$ into the prestreak $\NN$ via a strict-order-preserving map. However, with the help of multiplication with natural numbers we can still define comparison between rational numbers and elements of a prestreak.
			
			Let $X$ be a prestreak. Any $q \in \QQ$ can be written as $q = \tfrac{a-b}{c}$ where $a, b \in \NN$, $c \in \NN_{> 0}$. For $x \in X$ we define
			\[x < q \dfeq c \cdot x + b < a \qquad \text{and analogously} \qquad q < x \dfeq a < c \cdot x + b.\]
			This is well-defined --- independent of the choices for $a, b, c$. Let $q = \tfrac{a-b}{c} = \tfrac{a'-b'}{c'}$ where $a, b, a', b' \in \NN$, $c, c' \in \NN_{> 0}$. That means $c' a + c b' = c a' + c' b$, so
			\[c \cdot x + b < a \iff c' \cdot (c \cdot x + b) < c' \cdot a \iff c' c \cdot x + c' b < c' a \iff\]
			\[\iff c' c \cdot x + c' b + c b' < c' a + c b' \iff c' c \cdot x + c' b + c b' < c' a + c b' \iff\]
			\[\iff c c' \cdot x + c b' + c' b < c a' + c' b \iff c c' \cdot x + c b' + c' b < c a' + c' b \iff\]
			\[\iff c c' \cdot x + c b' < c a' \iff c' \cdot x + b' < a'\]
			(and similarly for $q < x$). Notice also that for $x \in X$ and $n \in \NN$ the statements $x < n$ and $n < x$ are unambiguous, regardless whether we view $\NN$ as a subset of $X$, or of $\QQ$. Moreover, if $X = \QQ$, then this order matches the usual one on the rationals, so again there is no ambiguity when using the symbol $<$.
			
			Unsurprisingly, this relation $<$ has similar properties as the strict order.
			\begin{proposition}\label{Proposition: prestreak_comparison_with_rationals}
				Let $X$ be a streak. The following holds for all $x, y \in X$, $q, r \in \QQ$:
				\[\lnot(x < q \land q < x),\]
				\[x < q \implies x < y \lor y < q,  \qquad  q < x \implies q < y \lor y < x,\]
				\[x < q \implies x < r \lor r < q,  \qquad  q < x \implies q < r \lor r < x,\]
				\[q < r \land r < x \implies q < x,  \qquad  q \leq r \land r < x \implies q < x,\]
				\[x < q \land q < r \implies x < r,  \qquad  x < q \land q \leq r \implies x < r,\]
				\[q < x \land x < r \implies q < r,  \qquad  x < q \land q < y \implies x < y,\]
				\[q < x \land r < y \implies q + r < x + y,  \qquad  x < q \land y < r \implies x + y < q + r,\]
				and if additionally $q, r \geq 0$, $x, y > 0$,
				\[q < x \land r < y \implies q \cdot r < x \cdot y,  \qquad  x < q \land y < r \implies x \cdot y < q \cdot r.\]
				Moreover, any prestreak morphism $f\colon X \to Y$ preserves the comparison with rationals:
				\[x < q \implies f(x) < q, \qquad q < x \implies q < f(x).\]
			\end{proposition}
			\begin{proof}
				Follows easily from the definitions.
			\end{proof}
			
		\subsection{Archimedean prestreaks}
		
			General prestreaks are in a certain sense too big to properly allow us to define real numbers. We'll cut them down in two ways (to obtain the notion of a streak); we consider the first way --- the \df{archimedean property} --- in this subsection.
			
			The archimedean property for reals, and more generally for an ordered field $X$, states that ``there are no infinite elements'', or more precisely, every element is bounded by some natural number: $\xall{a}{X}\xsome{n}{\NN}{a < n}$. However, in more general structures, this is not enough. In an ordered (unital) ring (where we don't have division) we have to generalize this to $\xall{a}{X}\xall{b}{X_{> 0}}\xsome{n}{\NN}{a < n \cdot b}$. In a prestreak we don't even have subtraction, so this needs to be further generalized.
			\begin{definition}
				A prestreak $X$ is \df{archimedean} when it satisfies
				\[\all[2]{a,b,c,d}{X}{b < d \implies \xsome{n}{\NN}{a + n \cdot b < c + n \cdot d}}.\]
				We denote the category of archimedean prestreaks and prestreak morphisms by $\Apstr$.
			\end{definition}
			
			We can always increase $n$, witnessing the archimedean property, if necessary.
			\begin{lemma}
				Suppose $X$ is a prestreak and $a, b, c, d \in X$, $n \in \NN$ satisfy $b < d$ and $a + n \cdot b < c + n \cdot d$. Then also $a + m \cdot b < c + m \cdot d$ for all $m \in \NN_{\geq n}$.
			\end{lemma}
			\begin{proof}
				Obvious induction on $m$.
			\end{proof}
			
			The actual point of the archimedean property is that it means that ``rational numbers are dense in $X$''. We now make precise what this means for prestreaks. We start by observing that the archimedean property lets us write a prestreak as a union of rational intervals (with regard to the comparison between prestreak elements and rationals, given at the end of the previous subsection).
					
			\begin{lemma}\label{Lemma: archimedean_prestreak_is_a_union_of_rational_intervals}
				Let $X$ be an archimedean prestreak. Then for any $b \in \NN_{> 0}$:
				\begin{enumerate}
					\item $\intoo[X]{0}{n} = \bigcup_{a \in \intoo[\NN]{0}{n b}} \intoo[X]{\tfrac{a-1}{b}}{\tfrac{a+1}{b}}$ for all $n \in \NN$ such that $n \cdot b \geq 2$,
					\item $X_{> 0} = \bigcup_{a \in \NN_{> 0}} \intoo[X]{\tfrac{a-1}{b}}{\tfrac{a+1}{b}}$,
					\item $X = \bigcup_{a \in \ZZ} \intoo[X]{\tfrac{a-1}{b}}{\tfrac{a+1}{b}}$.
				\end{enumerate}
			\end{lemma}
			\begin{proof}
				\begin{enumerate}
					\item
						For $x \in X$, $0 < x < n$, consider the disjunctions $\tfrac{m}{b} < x \lor x < \tfrac{m+1}{b}$ for $m \in \NN_{< n b}$. Write a finite sequence: in each of the disjunctions make a choice of a true disjunct, and write $0$ if the first one is chosen, and $1$ if the second one is. If the sequence contains no $1$s, let $a = n b - 1$; if the sequence contains no $0$s, let $a = 1$ (we need the condition $n \cdot b \geq 2$ in these two cases to ensure $a \in \intoo[\NN]{0}{n b}$). Otherwise, let $a$ be that $m$ for which the first $1$ appears. Then $x \in \intoo[X]{\tfrac{a-1}{b}}{\tfrac{a+1}{b}}$.
					\item
						By the archimedean property (take $a = x$, $b = 0$, $c = 0$, $d = 1$) we have for any $x \in X_{> 0}$ some $n \in \NN$ (without loss of generality assume $n \geq 2$, so that $n \cdot b \geq 2$) such that $x < n$. Now use the previous item.
					\item
						Take any $x \in X$. Use the archimedean property to obtain $n \in \NN$ such that $x + n > 0$ (take $a = 0$, $b = 0$, $c = x$, $d = 1$). By the previous item there exists $a \in \NN_{> 0}$ such that $x + n \in \intoo[X]{\tfrac{a-1}{b}}{\tfrac{a+1}{b}}$. It follows that $x \in \intoo[X]{\tfrac{a - n \cdot b - 1}{b}}{\tfrac{a - n \cdot b + 1}{b}}$.
				\end{enumerate}
			\end{proof}
			
			Rationals are dense in an archimedean prestreak in the following sense(s).
			\begin{lemma}\label{Lemma: density_of_rationals_in_an_archimedean_prestreak}
				Let $X$ be an archimedean prestreak.
				\begin{enumerate}
					\item
						For an element $x \in X$ and a rational $q \in \QQ$ we have
						\[x < q \implies \xsome{r}{\QQ}{x < r < q} \qquad \text{and} \qquad q < x \implies \xsome{r}{\QQ}{q < r < x}.\]
					\item 
						For $x, y \in X$ we have
						\[x < y \iff \some{q}{\QQ}{x < q < y} \iff \some{q, r}{\QQ}{x < q < r < y}.\]
				\end{enumerate}
			\end{lemma}
			\begin{proof}
				\begin{enumerate}
					\item
						Suppose $x < q$ and write $q = \tfrac{a-b}{c}$ where $a, b \in \NN$, $c \in \NN_{> 0}$. We then have $c \cdot x + b < a$, so by the archimedean property there exists $n \in \NN$ (necessarily greater than $0$) such that $c + n \cdot (c \cdot x + b) < n \cdot a$. Clearly we can then take $r = q - \frac{1}{n} = \frac{n a - (n b + c)}{n c}$.
						
						The case $q < x$ is analogous.
					\item
						The implications $\Leftarrow$ are clear, so it remains to prove
						\[x < y \implies \some{q, r}{\QQ}{x < q < r < y}.\]
						Assume $x < y$; then by the archimedean property there is $n \in \NN$ (necessarily $\geq 1$) with $1 + n \cdot x < n \cdot y$. By the previous lemma we have $a \in \ZZ$ such that $\tfrac{a-1}{3n} < x < \tfrac{a+1}{3n}$. We claim $x < \tfrac{a+1}{3n} < \tfrac{a+2}{3n} < y$. Only the last inequality needs proof. Write $a-1 = b - c$ where $b, c \in \NN$; then $a+2 = (b+3) - c$. We calculate
						\[b + 3 < c + 3n \cdot x + 3 = c + 3(n \cdot x + 1) < c + 3n \cdot y.\]
				\end{enumerate}
			\end{proof}
			
			We can go in the other direction as well and show that suitable ``density of rationals'' implies the archimedean property, obtaining its characterization.
			\begin{lemma}\label{Lemma: characterization_of_archimedean_prestreaks}
				Let $X$ be a prestreak such that for all $x, y \in X$ with $x < y$ there exist rationals $q, r, s, t \in \QQ$ with $q < x < r < s < y < t$. Then $X$ is archimedean.
			\end{lemma}
			\begin{proof}
				Take any $a, b, c, d \in X$ with $b < d$. By assumption there exist $r, s \in \QQ$ with $b < r < s < d$. Since $a < a + 1$ and $c < c + 1$, there are $q, t \in \QQ$ with $q < c$ and $a + 1 < t$. Let $k, l \in \NN_{> 0}$ be such that $\frac{1}{k} < s-r$ and $l > t-q$. Define $n \dfeq k \cdot l$.
				
				Note that
				\[\big(q + n \cdot s\big) - \big(t + n \cdot r\big) = k l \cdot (s-r) - (t-q) > 0,\]
				so
				\[a + n \cdot b < t + n \cdot r < q + n \cdot s < c + n \cdot d,\]
				proving the claim.
			\end{proof}
			
			These results have consequences for prestreak morphisms with an archimedean codomain: they not only preserve (as all morphisms) the order $<$ (both internal and with rationals), but also reflect it. First a more general lemma, though.
			
			\begin{lemma}\label{Lemma: preservation_and_reflection_of_order}
				Let $X$, $Y$ be prestreaks and $f\colon X \to Y$ a map (we do not assume that it is a prestreak morphism). Consider the following statements (each expressing preservation or reflection of a sort of $<$).
				\begin{enumerate}
					\item[(a)]
						$\xall{q}{\QQ}\all[1]{b}{X}{q < b \implies q < f(b)}$
					\item[(b)]
						$\xall{q}{\QQ}\all[1]{b}{X}{q < f(b) \implies q < b}$
					\item[(c)]
						$\xall{r}{\QQ}\all[1]{a}{X}{a < r \implies f(a) < r}$
					\item[(d)]
						$\xall{r}{\QQ}\all[1]{a}{X}{f(a) < r \implies a < r}$
					\item[(e)]
						$\all[1]{a, b}{X}{a < b \implies f(a) < f(b)}$
					\item[(f)]
						$\all[1]{a, b}{X}{f(a) < f(b) \implies a < b}$
				\end{enumerate}
				Then the following holds.
				\begin{enumerate}
					\item
						If $Y$ is archimedean, then
						\begin{itemize}
							\item
								(b) $\land$ (d) $\implies$ (f),
							\item
								(a) $\implies$ (d),
							\item
								(c) $\implies$ (b).
						\end{itemize}
					\item
						If $X$ is archimedean, then
						\begin{itemize}
							\item
								(a) $\land$ (c) $\implies$ (e),
							\item
								(d) $\implies$ (a),
							\item
								(b) $\implies$ (c).
						\end{itemize}
					\item
						If both $X$ and $Y$ are archimedean, then (a) $\iff$ (d), (c) $\iff$ (b), and each of (a) $\land$ (b), (a) $\land$ (c), (b) $\land$ (d), (c) $\land$ (d) implies all six statements (a), (b), (c), (d), (e), (f).
				\end{enumerate}
			\end{lemma}
			\begin{proof}
				\begin{enumerate}
					\item
						Let $a, b \in X$ with $f(a) < f(b)$. Since $Y$ is archimedean, there exists $q \in \QQ$ with $f(a) < q < f(b)$ by Lemma~\ref{Lemma: density_of_rationals_in_an_archimedean_prestreak}. The assumptions (b) and (d) give us $a < q < b$ whence $a < b$.
						
						As far as the second claim is concerned, assume (a) and take any $r \in \QQ$, $a \in X$ with $f(a) < r$. Since $Y$ is archimedean, there exists $q \in \QQ$ with $f(a) < q < r$ by Lemma~\ref{Lemma: density_of_rationals_in_an_archimedean_prestreak}. By cotransitivity $q < a \lor a < r$, but the disjunct $q < a$ leeds to contradiction, as it with (a) implies $f(q) < a$. Therefore $a < r$, as claimed.
						
						The implication (c) $\implies$ (b) is proved the same way.
					\item
						The same as in the previous item.
					\item
						It remains to prove only (a) $\land$ (b) $\implies$ (c) (as well as (c) $\land$ (d) $\implies$ (a), but that is shown the same way).
						
						Assume (a) and (b). Take $r \in \QQ$, $a \in X$ and suppose $a < r$. Since $X$ is archimedean, there exists $q \in \QQ$ with $a < q < r$ by Lemma~\ref{Lemma: density_of_rationals_in_an_archimedean_prestreak}. Assume we had $q < f(a)$. By (b) this means $q < a$, in contradiction to $a < q$. Thus $f(a) \leq q < r$, so $f(a) < r$.
				\end{enumerate}
			\end{proof}
			
			\begin{corollary}\label{Corollary: archimedean_prestreak_morphisms_reflect_order}
				For a morphism $f\colon X \to Y$ with $Y$ archimedean we have
				\[x < q \iff f(x) < q,  \qquad\qquad  q < y \iff q < f(y),\]
				\[x < y \iff f(x) < f(y),\]
				for all $x, y \in X$ and $q \in \QQ$.
			\end{corollary}
			\begin{proof}
				As a morphism, $f$ preserves all versions of $<$. By Lemma~\ref{Lemma: preservation_and_reflection_of_order} it reflects them as well.
			\end{proof}
			
			The archimedean property has another consequence: it lets us define comparison between elements of different prestreaks. Let $X$, $Y$ be archimedean prestreaks and $a \in X$, $b \in Y$. We define
			\[a < b \dfeq \xsome{q}{\QQ}{a < q < b}\]
			(and, as usual, $a \leq b \dfeq \lnot(b < a)$. Use the symbol $<$ for this relation as well is admissable since by the results of this subsection this new order is a generalization of all previously defined ones (in cases $X = Y$ or one of $X$, $Y$ is $\QQ$). In fact, a general invariance of order holds.
			\begin{lemma}\label{Lemma: invariance_of_order_across_archimedean_prestreaks}
				Let $X$, $X'$, $X''$, $Y$, $Y'$, $Y''$ be archimedean streaks and $f\colon X \to X'$, $g\colon X \to X''$, $h\colon Y \to Y'$, $k\colon Y \to Y''$ morphisms. Then for any $a \in X$, $b \in Y$
				\[f(a) < h(b) \iff g(a) < k(b) \qquad \text{and} \qquad f(a) \leq h(b) \iff g(a) \leq k(b).\]
			\end{lemma}
			\begin{proof}
				Using Corollary~\ref{Corollary: archimedean_prestreak_morphisms_reflect_order} we have
				\[f(a) < h(b) \iff \xsome{q}{\QQ}{f(a) < q < h(b)} \iff \xsome{q}{\QQ}{a < q < b} \iff\]
				\[\iff \xsome{q}{\QQ}{g(a) < q < k(b)} \iff g(a) < k(b).\]
				For $\leq$ just negate all sides of the equivalence.
			\end{proof}
			\begin{remark}
				Consider the statement of Lemma~\ref{Lemma: invariance_of_order_across_archimedean_prestreaks} when some of the morphisms are identities.
			\end{remark}
		
		\subsection{Streaks}
		
			The structure of a prestreak is the foundation, upon which we'll add the various additional structures, up to the point of reals, but is still too general. First of all, as discussed in the previous section, we require the archimedean property, so that we are bound to structures in which rationals are dense. Second, we'll want to have a preorder category (\ie one where there exists at most one morphism from one given object to another) which $\Pstr$ and $\Apstr$ are not, so we cut them down further, to finally get our desired notion.
			
			\begin{definition}\label{Definition: streak}
				\
				\begin{itemize}
					\item
						A prestreak is called \df{tight} when $\leq$ is antisymmetric (thus a partial order):
						\[a \leq b \land b \leq a \implies a = b,\]
						or equivalently, when $\apart$ is tight (thus a tight apartness):
						\[\lnot(a \apart b) \implies a = b.\]
					\item
						A \df{streak}\footnote{There are two reasons, why I settled for the name `streak'. First, if you draw examples (as subsets of the real line) and behold them from afar, they actually look like streaks (try some examples, such as generated by the following subsets of $\RR$: $\emptyset$, $\{-1\}$, $\{1.5\}$, $\{0.5\}$, $\{-1.5\}$, $\{-1, \pi\}$, \ldots). Second, originally I defined streaks directly rather than via prestreaks, and I liked the pun (``a \df{streak} is a strict linear order\ldots''), especially since it works not just in English, but also in my own language, even though the words are completely different (``\df{proga} je stroga linearna urejenost\ldots'').} is a tight archimedean prestreak.
				\end{itemize}
			\end{definition}
			We use the same notion of a morphism for streaks as for prestreaks, \ie the category of streaks $\Str$ is a full subcategory of $\Pstr$ (and $\Apstr$).
			
			Streaks are cancellative monoids for addition.
			\begin{proposition}
				For a streak $X$ and $a, b, x \in X$ it holds
				\[a + x = b + x \iff a = b\]
				(the left-to-right implication is called the \df{cancellation property}).
			\end{proposition}
			
			One of corollaries is that in streaks we can define \df{subtraction} as a partial operation. For a streak $X$ and $a, b \in X$ we define $b-a$ to be that $x \in X$ (if it exists) which satisfies $a + x = b$. If there is another $x' \in X$ with $a + x' = b$, it follows $x = x'$ by the cancellation property.
			
			We can always calculate the differences $a-a = 0$ and $a-0 = a$ for any $a \in X$. Moreover, given $a, b, c, d \in X$, if the differences $b-a$ and $d-c$ exist, then the difference $(b+d)-(a+c)$ exists also, namely $(b+d)-(a+c) = (b-a) + (d-c)$; to see this, simply calculate $a+c + \big((b-a) + (d-c)\big) = a + (b-a) + c + (d-c) = b+d$.
			
			Streak morphisms preserve subtraction: if $f\colon X \to Y$ is a morphism between streaks and $a, b \in X$ are elements such that $b-a$ exists in $X$, then the difference $f(b)-f(a)$ exists in $Y$, and $f(b-a) = f(b)-f(a)$ since
			\[f(a) + f(b-a) = f\big(a + (b-a)\big) = f(b).\]
			
			Tightness of $\apart$ enables us to infer equality of two elements, if they have the same lower bounds, or if they have the same upper bounds. However, with the archimedean property knowing just the lower or upper \emph{rational} bounds is sufficient.
			\begin{lemma}\label{Lemma: element_in_a_streak_determined_by_rational_bounds}
				Let $X$ be a streak and $a, b \in X$. The following statements are equivalent:
				\begin{enumerate}
					\item
						$a = b$,
					\item
						$\all{q}{\QQ}{q < a \iff q < b}$,
					\item
						$\all{q}{\QQ}{a < q \iff b < q}$,
					\item
						$\all[1]{q}{\QQ}{(q < a \implies q < b) \land (a < q \implies b < q)}$.
				\end{enumerate}
			\end{lemma}
			\begin{proof}
				Of course $(1)$ implies the other statements. Among $(2 \impl 1)$ and $(3 \impl 1)$ we prove only the former, as the latter is analogous.
				
				Suppose $a < b$. By Lemma~\ref{Lemma: density_of_rationals_in_an_archimedean_prestreak} there exists $q \in \QQ$ such that $a < q < b$, contradicting our assumption $(2)$ (since $a < q$ implies $\lnot(q < a)$). In exactly the same way we obtain a contradiction from $b < a$. Thus $a = b$ by tightness.
				
				Finally, we prove $(4 \impl 1)$. Suppose $a < b$; we then have $q \in \QQ$ with $a < q < b$ by Lemma~\ref{Lemma: density_of_rationals_in_an_archimedean_prestreak}. Applying $(4)$, we obtain $b < q < b$, a contradiction. Similarly $b < a$ leads to contradiction as well. Thus $a = b$.
			\end{proof}
			
			\begin{theorem}\label{Theorem: preservation_of_rationals_implies_streak_morphism}
				Let $X$ be a prestreak and $Y$ a streak.
				\begin{enumerate}
					\item
						There exists at most one map $f\colon X \to Y$ which preserves comparison with rational numbers on both sides, \ie $\xall{q}{\QQ}\all{a}{X}{q < a \implies q < f(a)}$ and $\xall{q}{\QQ}\all{a}{X}{a < q \implies f(a) < q}$.
					\item
						If $X$ is archimedean and a map $f\colon X \to Y$ preserves comparison with rationals on both sides, then $f$ is a morphism.
				\end{enumerate}
			\end{theorem}
			\begin{proof}
				\begin{enumerate}
					\item
						Suppose $f, g\colon X \to Y$ both preserve comparison with rationals. By Lemma~\ref{Lemma: preservation_and_reflection_of_order} they reflect it as well. Thus for any $x \in X$ and $q \in \QQ$
						\[q < f(x) \iff q < x \iff q < g(x),\]
						so $f(x) = g(x)$ by Lemma~\ref{Lemma: element_in_a_streak_determined_by_rational_bounds}.
					\item
						Let $X$ be archimedean and $f\colon X \to Y$ a map which preserves comparison with rationals. By Lemma~\ref{Lemma: preservation_and_reflection_of_order} it then preserves and reflects all versions of $<$.
						
						Take $a, b \in X$. Suppose $f(a + b) < f(a) + f(b)$; then there exist $q, r \in \QQ$ such that $f(a + b) < q < r < f(a) + f(b)$ by Lemma~\ref{Lemma: density_of_rationals_in_an_archimedean_prestreak}. As $f$ reflects order, we have $a + b < q$. Let $k \in \NN_{> 0}$ be large enough such that $\frac{1}{k} < \frac{r-q}{2}$. By Lemma~\ref{Lemma: archimedean_prestreak_is_a_union_of_rational_intervals} there exist $i, j \in \ZZ$ such that $a \in \intoo[X]{\frac{i-1}{k}}{\frac{i+1}{k}}$ and $b \in \intoo[X]{\frac{j-1}{k}}{\frac{j+1}{k}}$. Since $f$ preserves comparison with rationals, we infer $f(a) < \frac{i+1}{k}$ and $f(b) < \frac{j+1}{k}$, so
						\[f(a) + f(b) < \tfrac{i+1}{k} + \tfrac{j+1}{k} = \tfrac{i-1}{k} + \tfrac{j-1}{k} + \tfrac{2}{k} < a + b + r-q < q + r-q = r,\]
						a contradition with $f(a) + f(b) > r$.
						
						Similarly, suppose $f(a + b) > f(a) + f(b)$. Find $q, r \in \QQ$ such that $f(a) + f(b) < q < r < f(a + b)$, a large enough $k \in \NN_{> 0}$ so that $\frac{1}{k} < \frac{r-q}{2}$, and $i, j \in \ZZ$ such that $f(a) \in \intoo[Y]{\frac{i-1}{k}}{\frac{i+1}{k}}$ and $f(b) \in \intoo[Y]{\frac{j-1}{k}}{\frac{j+1}{k}}$. Then $a \in \intoo[X]{\frac{i-1}{k}}{\frac{i+1}{k}}$ and $b \in \intoo[X]{\frac{j-1}{k}}{\frac{j+1}{k}}$, so
						\[a + b < \tfrac{i+1}{k} + \tfrac{j+1}{k} = \tfrac{i-1}{k} + \tfrac{j-1}{k} + \tfrac{2}{k} < f(a) + f(b) + r-q < q + r-q = r,\]
						contradicting $f(a + b) > r$.
						
						This shows $f(a + b) = f(a) + f(b)$. In particular $f(0) = f(0 + 0) = f(0) + f(0)$, whence $f(0) = 0$ by cancellation.
						
						Assume now additionally that $a, b > 0$ and suppose $f(a \cdot b) < f(a) \cdot f(b)$. Find $q, r \in \QQ$, $f(a \cdot b) < q < r < f(a) \cdot f(b)$ (in particular $a \cdot b < q$), use the archimedean property to get $n \in \NN_{\geq 2}$ such that $f(a) < n$, $f(b) < n$, and choose $k \in \NN_{> 0}$ large enough so that $\frac{n}{k} < \frac{r - q}{4}$. By Lemma~\ref{Lemma: density_of_rationals_in_an_archimedean_prestreak} there are $i, j \in \intoo[\NN]{0}{n k}$ such that $f(a) \in \intoo[Y]{\frac{i-1}{k}}{\frac{i+1}{k}}$ and $f(b) \in \intoo[Y]{\frac{j-1}{k}}{\frac{j+1}{k}}$, therefore also $a \in \intoo[X]{\frac{i-1}{k}}{\frac{i+1}{k}}$ and $b \in \intoo[X]{\frac{j-1}{k}}{\frac{j+1}{k}}$. We calculate
						\[f(a) \cdot f(b) < \tfrac{i+1}{k} \cdot \tfrac{j+1}{k} = \tfrac{i-1}{k} \cdot \tfrac{j-1}{k} + \tfrac{2 (i + j)}{k^2} < a \cdot b + \tfrac{4 n k}{k^2} < q + \tfrac{4 n}{k} < r,\]
						a contradiction. Use a similar method to derive a contradiction from $f(a \cdot b) < f(a) \cdot f(b)$. We conclude $f(a \cdot b) = f(a) \cdot f(b)$. In particular $f(1) = f(1 \cdot 1) = f(1) \cdot f(1)$ whence $f(1) = 1$.
				\end{enumerate}
			\end{proof}
			
			\begin{corollary}\label{Corollary: streaks_preorder_category}
				There exists at most one morphism from a prestreak $X$ to a streak $Y$. In particular, $\Str$ is a preorder category.
			\end{corollary}
			\begin{proof}
				Every morphism preserves comparison with rationals, so we may apply the preceding theorem to get uniqueness.
			\end{proof}
			
			We have seen what happens when the codomain of a morphism is a streak. We consider now also the domain.
			\begin{proposition}\label{Proposition: streak_morphisms_injective}
				Let $f\colon X \to Y$ be a morphism from a streak $X$ to a prestreak $Y$. Then $f$ is injective.
			\end{proposition}
			\begin{proof}
				By definition morphisms preserve $<$ so they preserve $\apart$; using the contrapositive of this statement and tightness of $\apart$ in $X$, we obtain
				\[f(a) = f(b) \implies \lnot\big(f(a) \apart f(b)\big) \implies \lnot(a \apart b) \implies a = b\]
				for all $a, b \in X$.
			\end{proof}
			
			\begin{corollary}\label{Corollary: surjective_streak_morphism_is_iso}
				A surjective streak morphism is an isomorphism.
			\end{corollary}
			\begin{proof}
				A surjective streak morphism $f\colon X \to Y$ is bijective by Proposition~\ref{Proposition: streak_morphisms_injective} and $f^{-1}$ preserves order (that is, $f$ reflects order) by Corollary~\ref{Corollary: archimedean_prestreak_morphisms_reflect_order}. Also, there is a standard way in algebra, how one shows that a bijective homomorphism is an isomorphism, and this works to show that $f^{-1}$ preserves $+$, $0$, $\cdot$, $1$ in our case as well.
			\end{proof}
			
			Between Corollary~\ref{Corollary: streaks_preorder_category} and Proposition~\ref{Proposition: streak_morphisms_injective} we see, that we can view streaks (up to isomorphism) as certain subsets, ordered by inclusion. The smallest among these subsets are the natural numbers.
			\begin{proposition}\label{Proposition: N_initial_streak}
				$\NN$ is the initial streak.
			\end{proposition}
			\begin{proof}
				We already know that $\NN$ is an initial prestreak from Proposition~\ref{Proposition: N_initial_prestreak}, and since $\Str$ is a full subcategory in $\Pstr$, we just need to check that $\NN$ is a streak.
				
				Of course, $\leq$ is antisymmetric in $\NN$. As for the archimedean property, take any $a, b, c, d \in \NN$ with $b < d$. Set $n = \sup\{a, c\} - c + 1$; then $n \in \NN$ and
				\[a + n \cdot b < c + n + n \cdot b = c + n \cdot (b+1) \leq c + n \cdot d.\]
			\end{proof}
			
			The point of the introduction of streaks is that the other extreme --- the largest streak --- is the set of real numbers (indeed, streaks are, up to isomorphism, just subsets of $\RR$ which contain $0$, $1$, and are closed under addition and multiplication of positive elements).
			
			But before we discuss this (in Section~\ref{Section: reals}), we say a few more things about streaks. As one can gleam from Theorem~\ref{Theorem: preservation_of_rationals_implies_streak_morphism}, it is sufficient for a map between streaks just to preserve comparison with rationals, and it already follows that it is a streak morphism. In a similar vein, one can consider an alternative definition of a streak, where comparison with rationals is taken as a primitive piece of the structure (the property $a < b \iff \xsome{q}{\QQ}{a < q < b}$ from Lemma~\ref{Lemma: density_of_rationals_in_an_archimedean_prestreak} provides the connection between the two definitions).
			
			\begin{proposition}[Alternative definition of a streak]\label{Proposition: alt_streak}
				Let $X$ be equipped with relations $< \subseteq \QQ \times X$ and $< \subseteq X \times \QQ$ (which we denote by the same symbol) and operations $+\colon X \times X \to X$ and $\cdot\colon X_{> 0} \times X_{> 0} \to X_{> 0}$ (where $X_{> 0}$ stands for the set of all elements in $X$ which are bigger than the \emph{rational} $0$). Suppose the following conditions hold:
				\begin{itemize}
					\item
						boundedness: $\xall{a}{X}\xsome{q,r}{\QQ}{q < a < r}$,
					\item
						all possible cotransitivity conditions: for all $a \in X$ and $q, r \in \QQ$
						\[q < a \implies q < r \lor r < a,  \qquad  a < q \implies a < r \lor r < q,\]
						\[q < r \implies q < a \lor a < r,\]
					\item
						for all $a \in X$ and $q \in \QQ$
						\[q < a \implies \xsome{r}{\QQ}{q < r < a},  \qquad  a < q \implies \xsome{r}{\QQ}{a < r < q},\]
					\item
						asymetry: $\xall{a}{X}\xall{q}{\QQ}{\lnot\big(q < a \land a < q\big)}$,
					\item
						for all $a, b \in X$
						\[a = b \iff \all{q}{\QQ}{q < a \iff q < b} \iff \all{q}{\QQ}{a < q \iff b < q},\]
					\item
						both relations $<$ are open and their negations $\leq$ are closed,
					\item
						$+$ makes $X$ into a commutative monoid,
					\item
						for all $q, r \in \QQ$ and $x, y \in X$
						\[q < x \land r < y \implies q + r < x + y, \qquad x < q \land y < r \implies x + y < q + r,\]
					\item
						$\cdot$ makes $X_{> 0}$ into a commutative monoid and distributes over $+$,
					\item
						for all $q, r \in \QQ_{> 0}$ and $x, y \in X_{> 0}$
						\[q < x \land r < y \implies q \cdot r < x \cdot y, \qquad x < q \land y < r \implies x \cdot y < q \cdot r.\]
				\end{itemize}
				Then $X$ is a streak if we define for $a, b \in X$
				\[a < b \dfeq \some{q}{\QQ}{a < q \land q < b}.\]
				Conversely, any streak satisfies the above conditions.
			\end{proposition}
			\begin{proof}
				We already know from definitions and previous propositions that streaks satisfy above conditions. Here is the verification of the converse. For easier readability we split it into parts.
				
				\begin{itemize}
					\item\proven{properties of $<$ and $\leq$}
						Suppose we have $a < b$ and $b < a$, \ie there are $q, r \in \QQ$ such that $a < q < b$ and $b < r < a$. Since $<$ is decidable on $\QQ$ and $a$, $b$ are interchangable, we may without loss of generality assume $q \leq r$. By cotransitivity we have $r < q \lor q < a$, but the first disjunct is false, so $q < a$ which together with $a < q$ contradicts the asymetry condition. Hence $<$ is asymetric on $X$ as well.
						
						Suppose $a < b$, so $a < q < b$ for some $q \in \QQ$. Let $x \in X$. There is some $r \in \QQ$ such that $q < r < b$, so by cotransitivity $q < x \lor x < r$. If the first is the case, we have $a < x$, and if the second, then $x < b$.
						
						We have $a \leq b \iff \lnot(b < a) \iff \lnot\some{q}{\QQ}{a < q \land q < b} \iff \all{q}{\QQ}{a < q \implies b \leq q}$. Assume the last, then take any $q \in \QQ$ such that $a < q$. Then there is $r \in \QQ$, $a < r < q$. By cotransitivity $r < b \lor b < q$, but the first disjunt cannot hold by our assumption, so $b < q$. We conclude that $a \leq b$ is furthermore equivalent to $\all{q}{\QQ}{a < q \implies b < q}$, so $a \leq b \land b \leq a$ implies $a = b$, meaning that $\leq$ is antisymmetric.
						
						Since $\QQ$ is countable, the predicate $\some{q}{\QQ}{a < q \land q < b}$ is open, therefore $<$ is an open relation. Similarly we see that $\leq$ is closed, as $\lnot\some{q}{\QQ}{b < q \land q < a}$ is equivalent to $\xall{q}{\QQ}{\lnot\lnot(a \leq q \lor q \leq b)}$.
				\end{itemize}
				
				Note that all forms of transitivity hold. For $q, r \in \QQ$ and $x \in X$, if $q < r$ and $r < x$, then by cotransitivity $r < q \lor q < x$, but the first disjunct leads to contradiction with asymmetry. The other two cases $x < q < r \implies x < r$ and $q < x < r \implies q < r$ are shown similarly.
				
				Also observe that the proof of Lemma~\ref{Lemma: archimedean_prestreak_is_a_union_of_rational_intervals} uses only the properties, given in the text of this proposition, so we may use its results for $X$. We do so several times in the remainder of the proof.
				
				\begin{itemize}
					\item\proven{interaction of $+$ and $<$}
						Take any $a, b, x \in X$. Assume $a < b$; then we have $q \in \QQ$ with $a < q < b$ and $r \in \QQ$ with $q < r < b$. Let $k \in \NN_{> 0}$ be large enough so that $\frac{1}{k} < \frac{r-q}{2}$ and let $i \in \ZZ$ be such that $x \in \intoo[X]{\frac{i-1}{k}}{\frac{i+1}{k}}$. Then $a + x < q + \frac{i+1}{k} < r + \frac{i-1}{k} < b + x$.
						
						The converse direction has the similar idea, albeit it is slightly more involved. Suppose we have $a + x < q < r < b + x$ with $q, r \in \QQ$ and $x \in \intoo[X]{\frac{i-1}{k}}{\frac{i+1}{k}}$ where $k \in \NN$ is this time large enough for $\frac{1}{k} < \frac{r-q}{5}$ to hold. By cotransitivity $q - \frac{i-1}{k} < a \lor a < q - \frac{i-2}{k}$, but the first disjunct leads to contradiction $q < a + x$, therefore $a < q - \frac{i-2}{k}$. Similarly $r - \frac{i+2}{k} < b \lor b < r - \frac{i+1}{k}$, but the second disjunct leads to contradiction $b + x < r$, therefore $r - \frac{i+2}{k} < b$. Altogether we have $a < q - \frac{i-2}{k} < r - \frac{i+2}{k} < b$.
					\item\proven{interaction of $\cdot$ and $<$}
						This part works similarly as the previous one, but there are more bounds that need to be considered.
						
						Take $a, b, x \in X_{> 0}$, assume $a < b$ and find $q, r, s \in \QQ$ with $a < q < r < b$ and $0 < s < x$. Let $k \in \NN$ be large enough for $k > \frac{2 r}{(r-q) \cdot s}$ to hold. Find $i \in \NN_{> 0}$ with $x \in \intoo[X]{\frac{i-1}{k}}{\frac{i+1}{k}}$. We have $a \cdot x < q \cdot \frac{i+1}{k} < r \cdot \frac{i-1}{k} < b \cdot x$, the middle inequality being valid because of the following.
						\[s < x < \tfrac{i+1}{k} \implies \frac{1}{i+1} < \frac{1}{k s}\]
						\[\frac{r \cdot \tfrac{i-1}{k}}{q \cdot \tfrac{i+1}{k}} = \frac{r}{q} \cdot \frac{i-1}{i+1} = \frac{r}{q} \cdot \left(1 - \frac{2}{i+1}\right) > \frac{r}{q} \cdot \left(1 - \frac{2}{k s}\right) > \frac{r}{q} \cdot \left(1 - \frac{r-q}{r}\right) = 1\]
						
						For the converse, assume $a \cdot x < b \cdot x$ and find $q, q', r, r', s \in \QQ$ with $a \cdot x < q' < q < r < r' < b \cdot x$ and $0 < s < x$. Let $k \in \NN$ be large enough for $k > \frac{2 r}{(r-q) \cdot s}$ and $\frac{1}{k} < \frac{s}{2}$ to hold. Find $i \in \NN_{> 0}$ with $x \in \intoo[X]{\frac{i-1}{k}}{\frac{i+1}{k}}$. Since $\frac{i+1}{k} > x > s > \frac{2}{k}$, we have $i > 1$. Also, as before, $\frac{1}{i+1} < \frac{1}{k s}$ and
						\[\frac{r \cdot \frac{k}{i+1}}{q \cdot \frac{k}{i-1}} = \frac{r}{q} \cdot \frac{i-1}{i+1} > 1.\]
						By cotransitivity $q' \cdot \frac{k}{i-1} < a \lor a < q \cdot \frac{k}{i-1}$, but the first disjunct implies $q' < a \cdot x$, a contradiction. Similarly in $r \cdot \frac{k}{i+1} < b \lor b < r' \cdot \frac{k}{i+1}$ the second disjunct leads to contradiction $b \cdot x < r'$. In conclusion $a < q \cdot \frac{k}{i-1} < r \cdot \frac{k}{i+1} < b$.
					\item\proven{archimedean property}
						Since $X$ is a monoid for $+$, we have multiplication with  natural numbers and the statement of the archimedean property makes sense for $X$. Take any $a, b, c, d \in X$ and $q \in \QQ$ such that $b < q < d$. Take $r \in \QQ$ with $q < r < d$. Let $k \in \NN_{> 0}$ be large enough so that $\frac{1}{k} < \frac{r-q}{2}$, and let $i, j \in \ZZ$ be such that $a \in \intoo[X]{\frac{i-1}{k}}{\frac{i+1}{k}}$, $c \in \intoo[X]{\frac{j-1}{k}}{\frac{j+1}{k}}$. Since $\QQ$ is archimedean,\footnote{This follows from Theorem~\ref{Theorem: from_ring_streaks_to_field_streaks} below, but it is not difficult to check it directly. One can for example multiply the required inequality with all denominators and take negative terms to the other side of the inequality, thus translating the archimedean condition of $\QQ$ to that of $\NN$. Or one can verify that in any strictly ordered field $F$ the archimedean condition is equivalent to $\xall{x}{F}\xsome{n}{\NN}{x < n}$.} there is $n \in \NN$ with $\frac{i+1}{k} + n \cdot q < \frac{j-1}{k} + n \
 cdot r$. Clearly then $a + n \cdot b < c + n \cdot d$.
				\end{itemize}
			\end{proof}
			
			\begin{remark}
				Since $<$ is decidable on $\QQ$, two cotransitivity conditions can be restated as
				\[q < a \land r \leq q \implies r < a,  \qquad  a < q \land q \leq r \implies a < r\]
				for all $a \in X$, $q, r \in \QQ$. Thus the first three items in Proposition~\ref{Proposition: alt_streak} essentially correspond to the conditions for Dedekind cuts.
			\end{remark}
			
		\subsection{Multiplicative streaks}
		
			Somewhere along the path from streaks to the real numbers we'll eventually want the multiplication to become a total operation. Despite this fact, we did not assume the totality of $\cdot$ already in the definition of a (pre)streak, since it is much more convenient to start with multiplication just for positive numbers, as this preserves the order and makes defining $\cdot$ in concrete examples way easier. Compare, for example, multiplication of Dedekind reals just for positive ones, or as a total operation. The latter is especially problematic constructively, as separation of cases is not allowed (see~\cite{Richman_F_2008:_real_numbers_and_other_completions} for a discussion on this subject).
			
			Nevertheless, we want to see what happens when we do have total multiplication in a streak. Actually, the definition can be stated for prestreaks in general.
			
			\begin{definition}\label{Definition: multiplicative_prestreak}
				A prestreak $X$ is called \df{multiplicative} when it is equipped with a total multiplication operation $\cdot\colon X \times X \to X$ which is commutative, associative, distributes over addition, satisfies the \df{multiplicity condition}
				\[b < a \land d < c \implies a \cdot d + b \cdot c < a \cdot c + b \cdot d\]
				for all $a, b, c, d \in X$, restricts to the prestreak multiplication on $X_{> 0}$, and $1$ is the unit for total $\cdot$ also.
			\end{definition}
			
			First we discuss what the deal with the multiplicity condition is. If a prestreak has subtraction (that is, it is not just a monoid, but a group for $+$), then the multiplicity condition can be rewritten as
			\[0 < a-b \land 0 < c-d \implies 0 < (a-b) \cdot (c-d).\]
			But products of positive elements are positive in any prestreak by definition; thus the multiplicity condition is vacant, when we have subtraction. In general, it provides a necessary generalization of the fact that positive times positive is positive, when subtraction is not available.
			
			Even when we don't have total multiplication, the multiplicity condition still holds for positive elements in an archimedean prestreak.
			\begin{proposition}\label{Proposition: multiplicity_condition_for_positive_elements}
				Let $X$ be an archimedean prestreak and $a, b, c, d \in X$ such that $0 < b < a$ and $0 < d < c$ holds. Then we have $a \cdot d + b \cdot c < a \cdot c + b \cdot d$.
			\end{proposition}
			\begin{proof}
				Find rationals $q, r, s, t, u \in \QQ$ such that $b < q < r < a < s$ and $d < t < u < c < s$ and let $n \in \NN_{> 0}$ be large enough that $\frac{1}{n} < \frac{(r-q) (u-t)}{72 (s+1)}$, $\frac{1}{n} < \frac{r-q}{3}$ and $\frac{u-t}{3}$ hold. Let $i, j, k, l \in \NN$ be such that $a \in \intoo[X]{\frac{i-1}{n}}{\frac{i+1}{n}}$, $b \in \intoo[X]{\frac{j-1}{n}}{\frac{j+1}{n}}$, $c \in \intoo[X]{\frac{k-1}{n}}{\frac{k+1}{n}}$, $d \in \intoo[X]{\frac{l-1}{n}}{\frac{l+1}{n}}$.
				
				Since $\frac{i+1}{n} > a > r$ and $\frac{1}{n} < \frac{r-q}{3}$, we have $\frac{i}{n} > \frac{2r+q}{3}$. Similarly we obtain $\frac{j}{n} < \frac{2q+r}{3}$, $\frac{k}{n} > \frac{2u+t}{3}$ and $\frac{l}{n} < \frac{2t+u}{3}$. Hence $\frac{i-j}{n} > \frac{r-q}{3}$ and $\frac{k-l}{n} > \frac{u-t}{3}$. Also, $\frac{i+j+k+l-4}{n} < a + b + c + d < 4 s$, whence $\frac{i+j+k+l}{n} < 4 s + \frac{4}{n} \leq 4 (s+1)$.
				
				All of this allows us to show (using positiveness of terms when needed)
				\[\frac{(i-j) (k-l) - 2 (i+j+k+l)}{n^2} > \frac{r-q}{3} \cdot \frac{u-t}{3} - \frac{8 (s+1)}{n} > 0.\]
				Therefore (using some properties from Proposition~\ref{Proposition: alt_streak})
				\[a \cdot d + b \cdot c < \frac{i+1}{n} \cdot \frac{l+1}{n} + \frac{j+1}{n} \cdot \frac{k+1}{n} =\]
				\[= \frac{i-1}{n} \cdot \frac{k-1}{n} + \frac{j-1}{n} \cdot \frac{l-1}{n} - \frac{(i-j) (k-l) - 2 (i+j+k+l)}{n^2} <\]
				\[< \frac{i-1}{n} \cdot \frac{k-1}{n} + \frac{j-1}{n} \cdot \frac{l-1}{n} < a \cdot c + b \cdot d.\]
			\end{proof}
			
			The multiplicity condition can also be seen as the generalization of the condition $a < b \implies a \cdot x < b \cdot x$ for $a, b, x \in X_{> 0}$ in a general prestreak $X$. However, in a general prestreak the reverse implication was also assumed, and indeed we should also assume the reverse implication $a \cdot d + b \cdot c \apart a \cdot c + b \cdot d \implies b \apart a \land d \apart c$ for $\apart$ in a general multiplicative prestreak (the reverse implication for $<$ cannot hold; consider \eg $a = c = 0$, $b = d = 1$). However, we will not make this additional assumption, as we'll only be really interested in archimedean multiplicative prestreak, and for those this is automatically the case, as we check presently.
			
			\begin{lemma}\label{Lemma: bounds_in_multiplicative_prestreak}
				Let $X$ be a multiplicative prestreak, $n \in \NN$ and $a, b, c, d \in X$ elements, satisfying
				\[n \cdot b < 1 + n \cdot a, \qquad n \cdot d < 1 + n \cdot c,\]
				\[n \cdot a < 1 + n \cdot b, \qquad n \cdot c < 1 + n \cdot d.\]
				Then
				\[n^2 \cdot (a \cdot d + b \cdot c) < 1 + n^2 \cdot (a \cdot c + b \cdot d),\]
				\[n^2 \cdot (a \cdot c + b \cdot d) < 1 + n^2 \cdot (a \cdot d + b \cdot c).\]
			\end{lemma}
			\begin{proof}
				Use the multiplicity condition first for $n \cdot b < 1 + n \cdot a$ and $n \cdot d < 1 + n \cdot c$, then for $n \cdot a < 1 + n \cdot b$ and $n \cdot c < 1 + n \cdot d$, to obtain
				\[n \cdot d + n^2 \cdot a \cdot d + n \cdot b + n^2 \cdot b \cdot c < 1 + n \cdot a + n \cdot c + n^2 \cdot a \cdot c + n^2 \cdot b \cdot d,\]
				\[n \cdot c + n^2 \cdot b \cdot c + n \cdot a + n^2 \cdot a \cdot d < 1 + n \cdot b + n \cdot d + n^2 \cdot b \cdot d + n^2 \cdot a \cdot c.\]
				Add the two inequalities, then cancel all that can be cancelled, to obtain the first claimed result.
				
				Proceed analogously after writing the multiplicity condition first for the pair $n \cdot b < 1 + n \cdot a$, $n \cdot c < 1 + n \cdot d$, then for $n \cdot a < 1 + n \cdot b$, $n \cdot d < 1 + n \cdot c$, to get the second claimed inequality.
			\end{proof}
			
			\begin{proposition}\label{Proposition: reverse_multiplicity_condition}
				Let $X$ be a multiplicative prestreak.
				\begin{enumerate}
					\item
						We have
						\[b \apart a \land d \apart c \implies a \cdot d + b \cdot c \apart a \cdot c + b \cdot d\]
						for all $a, b, c, d \in X$.
					\item
						If $X$ is archimedean, the reverse implication
						\[a \cdot d + b \cdot c \apart a \cdot c + b \cdot d \implies b \apart a \land d \apart c\]
						also holds.
				\end{enumerate}
			\end{proposition}
			\begin{proof}
				\begin{enumerate}
					\item
						Easily follows from consideration of all four cases in $(b < a \lor a < b) \land (d < c \lor c < d)$.
					\item
						Assume $a \cdot d + b \cdot c < a \cdot c + b \cdot d$ (the case $a \cdot d + b \cdot c > a \cdot c + b \cdot d$ is proved analogously, or, easier still, just switch $a$ and $b$).
						
						Use the archimedean property to find $n \in \NN$ (necessarily $n > 0$) that fulfils
						\[1 + n \cdot (a \cdot d + b \cdot c) < n \cdot (a \cdot c + b \cdot d),\]
						as well as $b < a + n$, $a < b + n$, $d < c + n$, $c < d + n$. By cotransitivity all the following disjunctions are valid:
						\[n^2 \cdot a < n^2 \cdot b \lor n^2 \cdot b < 1 + n^2 \cdot a, \qquad n^2 \cdot c < n^2 \cdot d \lor n^2 \cdot d < 1 + n^2 \cdot c,\]
						\[n^2 \cdot b < n^2 \cdot a \lor n^2 \cdot a < 1 + n^2 \cdot b, \qquad n^2 \cdot d < n^2 \cdot c \lor n^2 \cdot c < 1 + n^2 \cdot d.\]
						Some of the cases obviously lead to the desired conclusion $b \apart a \land d \apart c$.
						
						Of the others, consider $n^2 \cdot a < n^2 \cdot b$ --- that is, $a < b$  --- and $n^2 \cdot d < 1 + n^2 \cdot c$. Applying the multiplicity condition, we get
						\[b \cdot n^2 \cdot d + a \cdot (1 + n^2 \cdot c) < b \cdot (1 + n^2 \cdot c) + a \cdot n^2 \cdot d,\]
						or equivalently,
						\[a + n^2 \cdot (a \cdot c + b \cdot d) < b + n^2 \cdot (a \cdot d + b \cdot c).\]
						Adding this to the inequality $1 + n \cdot (a \cdot d + b \cdot c) < n \cdot (a \cdot c + b \cdot d)$ from above, multiplied by $n$, then canceling all that we can, we obtain $a + n < b$, a contradiction.
						
						Now consider the case when all of $n^2 \cdot b < 1 + n^2 \cdot a$, $n^2 \cdot a < 1 + n^2 \cdot b$, $n^2 \cdot d < 1 + n^2 \cdot c$, $n^2 \cdot c < 1 + n^2 \cdot d$ hold. By Lemma~\ref{Lemma: bounds_in_multiplicative_prestreak} we then have
						\[n^4 \cdot (a \cdot c + b \cdot d) < 1 + n^4 \cdot (a \cdot d + b \cdot c).\]
						However, if we add this to the inequality $1 + n \cdot (a \cdot d + b \cdot c) < n \cdot (a \cdot c + b \cdot d)$ from above, multiplied by $n^3$, then cancel all that we can, we get $n^3 < 1$, a contradiction.
						
						The remaining possible cases are similar.
				\end{enumerate}
			\end{proof}
			
			We can now generalize the property, connecting $\cdot$ and $<$.
			\begin{proposition}
				Let $X$ be a multiplicative archimedean prestreak and $a, b, x \in X$. If $x > 0$, then
				\[a < b \iff a \cdot x < b \cdot x,\]
				and in case $x < 0$, we have
				\[a < b \iff a \cdot x > b \cdot x.\]
			\end{proposition}
			\begin{proof}
				The directions $(\impl)$ are just cases of the multiplicity condition when one of the elements is $0$.
				
				Conversely, suppose $x > 0$ and $a \cdot x < b \cdot x$. Proposition~\ref{Proposition: reverse_multiplicity_condition} then implies $b \apart a$ and $x \apart 0$. We thus have $a < b \lor b < a$. If the first disjunct holds, we are done. Assume the second one. Then by the already shown direction $(\impl)$ we have $b \cdot x < a \cdot x$, a contradiction.
				
				The case $x < 0$ is shown the same way.
			\end{proof}
			
			Here is how the total mulitplication behaves with regard to signs (answer: as expected).
			\begin{proposition}\label{Proposition: multiplication_and_signs}
				Let $X$ be a multiplicative prestreak and $a, b \in X$.
				\begin{enumerate}
					\item
						The following holds:
						\[a > 0 \land b > 0 \implies a \cdot b > 0,\]
						\[a > 0 \land b < 0 \implies a \cdot b < 0,\]
						\[a < 0 \land b > 0 \implies a \cdot b < 0,\]
						\[a < 0 \land b < 0 \implies a \cdot b > 0.\]
						Hence, $a \apart 0 \land b \apart 0 \implies a \cdot b \apart 0$.
					\item
						If $X$ is archimedean, then we also have reverse implications in the following sense:
						\[a \cdot b > 0 \implies \big(a > 0 \land b > 0\big) \lor \big(a < 0 \land b < 0\big),\]
						\[a \cdot b < 0 \implies \big(a > 0 \land b < 0\big) \lor \big(a < 0 \land b > 0\big).\]
						In particular $a \cdot b \apart 0 \implies a \apart 0 \land b \apart 0$.
				\end{enumerate}
			\end{proposition}
			\begin{proof}
				\begin{enumerate}
					\item
						These are special cases of the multiplicity condition when two of the elements are $0$.
					\item
						The statement $a \cdot b \apart 0 \implies a \apart 0 \land b \apart 0$ follows directly from Proposition~\ref{Proposition: reverse_multiplicity_condition} (by setting two of the elements to $0$). The specific signs of $a$ and $b$ can be determined by considering all the possible cases and excluding those which lead to contradiction by the previous item.
				\end{enumerate}
			\end{proof}
			
			In the remainder of the subsection we observe two results which hold when we actually have a streak (rather than just a prestreak): multiplication in a multiplicative streak extends not just the standard streak multiplication (as per the definition), but also multiplication with natural numbers; moreover, any such extension of multiplication is unique.
			
			\begin{lemma}\label{Lemma: total_multiplication_coincides_with_multiplication_with_natural_numbers_in_multiplicative_streaks}
				Let $X$ be a multiplicative streak and let $\cdot'$ denote the total multiplication on $X$. Then for any $x \in X$ and $n \in \NN$ we have $n \cdot' x = n \cdot x$ (where $\cdot$ denotes the usual multiplication with natural numbers). In particular $0 \cdot' x = 0$ and $1 \cdot' x = x$.
			\end{lemma}
			\begin{proof}
				We first check the two given special cases. As usual, we have
				\[0 \cdot' x = (0 + 0) \cdot' x = 0 \cdot' x + 0 \cdot' x\]
				whence $0 \cdot' x = 0 = 0 \cdot x$ by cancellation.
				
				As far as the case $n = 1$ is concerned, if $x = 0$, we already know $1 \cdot' 0 = 0$, and if $x > 0$, then $1 \cdot' x = 1 \cdot x = x$ because $\cdot'$ restricts to the usual multiplication on positive elements. Consider now a general $x \in X$; let $m \in \NN$ be such that $0 < x + m$ (use the archimedean property). Then
				\[x + m = 1 \cdot (x + m) = 1 \cdot' (x + m) = 1 \cdot' x + 1 \cdot' m = 1 \cdot' x + m.\]
				Cancel $m$ to obtain $1 \cdot' x = x = 1 \cdot x$.
				
				For general $n \in \NN$, use induction. The base is already covered. As for the inductive step, assume $n \cdot' x = n \cdot x$. Then
				\[(n+1) \cdot' x = n \cdot' x + 1 \cdot' x = n \cdot x + x = (n+1) \cdot x.\]
			\end{proof}
			
			\begin{proposition}\label{Proposition: uniqueness_of_multiplicative_structure_in_streaks}
				A streak has at most one multiplicative structure.
			\end{proposition}
			\begin{proof}
				Let $X$ be a streak, $\cdot\colon X_{> 0} \times X_{> 0} \to X_{> 0}$ its usual multiplication and $\cdot', \cdot''\colon X \times X \to X$ any two operations which make $X$ into a multiplicative streak.
				
				Take any $a, b \in X$. By the archimedean property there are $m, n \in \NN$ such that $0 < a + m$, $0 < b + n$. Then
				\[(a + m) \cdot (b + n) = (a + m) \cdot' (b + n) = a \cdot' b + n \cdot' a + m \cdot' b + m \cdot' n = a \cdot' b + n \cdot a + m \cdot b + m \cdot n\]
				by the previous lemma. The same formula holds for $\cdot''$ whence
				\[a \cdot' b + n \cdot a + m \cdot b + m \cdot n = a \cdot'' b + n \cdot a + m \cdot b + m \cdot n,\]
				and after cancellation, $a \cdot' b = a \cdot'' b$.
			\end{proof}
			
		\subsection{Dense streaks}\label{Subsection: dense_streaks}
		
			Concrete constructions of reals typically amount to some sort of completion of rationals. The point is, one starts with natural numbers, there are standard ways to construct rationals from them, and rationals are dense in reals (essentially by definition).
			
			However, there is no necessity to obtain the reals specifically from the rationals; any dense subset would do. Using some other set is not just a theoretical possibility, but is actually done in practice: many computer implementations of reals use the dyadic rationals (= rationals of the form $a/2^b$ with $a \in \ZZ$, $b \in \NN$). On the other hand, one could consider a set of ``primitives'', bigger that $\QQ$, say the field, obtained by adjoining radicals to $\QQ$ (thus having it closed for another operation, namely the taking of roots), and construct the reals from those.
			
			In any case, forming reals from some other dense set is no more difficult than forming them from rationals, so we'll do it in this greater generality. In this subsection we define and characterize, what it means for a streak to be ``dense''.
			
			\begin{definition}
				A streak $X$ is \df{dense} when for all $q, r \in \QQ$ with $q < r$ there exists $x \in X$ with $q < x < r$.
			\end{definition}
			
			Obviously $\QQ$ itself is dense by this definition (take $x = \frac{q + r}{2}$); in fact, we involved $\QQ$ in the definition of density because we can consider it a ``model dense streak''. However, density of a streak can be characterized without involving the rationals.
			
			\begin{lemma}\label{Lemma: countable_dense_substreak}
				Let $X$ be a streak with some $z \in \intoo[X]{-1}{0}$.
				\begin{enumerate}
					\item
						The substreak $S$, generated by $z$, is dense.
					\item
						Assume further that $X$ is multiplicative. Then the multiplicative substreak $M$, generated by $z$, is the image of the map $f\colon \finseq{\NN} \to X$,
						\[f\Big((a_i)_{i \in \NN_{< m}}\Big) \dfeq \sum_{i \in \NN_{< m}} a_i z^i,\]
						and thus a countable dense multiplicative substreak of $X$.\footnote{The point of the second item is to show that a dense streak has a countable dense substreak, at least when it is multiplicative. As we see later (in Subsection~\ref{Subsection: ring_streaks}), any streak can be embedded into a multiplicative streak, and we could consider $S$ and $M$ in that one, though clearly still $S \subseteq X$. $M$ is countable and $S \subseteq M$, so in classical mathematics $S$ is countable as well, implying that any dense streak has a countable dense substreak. Constructively this argument doesn't work, though, as a subset of a countable set need not be itself countable.}
				\end{enumerate}
			\end{lemma}
			\begin{proof}
				\begin{itemize}
					\item
						Take $q, r \in \QQ$ with $q < r$. We first find $x \in \intoo[S]{q}{r}$ in the special case $q > 0$.
						
						Find some $k \in \NN$ with $r < k$. We have $z+1 \in \intoo[X]{0}{1}$, so by Lemma~\ref{Lemma: density_of_rationals_in_an_archimedean_prestreak} there is some $s \in \QQ$ with $z+1 < s < 1$. Let $n \in \NN$ be large enough so that $s^n < \inf\left\{q, \frac{r-q}{2}\right\}$. Use the archimedean property to find $m \in \NN_{> 0}$ with $k < m \cdot (z+1)^n$. Construct a finite binary sequence $b\colon \NN_{\leq m} \to \{0, 1\}$ by choosing for each $i \in \NN_{\leq m}$ a true disjunct in $i \cdot (z+1)^n < q \lor q < (i+1) \cdot (z+1)^n$ and set $b_i \dfeq 0$ if the first disjunct is chosen, and $b_i \dfeq 1$ if the second is.
						
						We have $(z+1)^n < s^n < q$, so necessarily $b_0 = 0$. Likewise $q < r < k < m \cdot (z+1)^n$, so necessarily $b_m = 1$.
						
						Declare $j$ to be the first index, for which $b_j = 1$; then $q < (j+1) \cdot (z+1)^n$. Since $b_{j-1} = 0$, we have $(j-1) \cdot (z+1)^n < q$, and since also $2 \cdot (z+1)^n < 2 s^n < r-q$, we conclude $(j+1) \cdot (z+1)^n < r$. Thus $x \dfeq (j+1) \cdot (z+1)^n$ satisfies the required conditions.
						
						Consider now the general case (we no longer assume $q > 0$). Write $q = \frac{i-j}{k}$ where $i, j \in \NN$, $k \in \NN_{> 0}$. Since $z < 0$ (and therefore also $k \cdot z < 0$), there exists $n \in \NN$ with $j + k + n k \cdot z < i$, meaning $n \cdot z < q$. Let $t \in \NN_{> 0}$ be large enough so that $\frac{1}{t} < \frac{r-q}{2}$. Use Lemma~\ref{Lemma: archimedean_prestreak_is_a_union_of_rational_intervals} to find $u \in \ZZ$, so that $n \cdot z \in \intoo[X]{\frac{u-1}{t}}{\frac{u+1}{t}}$.
						
						We claim $0 < q - \frac{u-1}{t} < r - \frac{u+1}{t}$. The is because of
						\[\tfrac{u-1}{t} < n \cdot z < q, \qquad r - q > \tfrac{2}{t}.\]
						By the above there exists $x \in S$ with $q - \frac{u-1}{t} < x < r - \frac{u+1}{t}$. Then $x + n \cdot z \in \intoo[S]{q}{r}$.
					\item
						Obviously $0, 1 \in \im(f)$ and $\im(f)$ is closed under addition and multiplication, so $\im(f)$ is a multiplicative substreak of $X$. Also clearly any multiplicative substreak containing $z$ must contain all polynomials in $z$ with coefficients in $\NN$, so $\im(f)$ is indeed the smallest one.
						
						$M = \im(f)$ is countable since it is enumerated by the countable set $\finseq{\NN}$. Since it contains $S$, it must be dense itself.
				\end{itemize}
			\end{proof}
			
			\begin{theorem}[Characterization of dense streaks]\label{Theorem: characterization_of_dense_streaks}
				The following statements are equivalent for a streak $X$.
				\begin{enumerate}
					\item
						$X$ is dense.
					\item
						There exists an element in $X_{< 0}$ and $X$ has the \df{interpolation property}
						\[\all{a, b}{X}{a < b \implies \xsome{x}{X}{a < x < b}}.\]
					\item
						There exists an element in $X_{< 0}$ and an element in $\intoo[X]{0}{1}$.
					\item
						There exists an element in $\intoo[X]{-1}{0}$.
				\end{enumerate}
			\end{theorem}
			\begin{proof}
				\begin{itemize}
					\item\proven{$(1 \impl 2)$}
						Since $X$ is dense, we have an element in $\intoo[X]{-1}{0} \subseteq X_{< 0}$. As for the interpolation property, take $a, b \in X$, $a < b$. By Lemma~\ref{Lemma: density_of_rationals_in_an_archimedean_prestreak} there exist $q, r \in \QQ$ with $a < q < r < b$. Use density of $X$ again to find $x \in \intoo[X]{q}{r}$. Then $x \in \intoo[X]{a}{b}$.
					\item\proven{$(2 \impl 3)$}
						Obvious.
					\item\proven{$(3 \impl 4)$}
						Let $a \in X_{< 0}$ and $b \in \intoo[X]{0}{1}$ be the assumed elements. Use Lemma~\ref{Lemma: density_of_rationals_in_an_archimedean_prestreak} to find $q, r \in \QQ$ such that $0 < q < b < r < 1$ and let $n \in \NN_{> 0}$ be large enough so that $\frac{1}{n} < \inf\{q, 1-r\}$. By Lemma~\ref{Lemma: archimedean_prestreak_is_a_union_of_rational_intervals} there exists $i \in \ZZ$ (necessarily $i \leq 0$, since $a$ is negative) such that $a \in \intoo[X]{\frac{i-1}{n}}{\frac{i+1}{n}}$. If $i = 0$, then $a \in \intoo[X]{-1}{0}$. If $i$ is not divisable by $n$, then $a$ plus the floor of $\frac{-i}{n}$ is in $\intoo[X]{-1}{0}$. If $i$ is negative and divisable by $n$, then $a + \frac{-i-n}{n} + b \in \intoo[X]{-1}{0}$.
					\item\proven{$(4 \impl 1)$}
						We may use Lemma~\ref{Lemma: countable_dense_substreak} to produce a dense substreak of $X$. Hence $X$ itself is dense.
				\end{itemize}
			\end{proof}
			
			Generally we could substitute $\QQ$ for any dense streak in the various theorems we had up to this point. Here is just a taste.
			
			\begin{proposition}
				Let $X$ be a dense streak.
				\begin{enumerate}
					\item
						$X$ has the interpolation property with regard to any streaks $Y$, $Z$ in the following sense: $\xall{a}{Y}\all{b}{Z}{a < b \implies \xsome{x}{X}{a < x < b}}$.
					\item
						A streak $Y$ is dense if and only if it has the interpolation property with regard to $X$: $\all{a, b}{X}{a < b \implies \xsome{y}{Y}{a < y < b}}$.
				\end{enumerate}
			\end{proposition}
			\begin{proof}
				\begin{enumerate}
					\item
						By definition $a < b$ means there exists $q \in \QQ$ with $a < q < b$. Use Lemma~\ref{Lemma: density_of_rationals_in_an_archimedean_prestreak} to find $r \in \QQ$ with $q < r < b$. By density of $X$ we have $x \in X$ between $q$, $r$ and therefore also between $a$, $b$.
					\item
						The implication $(\impl)$ is a special case of the previous item. Conversely, take $q, r \in \QQ$, $q < r$. By density of $X$ we may find $a \in \intoo[X]{q}{\frac{q+r}{2}}$ and $b \in \intoo[X]{\frac{q+r}{2}}{r}$. By assumption there exists $y \in \intoo[Y]{a}{b}$ for which it then holds $q < y < r$.
				\end{enumerate}
			\end{proof}

	\section{Reflective structures}\label{Section: reflections}
	
		In this section we consider how to add additional structure to (pre)streaks. After all, we want to define and study reals with the help of streaks, but the reals have way more structure than a general streak, being a field and a lattice, among other things.
		
		We want the addition of new structures to satisfy the following criteria.
		
		1) \textbf{Addition of new structure to (pre)streaks is canonical.}
		
		Roughly speaking, this means that for every (pre)streak $X$ we construct a new (pre)streak $X'$ such that $X'$ has the wanted additional structure and is either the smallest such (pre)streak containing $X$, or the largest such which is contained in $X$.
		
		Formally, this is captured by the categorical notion of a \df{reflection}, or its dual \df{coreflection}. We recall the definitions.
		\begin{definition}
			Let $\kat{C}$ be a category and $\kat{D}$ its full subcategory.
			\begin{itemize}
				\item
					Suppose that for every object $X$ in $\kat{C}$ we are given an object $R(X)$ in $\kat{D}$ and an arrow $\eta_X\colon X \to R(X)$ in $\kat{C}$ such that for every object $Y$ in $\kat{D}$ and every arrow $f\colon X \to Y$ in $\kat{C}$ there exists a unique arrow $\overline{f}\colon R(X) \to Y$ in $\kat{D}$ such that $\overline{f} \circ \eta_X = f$.
					\[\xymatrix@+1em{
						X \ar[r]^{\eta_X} \ar[rd]_f  &  R(X) \ar@{-->}[d]_{\exists{!}}^{\overline{f}}  \\
						&  Y
					}\]
					Then we say that $R(X)$ is the \df{reflection} of $X$ in $\kat{D}$, $\eta_X$ is the \df{unit} of the reflection, and $\kat{D}$ is a \df{reflective subcategory} of $\kat{C}$. Furthermore, we can extend $R$ to a functor $R\colon \kat{C} \to \kat{D}$: for $f\colon X \to Y$ in $\kat{C}$ we define $R(f)\colon R(X) \to R(Y)$ by $R(f) \dfeq \overline{\eta_Y \circ f}$. We call this functor $R$ the \df{reflector}.
				\item
					Dually, suppose that for every object $X$ in $\kat{C}$ we are given an object $R(X)$ in $\kat{D}$ and an arrow $\epsilon_X\colon R(X) \to X$ in $\kat{C}$ such that for every object $Y$ in $\kat{D}$ and every arrow $f\colon Y \to X$ in $\kat{C}$ there exists a unique arrow $\overline{f}\colon Y \to R(X)$ in $\kat{D}$ such that $\epsilon_X \circ \overline{f} = f$.
					\[\xymatrix@+1em{
						R(X) \ar[r]^{\epsilon_X}  &  X  \\
						Y \ar[ru]_f \ar@{-->}[u]^{\exists{!}}_{\overline{f}}  &
					}\]
					Then we say that $R(X)$ is the \df{coreflection} of $X$ in $\kat{D}$, $\epsilon_X$ is the \df{counit} of the coreflection, and $\kat{D}$ is a \df{coreflective subcategory} of $\kat{C}$. Furthermore, we can extend $R$ to a functor $R\colon \kat{C} \to \kat{D}$: for $f\colon Y \to X$ in $\kat{C}$ we define $R(f)\colon R(Y) \to R(X)$ by $R(f) \dfeq \overline{f \circ \epsilon_Y}$. This $R$ is called a \df{coreflector}.
			\end{itemize}
		\end{definition}
		
		In categorical language, the full subcategory $\kat{D}$ is reflective in $\kat{C}$ when the functor $R$ is left adjoint to the inclusion functor $\kat{D} \hookrightarrow \kat{C}$, with $\eta$ the unit of this adjunction (and dually, $\kat{D}$ is coreflective in $\kat{C}$ when $R$ is right adjoint to $\kat{D} \hookrightarrow \kat{C}$, and $\epsilon$ the counit of the adjunction).
		
		The condition simplifies for streaks.
		\begin{lemma}\label{Lemma: reflection_in_a_preorder_category}
			Let $\kat{C}$ be a preorder category such as $\Str$ (recall Corollary~\ref{Corollary: streaks_preorder_category}), let $\kat{D}$ be a full subcategory of $\kat{C}$ and let $R$ be a mapping from the objects of $\kat{C}$ to the objects of $\kat{D}$. Suppose the following holds:
			\begin{itemize}
				\item
					for every object $X$ in $\kat{C}$ there exists a morphism $X \to R(X)$,
				\item
					for every object $X$ in $\kat{C}$ and $Y$ in $\kat{D}$, if there exists a morphism $X \to Y$, then there exists a morphism $R(X) \to Y$.
			\end{itemize}
			Then $\kat{D}$ is reflective in $\kat{C}$ and $R$ a reflector. (An analogous statement dually holds for coreflections.)
		\end{lemma}
		\begin{proof}
			By assumption $\kat{C}$ is a preorder category, so let $\eta_X$ denote the unique given arrow $X \to R(X)$. Take any $f\colon X \to Y$. By assumption there exists a morphism $R(X) \to Y$ --- again unique, because $\kat{C}$ is a preorder category --- that we denote by $\overline{f}$. Finally, since there can be at most one morphism $X \to Y$, we have $f = \overline{f} \circ \eta_X$.
		\end{proof}
		
		2) \textbf{Addition of new structure to (pre)streaks is modular.}
		
		By this we mean that the addition of a new structure should not spoil any structure we added before; we can add any selection of structures we want. Formally, this means that methods of adding structures should commute. For example, we should (up to isomorphism) obtain the same result whether we first added the ring and then the lattice structure, or vice versa, and in both cases we should end up with a smallest superset which is both a ring and a lattice.
		
		\begin{lemma}\label{Lemma: commutativity_of_(co)reflections}
			Let $R'\colon \kat{C} \to \kat{D}'$, $R''\colon \kat{C} \to \kat{D}''$ be reflections, for which the restrictions $\rstr{R'}_{\kat{D}''}^{\kat{D}''}$ and $\rstr{R''}_{\kat{D}'}^{\kat{D}'}$ exist. Then they commute in the sense
			\[\rstr{R''}_{\kat{D}'}^{\kat{D}' \cap \kat{D}''} \circ R' \ism \rstr{R'}_{\kat{D}''}^{\kat{D}' \cap \kat{D}''} \circ R'',\]
			and these composita determine a reflection $\kat{C} \to \kat{D}' \cap \kat{D}''$. (Likewise for coreflections.)
		\end{lemma}
		\begin{proof}
			It is evident that if $\rstr{R'}_{\kat{D}''}^{\kat{D}''}$ and $\rstr{R''}_{\kat{D}'}^{\kat{D}'}$ exist, then so do $\rstr{R''}_{\kat{D}'}^{\kat{D}' \cap \kat{D}''}$, $\rstr{R'}_{\kat{D}''}^{\kat{D}' \cap \kat{D}''}$, and these are again reflections. Furthermore, composition of reflections is a reflection, and any two reflections onto the same full subcategory are isomorphic (being the left adjoints to the same inclusion functor).
		\end{proof}
		
		This lemma means that to show that two reflections commute, we need to verify that imposition of new structure by one of the reflections preserves the structure, garanteed by the other. In the example above, if we are making a ring out of a lattice streak, the result will again be a lattice, and vice versa. The lemma then garantees that both ways of forming a lattice ring streak are isomorphic.
		
		3) \textbf{Addition of new structure behaves well with regard to the universal properties, used in definitions.}
		
		As mentioned, the reals will be given as the terminal streak. We have already seen, that natural numbers are the initial streak, and in this section we'll characterise other number sets via the universal property as well. Here is the relevant lemma.
		
		\begin{lemma}\label{Lemma: interaction_of_reflections_with_initiality_and_terminality}
			Let $\kat{D}$ be a (full) reflective subcategory of $\kat{C}$, witnessed by $R$ and $\eta$.
			\begin{enumerate}
				\item\label{Lemma: interaction_of_reflections_with_initiality_and_terminality: initiality}
					If $\zero$ is an initial object of $\kat{C}$, then $R(\zero)$ is an initial object of $\kat{D}$.
				\item\label{Lemma: interaction_of_reflections_with_initiality_and_terminality: terminality}
					If $\one$ is a terminal object of $\kat{C}$, then $R(\one)$ is a terminal object of both $\kat{C}$ and $\kat{D}$; in particular $R(\one) \ism \one$.
			\end{enumerate}
		\end{lemma}
		\begin{proof}
			\begin{enumerate}
				\item
					Proof for categorists: left adjoints preserve colimits.
					
					For everyone else, take any object $X$ in $\kat{D}$. Then $X$ is an object also in $\kat{C}$, so there exists a morphism $\zero \to X$. By the definition of reflection there is a morphism $R(\zero) \to X$.
					
					To prove uniqueness, take any two morphisms $f, g\colon R(\zero) \to X$. They both make the diagram
					\[\xymatrix@+1em{
						\zero \ar[r]^{\eta_\zero} \ar[rd]_{\ini[X]}  &  R(\zero) \ar@<-0.5ex>[d]_f \ar@<0.5ex>[d]^g  \\
						&  X
					}\]
					commute since there is only one morphism $\zero \to X$. By the definition of reflection we have $f = g$.
				\item
					Consider the maps $\xymatrix{\one \ar@/^1ex/[r]^{\eta_\one} & R(\one) \ar@/^1ex/[l]^{\trm[R(\one)]}}$. Since there exists only one map $\one \to \one$, we have $\trm[R(\one)] \circ \eta_\one = \id[\one]$. This also implies that the diagram
					\[\xymatrix@+2em{
						\one \ar[r]^{\eta_\one} \ar[rd]_{\eta_\one}  &  R(\one) \ar@<-0.5ex>[d]_{\id[R(\one)]} \ar@<0.5ex>[d]^{\eta_\one \circ \trm[R(\one)]}  \\
						&  R(\one)
					}\]
					commutes, so by the definition of reflection we have $\eta_\one \circ \trm[R(\one)] = \id[R(\one)]$. We conclude $\one \ism R(\one)$.
					
					An object, isomorphic to a terminal one, is terminal itself, and since $R(\one)$ lies in the full subcategory $\kat{D}$, it is terminal there as well.
			\end{enumerate}
		\end{proof}
		
		We now study the concrete examples of (co)reflections, relevant for us.
		
		\subsection{Positive part}
		
			As a warmup exercise, we consider the coreflection of taking the positive part of a (pre)streak (together with $0$). The idea is that we can turn every prestreak into one with total multiplication by just restricting the prestreak to the multiplication domain.
			
			Specifically, we define a functor $\pos{}\colon \Pstr \to \Pstr$ by $\pos{X} \dfeq X_{> 0} \cup \{0\}$ and $\pos{f}(x) \dfeq f(x)$. Clearly if $X$ is a prestreak, then so is $\pos{X}$, and for any morphism $f$ we have $f(0) = 0$ and $0 < x \implies 0 < f(x)$, so this functor is well defined.
			
			Clearly $\pos{}$ restricts to a functor on streaks $\pos{}\colon \Str \to \Str$.
			
			Let $f\colon Y \to X$ be a morphism between (pre)streaks where $Y = \pos{Y}$. Then we can restrict $f$ to $\rstr{f}^{\pos{X}}$ since, as mentioned, morphisms preserve $0$ and $<$. Thus $\pos{}$ (more precisely its corestriction to (pre)streaks $X = \pos{X}$) is a coreflection on the category of (pre)streaks.
			
			\begin{proposition}\label{Proposition: positive_part_is_multiplicative}
				For any (archimedean pre)streak $X$ the (pre)streak $\pos{X}$ is multiplicative.
			\end{proposition}
			\begin{proof}
				The total multiplication in $\pos{X}$ is of course given as an extension of the one on $X_{> 0}$ by declaring that $0$ times anything (and anything times $0$) is $0$. Clearly this multiplication is commutative, associative, distributes over $+$ and has $1$ as the unit.
				
				As for the multiplicity condition, take any $a, b, c, d \in \pos{X}$ with $b < a$, $d < c$. Clearly then $a$ and $c$ must be positive. The fact that $X_{> 0}$ is closed under multiplication by definition deals with the case $b = d = 0$. If precisely one of $b$, $d$ is $0$, the multiplicity condition amounts to the standard connection between $<$ and $\cdot$ in a prestreak. If all $a$, $b$, $c$, $d$ are positive, use Proposition~\ref{Proposition: multiplicity_condition_for_positive_elements}.
			\end{proof}
		
		\subsection{Archimedean property}\label{Subsection: archimedean_coreflection}
		
			Before we start adding additional structure to streaks, we want to have a way to transform an arbitrary prestreak to a streak. This means we need to impose the archimedean property and antisymmetry of $\leq$. We deal with the first in this subsection, and with the second in the next.
			
			Given a prestreak $(X, <, +, 0, \cdot, 1)$ we define a new prestreak $(\arch(X), <', +, 0, \cdot, 1)$ by $\arch(X) \dfeq \st{a \in X}{\some{n}{\NN}{a < n \land 0 < a + n}}$ where for $a, b \in \arch(X)$ we define
			\[a <' b \dfeq \xsome{n}{\NN}{n \cdot a + 1 < n \cdot b},\]
			and the new algebraic operations are the same as (more precisely, the restrictions to $\arch(X)$ of) the old ones.
			
			We claim that $\arch(X)$ is a prestreak. Since $<'$ is given by an open condition ($<$ is open and $\NN$ is overt), it is an open relation. We certainly have $0 \in \arch(X)$ (take $n = 1$) and $1 \in \arch(X)$ (take $n = 2$). Fix $a, b \in X$ and $m, n \in \NN$ such that $a < m$, $0 < a + m$, $b < n$, $0 < b + n$. Then $a + b < m + n$ and $0 < a + b + m + n$, so $a + b \in \arch(X)$. Assume now additionally that $a, b >' 0$, \ie there are $j, k \in \NN$ such that $1 < j \cdot a$ and $1 < k \cdot b$. Then $1 < j \cdot k \cdot a \cdot b$, so $a \cdot b >' 0$. Also, we have $a \cdot b < m \cdot n \leq 3 m n$ and $0 < (a + m) \cdot (b + n) = a \cdot b + n \cdot a + m \cdot b + m \cdot n \leq a \cdot b + 3 m n$. The other prestreak conditions are immediate.
			
			More to the point, $\arch(X)$ is an archimedean prestreak. To prove this, take any $a, b, c, d \in X$ whose presence in $\arch(X)$ is witnessed by $i, j, k, l \in \NN$, and let $b <' d$, witnessed by $m \in \NN$. Define $n \dfeq 1$ and $N \dfeq (i + k + 1) \cdot m$; we claim $n \cdot (a + N \cdot b) + 1 < n \cdot (c + N \cdot d)$, and so $a + N \cdot b <' c + N \cdot d$.
			\[a + N \cdot b + 1 = a + (i + k + 1) m \cdot b + 1 < i + c + k + (i + k + 1) m \cdot b + 1 =\]
			\[= c + (i + k + 1) \cdot (m \cdot b + 1) \leq c + N \cdot d\]
			
			Observe that if $f\colon X \to Y$ is a morphism between prestreaks, then its restriction $\arch(f)\colon \arch(X) \to \arch(Y)$, $\arch(f)(x) \dfeq f(x)$, is well defined since $f$ preserves all structure. Thus we've defined a functor $\arch\colon \Pstr \to \Apstr$.
			
			We claim that $\arch$ is a coreflector, with the counit of the coreflection being the inclusion $\alpha_X\colon \arch(X) \hookrightarrow X$. This is indeed a morphism: if $a, b \in X$ and $a <' b$, witnessed by $n \in \NN$ (necessarily $n > 0$), it follows $n \cdot a < n \cdot a + 1 < n \cdot b$, so $a < b$.
			
			Let $Y$ be an archimedean prestreak and $f\colon Y \to X$ a morphism. Clearly the image of $f$ is contained in $\arch(X)$ since $Y$ is archimedean and $f$ preserves the prestreak structure. Take $a, b \in Y$ such that $a < b$. By the archimedean property of $Y$ there exists $n \in \NN$ such that $1 + n \cdot a < 0 + n \cdot b$, therefore $1 + n \cdot f(a) < n \cdot f(d)$, and so $f(a) <' f(b)$.
			
			For any prestreak $X$ we have
			\[\arch(\pos{X}) = \pos{\arch(X)} = \{0\} \cup \st{a \in X_{> 0}}{\xsome{n}{\NN}{a < n}},\]
			and the strict order relation is in both orders of composition given as $a <' b \iff \xsome{n}{\NN}{n \cdot a + 1 < n \cdot b}$. In short, the functors $\pos{}$ and $\arch$ commute.
		
		\subsection{Partial order}
		
			In this subsection we impose the second streak condition onto prestreaks, namely the antisymmetry of $\leq$, or equivalently, tightness of $\apart$. Together with the result from the previous subsection, this enables us to canonically turn any prestreak into a streak. Unlike $\arch$ (and $\pos{}$) thus far which were coreflections, imposing antisymmetry (and all the further structures that we mention) is a reflection.
			
			The way to do it is the completly standard way to turn a preorder into a partial order, or an apartness relation into a tight one. Recall that we already defined for any prestreak (or even a mere strict order) $X$ for $a, b \in X$ to be equivalent, $a \nap b$, when $a \leq b \land b \leq a$, or equivalently, $\lnot(a \apart b)$, holds. Denote $Q(X) \dfeq X/_\nap$ and let $\theta_X\colon X \to Q(X)$ be the quotient map.
			
			It follows from the prestreak axioms that operations commute with the order structure, so they induce corresponding operations on the quotient $Q(X)$. Explicitly, we define for $a, b \in X$
			\[[a] < [b] \dfeq a < b, \qquad [a] + [b] \dfeq [a + b]\]
			whence it follows that $[0]$ is the zero element in $Q(X)$. For $a, b \in X_{> 0}$ we furthermore define
			\[[x] \cdot [y] \dfeq [x \cdot y],\]
			so the multiplicative unit in $Q(X)$ is $[1]$.
			
			Finally, recalling from Section~\ref{Section: setting} that quotients have the quotient topology, it is clear that the relation $<$ is open and $\leq$ closed in $Q(X)$.
			
			Since $\nap$ is equality on $Q(X)$ by definition, we conclude that $Q(X)$ is a tight prestreak.
			
			Clearly, if $Y$ is a prestreak with tight $<$, then a morphism $f\colon X \to Y$ factors as $f = \overline{f} \circ \theta_X$ where $\overline{f}([x]) \dfeq f(x)$ (the point is, this map is well defined since any morphism preserves $\nap$ which on $Y$ is simply the equality). As such, $Q$ is a reflector (with $\theta$ the unit of the reflection) of prestreaks into tight prestreaks.
			
			One easily verifies that $Q$ commutes with $\arch$, so we have a canonical way of turning a prestreak into a streak (take either of the compositions $Q \circ \arch$, $\arch \circ Q$). However, this canonical way is neither a reflection nor a coreflection, but rather a composition of both.
			
			It is also obvious that $Q$ commutes with $\pos{}$.
		
		\subsection{Lattices}
		
			Now that we've exhibited the way of turning prestreaks into streaks, we'll focus on the latter. The reason is that prestreaks are not particularly amenable to adding additional structure; for example, in this subsection we want to add the lattice structure, but suprema and infima are not uniquely defined unless $\leq$ is antisymmetric. The usefulness of prestreaks is that they are stepping stones toward streaks: typically our construction will entail first a construction of a prestreak (even if we started with a streak), and the desired reflected object will be its quotient (as in the previous subsection).
			
			We'll break the lattice structure into two parts --- meet- and join-semilattices. We start with the former.
			
			The idea is to represent an infimum of a finite set by that set itself. We restrict to inhabited finite sets, as we need only infima of those for a (semi)lattice, and moreover the empty set would represent $\infty$ which would later spoil the archimedean condition.
			
			To this end we denote the set of inhabited finite subsets of a set $X$ by $\inhfin(X)$. This set always exists under our assumptions; we can represent it for example as a quotient of $FS(X) \dfeq \st{a \in \finseq{X}}{\lnth(a) > 0}$ (recall that $\finseq{X}$ is the set of finite sequences of elements in $X$ and $\lnth(a)$ is the length of the sequence $a$).
			
			We claim that if $X$ has a prestreak structure, then so does $\inhfin(X)$ in the following way. Let $A, B \in \inhfin(X)$.
			\[A < B \dfeq \xsome{a}{A}\xall{b}{B}{a < b}  \qquad  A + B \dfeq \st{a + b}{a \in A \land b \in B}\]
			Clearly $+$ makes $\inhfin(X)$ into a commutative monoid, with $\{0\}$ as the additive unit. We have $\{0\} < A \iff \xall{a}{A}{0 < a}$, so we can define multiplication on $\inhfin(X)_{> \{0\}}$ by
			\[A \cdot B \dfeq \st{a \cdot b}{a \in A \land b \in B}.\]
			Again it is clear that this makes $\inhfin(X)_{> \{0\}}$ into a monoid, with $\{1\}$ as the multiplicative unit.
			
			Before we prove the other prestreak conditions, we note that the definition of $<$ on $\inhfin(X)$ is equivalent to the ostensibly weaker version where we swap the quantifiers.
			\begin{lemma}\label{Lemma: alt_less_than_on_finite_subsets}
				For all $A, B \in \inhfin(X)$ we have
				\[A < B \iff \xall{b}{B}\xsome{a}{A}{a < b}.\]
			\end{lemma}
			\begin{proof}
				Clearly the left side implies the right one. We prove the converse.
				
				Represent the two inhabited finite subsets as $A = \{a_0, \ldots, a_{m-1}\}$, $B = \{b_0, \ldots, b_{n-1}\}$ where $m, n \in \NN_{> 0}$. Let $s\colon \NN_{< n} \to \NN_{< m}$ be a finite sequence, such that $a_{s(j)} < b_j$ for all $j \in \NN_{< n}$. Construct a binary matrix $M = \big[c_{i,j}\big]_{(i, j) \in \NN_{< m} \times \NN_{< n}} \in \{0, 1\}^{m \times n}$ in the following way. For each $(i, j) \in \NN_{< m} \times \NN_{< n}$ choose a true disjunct in $a_{s(j)} < a_i \lor a_i < b_j$. Set $c_{k,j} \dfeq 0$ if the first disjunct was chosen, and $c_{k,j} \dfeq 1$ if the second one was (clearly we have $c_{s(j),j} = 1$ for all $j \in \NN_{< n}$).
				
				Suppose that every row of $M$ contained a zero. Let $z\colon \NN_{< m} \to \NN_{< n}$ be a finite sequence such that $z(i)$ is the smallest index with $c_{i,z(i)} = 0$, and let the infinite sequence $t\colon \NN \to \NN_{< m}$ be inductively defined by $t(0) \dfeq 0$ and $t(k+1) \dfeq s(z(t(k)))$. Then $t$ must be injective since $a \circ t$ is strictly decreasing, but there is no injective map $\NN \to \NN_{< m}$ --- a contradiction. Thus there exists a row of $M$ which contains only ones, proving $A < B$.
			\end{proof}
			
			\begin{theorem}
				$\inhfin(X)$ is a prestreak.
			\end{theorem}
			\begin{proof}
				It remains to verify all the prestreak conditions involving $<$. Take any $A, B, X \in \inhfin(X)$ and represent them as $A = \{a_0, \ldots, a_{m-1}\}$, $B = \{b_0, \ldots, b_{n-1}\}$, $X = \{x_0, \ldots, x_{p-1}\}$ where $m, n, p \in \NN_{> 0}$.
				\begin{itemize}
					\item\proven{asymmetry of $<$}
						Assume $A < B$ and $B < A$. Then there is some $a \in A$, smaller than all the elements in $B$, and some $b \in B$, smaller than all the elements in $A$ which means $a < b$ and $b < a$, a contradiction.
					\item\proven{cotransitivity of $<$}
						Suppose $A < B$, and let $a \in A$ be an element, smaller than all the elements in $B$. Construct a binary matrix $M = \big[m_{k,j}\big]_{(k, j) \in \NN_{< p} \times \NN_{< n}} \in \{0, 1\}^{p \times n}$ in the following way. For each $(k, j) \in \NN_{< p} \times \NN_{< n}$ choose a true disjunct in $a < x_k \lor x_k < b_j$. Set $m_{k,j} \dfeq 0$ if the first disjunct was chosen, and $m_{k,j} \dfeq 1$ if the second one was. If each row of $M$ contains a zero, then $A < X$. Otherwise there is a row which contains only ones, and then $X < B$.
					\item\proven{$<$ is open, $\leq$ is closed}
						Since `finite' implies `countable', the formula for $<$ on $\inhfin(X)$ is clearly an open predicate. For $\leq$ we have (using properties of intuitionistic logic)
						\[A \leq B \iff \lnot\xsome{b}{B}\xall{a}{A}{b < a} \iff \xall{b}{B}{\lnot\xall{a}{A}{b < a}} \iff\]
						\[\iff \xall{b}{B}{\lnot\xall{a}{A}{\lnot\lnot(b < a)}} \iff \xall{b}{B}{\lnot\xall{a}{A}{\lnot{a \leq b}}} \iff\]
						\[\iff \xall{b}{B}{\lnot\lnot\xsome{a}{A}{a \leq b}}.\]
						We obtained a closed predicate.
				\end{itemize}
				Hereafter we use the alternative definition of $<$ on $\inhfin(X)$ from Lemma~\ref{Lemma: alt_less_than_on_finite_subsets}.
				\begin{itemize}
					\item\proven{addition preserves $<$}
						Assume $A < B$. Take any $b + x \in B + X$. By assumption there exists $a \in A$ such that $a < b$. Then $a + x < b + x$, so $A + X < B + X$.
					\item\proven{addition reflects $<$}
						Assume $A + X < B + X$ and take any $b \in B$. Define sequences $s\colon \NN_{< p} \to \NN_{< m}$, $t\colon \NN_{< p} \to \NN_{< p}$ such that $a_{s(k)} + x_{t(k)} < b + x_k$ for each $k \in \NN_{< p}$. Choose a true disjunct in $a_{s(k)} + x_{t(k)} < a_{s(k)} + x_k \lor a_{s(k)} + x_k < b + x_k$. If the second one holds, we are done. Otherwise repeat this procedure with $t(k)$ instead of $k$. Eventually the second disjunct will be chosen since there are only finitely many elements in $X$ and up to that point we have $x_{t(k)} < x_k$.
					\item\proven{multiplication preserves and reflects $<$}
						Proven in exactly the same way as for addition.
				\end{itemize}
			\end{proof}
			
			Note that $\inhfin$ can be made into a functor by defining for a morphism $f\colon X \to Y$
			\[\inhfin(f)\colon \inhfin(X) \to \inhfin(Y),  \qquad  \inhfin(f)(A) \dfeq f(A).\]
			It can be easily seen that $\inhfin(f)$ is again a morphism.
			
			Also, a prestreak $X$ can be embedded into $\inhfin(X)$ via a map $\tau_X\colon X \to \inhfin(X)$, $\tau_X(a) \dfeq \{a\}$. Again, it is evident that $\tau_X$ is a morphism.
			
			Observe that for $n \in \NN$ and $A \in \inhfin(X)$ we have $n \cdot A \nap \st{n \cdot a}{a \in A}$. This makes it easier to prove the archimedean property.
			\begin{proposition}\label{Proposition: lattice_structure_preserves_archimedean_property}
				If $X$ is an archimedean prestreak, then so is $\inhfin(X)$.
			\end{proposition}
			\begin{proof}
				Take any $A, B, C, D \in \inhfin(X)$ such that $B < D$, \ie we have $b \in B$, smaller than all elements in $D$. Fix also some $a \in A$.
				
				By the archimedean property of $X$ we can find for any $c \in C$, $d \in D$ some $n \in \NN$ such that $a + n \cdot b < c + n \cdot d$. Let $m \in \NN$ be an upper bound for all of these finitely many choices of $n$. Then we have $A + m \cdot B < C + m \cdot D$.
			\end{proof}
			
			However, even if the prestreak $X$ is tight, $\inhfin(X)$ is not (intuitively, different sets can have the same infimum). If we want to get a streak that way, we need to compose it with $Q$. That is, for a streak $X$ we define $X^\land \dfeq Q(\inhfin(X))$. By the previous proposition $\inhfin(X)$ is archimedean, so it follows from the results in the previous section that $X^\land$ is again a streak.
			
			More to the point, it is a meet-semilattice; the infimum is given simply by $\inf\{[A], [B]\} = [A \cup B]$.
			
			Let $Y$ be any meet-semilattice streak and $f\colon X \to Y$ a morphism. Then one can define a morphism $\overline{f}\colon X^\land \to Y$ by $\overline{f}([A]) \dfeq \bigwedge_{a \in A} f(a)$. This map is well-defined: for $A, B \in \inhfin(X)$, if $\bigwedge_{a \in A} f(a) < \bigwedge_{b \in B} f(b)$, then by the definition of infima (and Proposition~\ref{Proposition: finite_strict_suprema_and_infima}) there exists $a \in A$ which is smaller than all $b \in B$, so $A < B$. Thus $\bigwedge_{a \in A} f(a) \apart \bigwedge_{b \in B} f(b)$ implies $[A] \apart [B]$, the contrapositive of which means $f$ is well defined.
			
			It is easily seen that it is also a morphism (which moreover preserves finite infima). Obviously $\overline{f} \circ \iota_X = f$ where $\iota_X \dfeq \theta_X \circ \tau_X$. We conclude that ${}^\land$ is a reflector (with $\iota$ the unit of the reflection) from streaks to meet-semilattice streaks.
			
			Clearly ${}^\land$ commutes with $\pos{}$ since an infimum of an inhabited finite set is positive if and only if all its elements are. Due to postcomposition with $Q$ it is also clear that ${}^\land$ commutes with $Q$. From Proposition~\ref{Proposition: lattice_structure_preserves_archimedean_property} it also quickly follows that ${}^\land$ commutes with $\arch$.
			
			We have seen how to adjoin finite infima to a streak; now we deal with suprema. The idea is largely the same, but there is a technical complication. If an inhabited finite set is to represent the supremum of its elements, we need to define $<$ as $A < B \iff \xall{a}{A}\xsome{b}{B}{a < b}$ (or equivalently $ \xsome{b}{B}\xall{a}{A}{a < b}$). However, this would mean that $A$ is positive when $\xsome{a}{A}{0 < A}$, so we cannot take for $A \cdot B$ simply all possible products of elements from $A$ and $B$, partially because they might not be defined (in general $\cdot$ is defined only on $\{0\} \cup X_{> 0}$), but even if they are (say, $X$ is multiplicative), we might not get the correct result (the product of $\{-2, 1\}$ with itself should be equivalent to $1$, not to $4$). What we essentially want for $A \cdot B$ is the set of all products from $A_{> 0}$ and $B_{> 0}$, but constructively these subsets of finite sets need not be again finite (they are if $<$ is decidable,
  but in that case adjoining suprema is a pointless exercise anyway since an inhabited finite set already containes its supremum in a decidable linear order).
			
			The consequence is that we cannot in general define multiplication already on $\inhfin(X)$, but with a trick we can still do it on the quotient. In order not to repeat myself with all the other stuff though, we'll use this opportunity to construct $X^\lor$ and its streak structure in a different way, with the alternative definition of a streak from Proposition~\ref{Proposition: alt_streak}.
			
			Let $X$ now be a streak from the start. We equip the set $FS(X) \dfeq \st{a \in \finseq{X}}{\lnth(a) > 0}$ with the two comparison relations with rationals, defined for $a \in FS(X)$ and $q \in \QQ$ by
			\[q < a \dfeq \xsome{i}{\NN_{< \lnth(a)}}{q < a_i},  \qquad  a < q \dfeq \xall{i}{\NN_{< \lnth(a)}}{a_i < q}.\]
			Furthermore, define the equivalence relation $\nap$ for $a, b \in FS(X)$ by
			\[a \nap b \dfeq \all{q}{\QQ}{q < a \iff q < b}\]
			and let $X^\lor \dfeq FS(X)/_\nap$. Clearly the predicates defining $<$ on $FS(X)$ are open, and their negations, given by
			\[a \leq q \iff \lnot(q < a) \iff \xall{i}{\NN_{< \lnth(a)}}{a_i \leq q},\]
			\[q \leq a \iff \lnot(a < q) \iff \lnot\lnot\xsome{i}{\NN_{< \lnth(a)}}{q \leq a_i},\]
			are closed, so this is then the case also on the quotient $X^\lor$.
			
			For $[a], [b] \in X^\lor$ we define
			\[[a] + [b] \dfeq \big[(a_i + b_j)_{(i, j) \in \NN_{< \lnth(a)} \times \NN_{< \lnth(b)}}\big].\]
			Notice that the zero element in $X^\lor$ is $[(0)]$ and that $0 < [a] \iff \xsome{j}{\NN_{< \lnth(a)}}{0 < a_j}$. It follows that any positive element $[a]$ can be represented as $[a']$ where $a'$ has only positive entries. To see this, fix $j \in \NN_{< \lnth(a)}$ such that $0 < a_j$, then for each $i \in \NN_{< \lnth(a)}$ choose a true disjunct in $0 < a_i \lor a_i < a_j$. Let $a'$ be the tuple of all $a_i$s for which the first disjunct was chosen (clearly $a_j$ itself appears in $a'$). Then $[a] = [a']$.
			
			Thus when defining multiplication we may without loss of generality assume that all terms in $a$ and $b$ are positive, and then we define
			\[[a] \cdot [b] = [a'] \cdot [b'] \dfeq \big[(a'_i \cdot b'_j)_{(i, j) \in \NN_{< \lnth(a')} \times \NN_{< \lnth(b')}}\big].\]
			Note that $[(1)]$ is the multiplicative unit.
			
			One can verify the other conditions that $X^\lor$ is a streak similarly as for $X^\land$ above. Moreover, it is a join-semilattice for the supremum $\sup\{[a], [b]\} = [a \cnct b]$ (where $\cnct$ denotes concatenation).
			
			We define $\sigma_X\colon X \to X^\lor$ by $\sigma_X(a) \dfeq [(a)]$. If $Y$ is a join-semilattice streak, we can extend any morphism $f\colon X \to Y$ to $\overline{f}\colon X^\lor \to Y$, $\overline{f}([a]) \dfeq \bigvee_{i \in \NN_{< \lnth(a)}} f(a_i)$. Altogether we conclude that we have a reflection of streaks onto join-semilattice streaks.
			
			If $X$ was a join-semilattice streak from the start, then $X^\land$ still is; the binary supremum is given as
			\[\sup\big\{[A], [B]\big\} = \big[\st{\sup\{a, b\}}{a \in A \land b \in B}\big].\]
			Similarly we see that if $X$ is a meet-semilattice, then $X^\lor$ is. We conclude (by Lemma~\ref{Lemma: commutativity_of_(co)reflections}) that the reflection ${}^\land$ and ${}^\lor$ commute, and their composition determines a reflection from streaks to lattice streaks.
			
			That ${}^\lor$ commutes with other previously mentioned (co)reflections can be checked similarly as for ${}^\land$.
			
			As a conclusion to this subsection we observe, how infima and suprema interact with the algebraic operations.
			\begin{proposition}\label{Proposition: interaction_of_bounds_with_algebra}
				Let $X$ be a streak, $x \in X$ and $A, B \subseteq X$ inhabited finite subsets which have infima in $X$.
				\begin{enumerate}
					\item
						We have
						\[\inf{A} + \inf{B} = \inf\st{a + b}{a \in A \land b \in B}\]
						(in particular, the last infimum exists in $X$). Specifically, $(\inf{A}) + x = \inf\st{a + x}{a \in A}$.
					\item
						If all elements of $A$ and $B$ are positive, then so are $\inf{A}$, $\inf{B}$ and we have
						\[\inf{A} \cdot \inf{B} = \inf\st{a \cdot b}{a \in A \land b \in B}\]
						(in particular, the last infimum exists in $X$). Specifically, if $x > 0$, then $(\inf{A}) \cdot x = \inf\st{a \cdot x}{a \in A}$.
				\end{enumerate}
				The analogous statement holds for suprema.
			\end{proposition}
			\begin{proof}
				One can check these equalities by writing out the definitions of all of these infima. Here is a fancier proof, though, with the help of the results from this subsection.
				
				Since streak morphisms preserve and reflect the order, they also preserve infima. Thus
				\[\iota_X\big(\inf{A} + \inf{B}\big) = \iota_X(\inf{A}) + \iota_X(\inf{B}) = [A] + [B] = [A + B] = \iota_X\big(\inf(A + B)\big).\]
				As a streak morphism, $\iota_X$ is injective, so $\inf{A} + \inf{B} = \inf(A + B)$. The part with $x$ is a special case when $B = \{x\}$.
				
				The same trick works for products, as well as for suprema.
			\end{proof}
		
		\subsection{Rings}\label{Subsection: ring_streaks}
		
			In this subsection we discuss the ring structure of (pre)streaks. First the definition.
			
			\begin{definition}
				A \df{ring prestreak} is a multiplicative prestreak, for which $+$ and $\cdot$ form a ring (and therefore a unital commutative ring) --- meaning that we have \df{subtraction} as another operation.
			\end{definition}
			
			The existence of subtraction is sufficient for a multiplicative prestreak to be a ring prestreak. In the case od streaks we can make a stronger statement. Recall that we have subtraction in any streak, albeit only as a partial operation in general. We claim that being a ring streak is equivalent to this subtraction being total (we need not assume that the streak is multiplicative).
			\begin{theorem}\label{Theorem: streak_with_total_subtraction_is_ring}
				Let $X$ be a streak, in which subtraction is a total operation. Then there exists a unique extension of multiplication to the whole $X$ which makes $X$ into a ring streak.
			\end{theorem}
			\begin{proof}
				We know that subtraction is uniquely defined in a streak. We also already know from Proposition~\ref{Proposition: uniqueness_of_multiplicative_structure_in_streaks} that $X$ has at most one multiplicative structure. We use the idea from the proof of this proposition to define it, then show that it fulfils the criteria for a ring streak.
				
				Given any $a, b \in X$, use the archimedean property to find $m, n \in \NN$ such that $a + m > 0$, $b + n > 0$. Thus the product of $a + m$ and $b + n$ is given by the usual streak product. Any ring streak multiplication is by definition distributive over addition, therefore it must hold
				\[(a + m) \cdot (b + n) = a \cdot b + n \cdot a + m \cdot b + m \cdot n,\]
				and we also know from Lemma~\ref{Lemma: total_multiplication_coincides_with_multiplication_with_natural_numbers_in_multiplicative_streaks} that when multiplying a streak element with a natural number, there is no difference between the total multiplication and the inductive definition. Therefore we define the total multiplication in $X$ by
				\[a \cdot b \dfeq (a + m) \cdot (b + n) - n \cdot a - m \cdot b - m \cdot n.\]
				We claim this is well defined. Let $m', n' \in \NN$ be another natural numbers, for which $a + m' > 0$, $b + n' > 0$ holds. We have
				\[\Big((a + m') \cdot (b + n') - n' \cdot a - m' \cdot b - m' \cdot n'\Big) -\]
				\[- \Big((a + m) \cdot (b + n) - n \cdot a - m \cdot b - m \cdot n\Big) =\]
				\[= (a + m') \cdot (b + n') - (a + m) \cdot (b + n) - (n' - n) \cdot a - (m' - m) \cdot b - m' \cdot n' + m \cdot n.\]
				Suppose first that $m' \geq m$, $n' \geq n$, and let us have $m' = m + \Delta{m}$, $n' = n + \Delta{n}$. Then we can continue the calculation above as
				\[(a + m + \Delta{m}) \cdot (b + n + \Delta{n}) - (a + m) \cdot (b + n) -\]
				\[- \Delta{n} \cdot a - \Delta{m} \cdot b - (m + \Delta{m}) \cdot (n + \Delta{n}) + m \cdot n =\]
				\[= \Delta{n} \cdot (a + m) + \Delta{m} \cdot (b + n) + \Delta{m} \cdot \Delta{n} -\]
				\[- \Delta{n} \cdot a - \Delta{m} \cdot b - \Delta{m} \cdot n - \Delta{n} \cdot m - \Delta{m} \cdot \Delta{n} = 0.\]
				The other combinations for $m, m', n, n'$ can be dealt with in the same way.
				
				In particular, if $a$, $b$ are positive from the start, we can take $m' = n' = 0$ which makes it clear that this multiplication extends the usual streak one on positive elements.
				
				We get $0 \cdot a = 0$ and $1 \cdot a = 1$ as in Lemma~\ref{Lemma: total_multiplication_coincides_with_multiplication_with_natural_numbers_in_multiplicative_streaks}. Commutativity, associativity and distributivity are straightforward to check. Finally, as per discussion right after Definition~\ref{Definition: multiplicative_prestreak}, in the presence of subtraction the multiplicity condition amounts to products of positive elements being positive, something which holds in any (pre)streak.
			\end{proof}
			
			We now observe that any multiplicative (pre)streak can be turned into a ring streak in a canonical way, using the standard idea, how to turn a semigroup into a group: by taking formal differences.
			
			Let $X$ be a multiplicative prestreak. We equip $FD(X) \dfeq X \times X$ with order and operations thusly: for $(a, b), (c, d) \in FD(X)$ (intuitively $(a, b)$ represents $a-b$) let
			\[(a, b) < (c, d) \dfeq a + d < c + b,\]
			\[(a, b) + (c, d) \dfeq (a + c, b + d),  \qquad  (a, b) \cdot (c, d) \dfeq (a \cdot c + b \cdot d, a \cdot d + b \cdot c).\]
			Since addition is defined componentwise, $FD(X)$ is clearly a monoid for it, with $(0, 0)$ as the unit. We have $(0, 0) < (a, b) \iff b < a$, therefore if two elements of $FD(X)$ are positive, so is their product by the definition of $<$ and the multiplicity condition of $X$. More generally, $FD(X)$ also satisfies the multiplicity condition: if $(a, a'), (b, b'), (c, c'), (d, d') \in FD(X)$ satisfy $(b, b') < (a, a')$ and $(d, d') < (c, c')$, then $b + a' < a + b'$ and $d + c' < c + d'$ whence by multiplicity of $X$
			\[(a + b') \cdot (d + c') + (b + a') \cdot (c + d') < (a + b') \cdot (c + d') + (b + a') \cdot (d + c')\]
			which, when calculated, is exactly the required condition
			\[(a, a') \cdot (d, d') + (b, b') \cdot (c, c') < (a, a') \cdot (c, c') + (b, b') \cdot (d, d').\]
			The remaining conditions to conclude, that $FD(X)$ is again a multiplicative prestreak, are easy enough.
			
			We can extend $FD$ to a functor (from the full subcategory of multiplicative prestreaks to itself) by defining for $f\colon X \to Y$ simply $F(f)(a, b) \dfeq (f(a), f(b))$.
			
			Note that $n \cdot (a, b) = (n \cdot a, n \cdot b)$ for $n \in \NN$, $a, b \in X$. It follows that $FD$ preserves the archimedean condition.
			\begin{lemma}
				If $X$ is archimedean, then so is $FD(X)$.
			\end{lemma}
			\begin{proof}
				Take $(a, a'), (b, b'), (c, c'), (d, d') \in FD(X)$ such that $(b, b') < (d, d')$, meaning $b + d' < b' + d$. We have
				\[(a, a') + n \cdot (b, b') < (c, c') + n \cdot (d, d') \iff\]
				\[\iff (a + n \cdot b, a' + n \cdot b') < (c + n \cdot d, c' + n \cdot d') \iff\]
				\[\iff a + n \cdot b + c' + n \cdot d' < a' + n \cdot b' + c + n \cdot d \iff\]
				\[\iff (a + c') + n \cdot (b + d') < (a' + c) + n \cdot (b' + d),\]
				and such $n \in \NN$ exists by the archimedean property of $X$.
			\end{proof}
			
			In summary, if $X$ is a multiplicative prestreak, so is $FD(X)$, and if $X$ is further archimedean, $FD(X)$ is as well. Note that we have an embedding $X \to FD(X)$, given by $x \mapsto (x, 0)$. It is clear that this is a morphism.
			
			However, $FD(X)$ is not tight, even if $X$ is --- we have for example $(0, 0) \nap (1, 1)$. Nor is it a ring prestreak, unless $X$ was to start with: if we want to have $(x, y)$ such that $(a, b) + (x, y) = (0, 0)$, then necessarily $x = -a$, $y = -b$.
			
			Both issues are solved by composing $FD$ with $Q$. We want even more, though --- to turn an arbitrary (not necessarily multiplicative) streak into a ring streak. The idea is to first turn a streak $X$ into a multiplicative one by applying $\pos{}$, then use $FD$ to get formal differences, and finally $Q$ to get tightness and actual subtraction. Thus we define $\ring \dfeq Q \circ FD \circ \pos{}$.
			
			\begin{theorem}\label{Theorem: ring_streaks}
				If $X$ is a streak, then $\ring(X)$ is a ring streak and there exists a morphism $\rho_X\colon X \to \ring(X)$ which makes $\ring$ a reflection of streaks into ring streaks.
			\end{theorem}
			\begin{proof}
				If $X$ is a streak, $\pos{X}$ is a multiplicative streak by Proposition~\ref{Proposition: positive_part_is_multiplicative}, so by the above $FD(\pos{X})$ is an archimedean prestreak, so $\ring(X)$ is a streak.
				
				Taking any $[(a, b)] \in \ring(X)$, we have $[(a, b)] + [(b, a)] = [(a + b, b + a)] = [(0, 0)]$, so $-[(a, b)] = [(b, a)]$. In conclusion, $\ring(X)$ has total subtraction, so is a ring streak by Theorem~\ref{Theorem: streak_with_total_subtraction_is_ring}.
				
				Take any $x \in X$. By the archimedean property there exists $n \in \NN$ with $x + n > 0$. Define
				\[\rho_X(x) \dfeq [(x + n, n)].\]
				This is well defined: suppose we also have $x + n' > 0$. Without loss of generality assume $n \leq n'$ and let $n' = n + \Delta{n}$. Then $[(x + n', n')] = [(x + n + \Delta{n}, n + \Delta{n})] = [(x + n, n)]$.
				
				The fact that $\rho_X$ is a morphism is a straightforward verification.
				
				If $Y$ is a ring streak and $f\colon X \to Y$ a morphism, then we can define a map $\ring(X) \to Y$ by $[(a, b)] \mapsto f(a) - f(b)$. It is easy to check that this defines a morphism, so by Lemma~\ref{Lemma: reflection_in_a_preorder_category} $\ring$ is a reflection of streaks into ring streaks.
			\end{proof}
			
			Naturally, $\rho_X$, as a streak morphism, is injective, so it provides an embedding of an arbitrary streak into a ring streak.
			
			Note that when applying the above construction to $\NN$, we get the standard way of constructing $\ZZ$. Recalling Lemma~\ref{Lemma: interaction_of_reflections_with_initiality_and_terminality} and the fact that $\NN$ is the initial streak, we have a way to characterise the integers in our setting: $\ZZ$ is the initial ring streak.
			
			So, does this reflection commute with the other (co)reflections we had so far? It clearly doesn't commute with $\pos{}$, as ring streaks have negative elements (all of $\ZZ$, in fact, as we have seen). It trivially commutes with $Q$ (even if we extend the domain of $\ring$ to archimedean prestreaks, as we could). It is meaningless to ask whether $\ring$ commutes with $\arch$, as $\ring$ is defined only on (archimedean pre)streaks.
			
			As far as lattices are concerned, $\ring$ commutes with ${}^\lor \circ {}^\land \ism {}^\land \circ {}^\lor$. This basically amounts to the observation $\sup(-A) = -\inf(A)$ and $\inf(-A) = -\sup(A)$; we leave the precise verification to the reader. Interestingly, $\ring$ does not commute with individual ${}^\land$ and ${}^\lor$ (except in special cases such as classical mathematics, where all tight strict orders are lattices) --- it easily follows from these formulae that a semilattice ring streak is automatically a lattice ring streak.
		
		\subsection{Fields}
		
			Just like there is a standard method to turn semigroups to groups (which we used in the previous subsection to turn streaks into ring streaks) --- namely taking the formal differences --- there is a standard way to turn rings (nontrivial commutative ones without zero divisors, to be precise) into fields, namely taking the field of fractions. In this subsection we adopt this method to streaks.
			
			\begin{definition}
				A \df{field prestreak} is a ring prestreak, in which all positive elements are invertible.
			\end{definition}
			
			For $x \in X$ to be \df{invertible} of course means that there exists $y \in X$ with $x \cdot y = 1$ (therefore also $y \cdot x = 1$). As usual, if the inverse of $x$ exists, it is unique, and is denoted by $x^{-1}$, and the division is given by $a/b = \frac{a}{b} \dfeq a \cdot b^{-1}$.
			
			The definition of a field streak is, of course, a bit minimalistic: we usually require for a field that the invertible elements are precisely the nonzero ones. But this follows from the definition (at least for archimedean prestreaks).
			
			\begin{proposition}
				Let $X$ be a field prestreak. Then every $x \in X_{\apart 0}$ is invertable. If $X$ is archimedean, the converse also holds: if $x \in X$ is invertible, then $x \apart 0$.
			\end{proposition}
			\begin{proof}
				If $x \apart 0$, then $x > 0$ or $x < 0$. If the first, then $x$ is invertible by assumption. If the second, then $-x > 0$ and $x^{-1} = -(-x)^{-1}$.
				
				Suppose now that for $x \in X$ we have some $y \in X$ with $x \cdot y = 1 \apart 0$. If $X$ is archimedean, then $x \apart 0$ by Proposition~\ref{Proposition: multiplication_and_signs}.
			\end{proof}
			
			The idea, then, to get a field, is to make all positive elements invertable in a given (pre)streak. It makes sense to restrict oneself just to inverting positive elements, as division by them preserves $<$ and makes the definitions simpler.
			
			Let $X$ be a multiplicative prestreak. We define order and operations on $FQ(X) \dfeq X \times X_{> 0}$ for $(a, b), (c, d) \in FQ(X)$ as
			\[(a, b) < (c, d) \dfeq a \cdot d < b \cdot c,\]
			\[(a, b) + (c, d) \dfeq (a \cdot d + b \cdot c, b \cdot d),  \qquad  (a, b) \cdot (c, d) \dfeq (a \cdot c, b \cdot d).\]
			Clearly $FQ(X)$ is closed under the above defined $+$ and $\cdot$, as the product of two positive elements is again positive. The additive unit is $(0, 1)$ and the multiplicative unit $(1, 1)$.
			
			The proof that the above satisfies the multiplicative prestreak conditions is done much like the proofs up to this point (and of course like the standard proof for field of fractions), so we skip it.
			
			$FQ$ is a functor from the category of multiplicative prestreaks to itself: for $f\colon X \to Y$, define $FQ(f)(a, b) \dfeq (f(a), f(b))$. Also, every multiplicative prestreak $X$ embeds into $FQ(X)$ via the morphism $\upsilon_X\colon X \to FQ(X)$, $\upsilon_X(a) \dfeq (a, 1)$.
			
			Note that for $n \in \NN_{> 0}$ we have $n \cdot (a, b) = (n \cdot a \cdot b^{n-1}, b^n) \nap (n \cdot a, b)$.
			\begin{proposition}
				If $X$ is archimedean, then so is $FQ(X)$.
			\end{proposition}
			\begin{proof}
				Take $(a, a'), (b, b'), (c, c'), (d, d') \in FQ(X)$ such that $(b, b') < (d, d')$, \ie $b \cdot d' < b' \cdot d$, and so in turn $a' \cdot b \cdot c' \cdot d' < a' \cdot b' \cdot c' \cdot d$. We have
				\[(a, a') + n \cdot (b, b') < (c, c') + n \cdot (d, d') \iff\]
				\[\iff (a \cdot b' + n \cdot b \cdot a', a' \cdot b') < (c \cdot d' + n \cdot d \cdot c', c' \cdot d') \iff\]
				\[\iff a \cdot b' \cdot c' \cdot d' + n \cdot a' \cdot b \cdot c' \cdot d' < a' \cdot b' \cdot c \cdot d' + n \cdot a' \cdot b' \cdot c' \cdot d,\]
				and such $n \in \NN$ exists by the archimedean property of $X$.
			\end{proof}
			
			As in the previous subsections, at this point we note that, unlike the archimedean property, tightness is not preserved (\eg $(0, 1) \nap (0, 2)$), so to get the desired streak, we need to apply $Q$ at the end.
			
			\begin{theorem}\label{Theorem: from_ring_streaks_to_field_streaks}
				\
				\begin{enumerate}
					\item
						If $X$ is an archimedean multiplicative prestreak, then so is $FQ(X)$, and $Q(FQ(X))$ is a multiplicative streak.
					\item
						If $X$ is a ring streak, then $Q(FQ(X))$ is a field streak.
				\end{enumerate}
			\end{theorem}
			\begin{proof}
				\begin{enumerate}
					\item
						Immediate, by the above.
					\item
						$Q(FQ(X))$ has total substraction; we have $-[(a, b)] = [(-a, b)]$. Thus it is a ring streak by Theorem~\ref{Theorem: streak_with_total_subtraction_is_ring}. The statement $[(a, b)] > 0$ is equivalent to $a > 0$, in which case we easily see $[(a, b)]^{-1} = [(b, a)]$. All positive elements are invertible, so $X$ is a field streak.
				\end{enumerate}
			\end{proof}
			
			Let $\field \dfeq Q \circ FQ \circ \ring$ and $\varphi_X \dfeq \theta_X \circ \upsilon_X \circ \rho_X$.
			\begin{theorem}
				The functor $\field$ is a reflection from streaks to field streaks, with $\varphi$ the unit of the reflection.
			\end{theorem}
			\begin{proof}
				Clearly a composition of reflections is again a reflection, so it is sufficient to check that $Q \circ FQ$ is a reflection from ring streaks to field streaks. Given a morphism $f\colon X \to Y$ from a ring streak to a field streak, define $\overline{f}([(a, b)]) \dfeq f(a)/f(b)$. It is easy to check that this works.
			\end{proof}
			
			In the previous subsection we used Lemma~\ref{Lemma: interaction_of_reflections_with_initiality_and_terminality} to characterise the integers $\ZZ$ as the initial ring streak. A similar argument shows that the rationals $\QQ$ are the initial field streak.
			
			Commutativity of $\field$ with other (co)reflections is dealt with much like the case of $\ring$ was at the end of the previous subsection; in the case of lattices one just has to additionally take into account the formulae $\sup(A^{-1}) = (\inf{A})^{-1}$ and $\inf(A^{-1}) = (\sup{A})^{-1}$ (when all elements of $A$ are positive). Of course, $\field$ and $\ring$ themselves commute; we have $\field \circ \ring \ism \field \ism \ring \circ \field$.
		
		\subsection{Halved rings}
		
			In Subsection~\ref{Subsection: dense_streaks} we considered the property of streaks being dense (intuitively, in $\RR$) --- necessary if a streak is to be used in a construction of reals as its completion (of some sort).
			
			Is there a canonical way to impose density on a streak via a (co)reflection? Directly, no --- unless a streak $X$ is already dense, there is no smallest dense streak, containing $X$ (and certainly no streak, contained in $X$, will do).
			
			However, there is a way to turn a streak into one wherein every two elements have their average, fulfilling the interpolation property (recall Theorem~\ref{Theorem: characterization_of_dense_streaks}). Clearly this is equivalent to that for every element we also have its half. If we further require that the resulting streak is a ring, then we also have negative elements, so a dense streak by Theorem~\ref{Theorem: characterization_of_dense_streaks}.
			
			\begin{definition}
				A \df{halved ring (pre)streak} is a ring (pre)streak, in which $2$ is invertible, \ie we have $2^{-1} = \frac{1}{2} \in X$.
			\end{definition}
			
			Clearly if $X$ is a halved ring (pre)streak, then for any $x \in X$ and $n \in \NN$ we also have $\frac{x}{2^n} \in X$. One way to turn a streak into a halved ring one is to apply $\field$, then restrict to the equivalence classes possessing a representative of the form $(x, 2^n)$. Obviously $\phi$ corestricts to this, and one can easily check that this is a reflection from streaks to halved ring streaks.
			
			One can also perform a direct construction, without a detour over fields. Let $X$ be a ring streak (apply $\ring$ first if necessary). Consider the set of all $(x, n) \in X \times \NN$ (intuitively, $(x, n)$ respresents $\frac{x}{2^n}$). Let $\monus\colon \NN \times \NN \to \NN$ denote the \df{cutoff subtraction}:
			\[m \monus n \dfeq \begin{cases} m - n & \text{ if } m \geq n,\\ 0 & \text{ if } m \leq n \end{cases} \quad = \sup\{m, n\} - n\]
			for $m, n \in \NN$. For $(a, m), (b, n) \in X \times \NN$ define
			\[(a, m) < (b, n) \dfeq a \cdot 2^{n \monus m} < b \cdot 2^{m \monus n},\]
			\[(a, m) + (b, n) \dfeq \big(a \cdot 2^{n \monus m} + b \cdot 2^{m \monus n}, \sup\{m, n\}\big),\]
			\[(a, m) \cdot (b, n) \dfeq (a \cdot b, m + n).\]
			It is easy to check that this makes $X \times \NN$ into a halved ring prestreak, into which $X$ embeds via $x \mapsto (x, 0)$. Applying $Q$ at the end (to get a streak), we obtain the desired reflection (commutativity of which with other (co)reflections is dealt with as for $\ring$).
			
			In conclusion, we have a way how to transform any streak into a halved ring streak, and in particular, a dense ring streak. While density is enough for various constructions of reals (see Section~\ref{Section: models_of_reals}), it is useful to have the ring structure in addition, as this simplifies some formulae. In particular, it is no coincidence that diadic rationals are often used in computing to construct the reals, and diadic rationals can be characterized as the initial halved ring streak.

	\section{Real numbers}\label{Section: reals}
	
		With the general theory of streaks behind us, it is time to finally focus on the reals.
		
		As mentioned, we wanted streaks to be broad enough to already contain order, algebraic and topological structure and capture connections between them, but still general enough so that all the usual number sets, up to and including $\RR$, are streaks.
		
		We haven't used real numbers up to this point of the paper yet; the idea is that we can use streaks to now define them, so that
		\begin{itemize}
			\item
				the various notions of reals in various mathematical models satisfy our definition (we verify this in Section~\ref{Section: models_of_reals}), and
			\item
				the definition does not rely on particular constructions (such as Dedekind cuts or equivalence classes of Cauchy sequences of rationals); rather, it describes properties which we definitely want to hold for reals, and determines them up to isomorphism.
		\end{itemize}
		As mentioned in the introduction, there is already such a definition of reals, namely that they are a Dedekind-complete ordered field. Unfortunately, this does not work constructively; already the simple statement that every binary sequence $\NN \to \{0, 1\}$ has a supremum in $\RR$ implies \lpo[]. Also, different constructive models of reals need not even be isomorphic; for example, Cauchy reals can always be embedded into Dedekind reals (when we can construct both), but while the converse holds assuming countable choice, it does not hold in general~\cite{Lubarsky:2007:CCC:1224238.1224293}.
		
		Our definition of reals is a formalization of the following: $\RR$ is a set, equipped with an order relation $<$ (satisfying the usual properties), in which we can add and multiply, and it is the completion of rationals in the following sense: it contains $\QQ$, and is the largest such structure, in which rationals are dense.
		\begin{definition}\label{Definition: reals}
			$\RR$ is the terminal streak.
		\end{definition}
		Longer version: $\RR$ is the terminal object in the category $\Str$. We are not saying at this point, that a terminal streak necessarily exists; merely that we label any such with $\RR$. As a terminal object in a category, it is determined up to isomorphism; we use the definite article in ``the terminal streak'' in this sense.
		
		Aside from knowing immediately that $\RR$ is determined up to isomorphism (something not obvious from ``Dedekind-complete ordered field''), the real value of a categorical definition via a universal property is that we immediately know that $\RR$ possesses every reflective structure, in particular all the structure from the previous section. This isn't just a theoretical result; as we will see in Section~\ref{Section: models_of_reals}, this gives us explicit formulae for operations in concrete models of reals.
		
		\begin{theorem}\label{Theorem: structure_of_reals}
			If $\kat{D}$ is any reflective subcategory of $\Str$, closed under isomorphisms,\footnote{A full subcategory $\kat{D} \subseteq \kat{C}$ is \df{closed under isomorphisms} when for all objects $X$, $Y$ of $\kat{C}$, if $X$ is in $\kat{D}$ and $X \ism Y$, then $Y$ is in $\kat{D}$ also.} and $\RR$ exists, then it also lies in $\kat{D}$. In particular, $\RR$ is a lattice field streak.
		\end{theorem}
		\begin{proof}
			Use Lemma~\ref{Lemma: interaction_of_reflections_with_initiality_and_terminality}(\ref{Lemma: interaction_of_reflections_with_initiality_and_terminality: terminality}) (and note that the subcategories of lattice streaks and field streaks are closed under isomorphisms).
		\end{proof}
		
		\begin{remark}
			Actually, using Lemma~\ref{Lemma: interaction_of_reflections_with_initiality_and_terminality}(\ref{Lemma: interaction_of_reflections_with_initiality_and_terminality: terminality}) in full, we get more: $\RR$ is in fact also the terminal lattice streak, terminal ring streak, terminal field streak, terminal lattice ring streak and terminal lattice field  streak. It even goes in the other direction one step: $\RR$ is the terminal archimedean prestreak (if $X$ is an archimedean prestreak, then $\trm[Q(X)] \circ \theta_X\colon X \to \RR$ is a morphism, unique by Corollary~\ref{Corollary: streaks_preorder_category}). However, that is as far as it goes; $\RR$ is in general not a terminal prestreak.
		\end{remark}
		
		We have seen what the order and the algebraic structure of the reals is, but we want to say something about its topological structure as well. Since $\RR$ is a ring and a lattice, we can define the \df{absolute value} $|\insarg|\colon \RR \to \RR$ by $|a| \dfeq \sup\{a, -a\}$. This has all the expected properties.
		\begin{proposition}
			The following holds for all $a, b \in \RR$:
			\begin{enumerate}
				\item
					$|a| \geq 0$,
				\item
					$|a| = a \iff 0 \leq a$, and so in particular $|0| = 0$ and $|1| = 1$;
				\item
					$|a| > 0 \iff a \apart 0$, and consequently (together with the first item) $|a| = 0 \iff a = 0$,
				\item
					$|a + b| \leq |a| + |b|$,
				\item
					$|a \cdot b| = |a| \cdot |b|$.
			\end{enumerate}
		\end{proposition}
		\begin{proof}
			\begin{enumerate}
				\item
					If it were $|a| < 0$, then $|a| < a \lor a < 0$ and $|a| < -a \lor -a < 0$, but the first disjuncts contradict the definition of the absolute value, so $a < 0$ and $-a < 0$, and after summing, $0 < 0$, a contradiction.
				\item
					$|a| = a \iff \sup\{a, -a\} = a \iff -a \leq a \iff 0 \leq 2 a \iff 0 \leq a$
				\item
					Suppose $a > 0$; then $-a < 0 < a$, so $|a| = \sup\{a, -a\} = a > 0$. Similarly for $a < 0$. Conversely, suppose $|a| > 0$. Then $0 < a \lor a < |a|$ and $0 < -a \lor -a < |a|$. If we choose the first disjunct at least once, we have $a \apart 0$, but we cannot choose the second disjunct both times by Proposition~\ref{Proposition: finite_strict_suprema_and_infima}.
				\item
					Use Proposition~\ref{Proposition: interaction_of_bounds_with_algebra} to calculate
					\[|a| + |b| = \sup\{a, -a\} + \sup\{b, -b\} = \sup\{a + b, a - b, -a + b, -a - b\} \geq\]
					\[\geq \sup\{a + b, -a - b\} = |a + b|.\]
				\item
					Assume $|a \cdot b| < |a| \cdot |b|$; then $|a| \cdot |b| > 0$, so $|a| > 0$ and $|b| > 0$ by Proposition~\ref{Proposition: multiplication_and_signs} and the first item. By the third item $a \apart 0$ and $b \apart 0$. The consideration of all four cases leads to contradiction.
					
					Similarly, if $|a \cdot b| > |a| \cdot |b|$, then $|a \cdot b| > 0$, so $a \cdot b \apart 0$. Again using Proposition~\ref{Proposition: multiplication_and_signs} and considering all the cases, we obtain a contradiction. In conclusion $|a \cdot b| = |a| \cdot |b|$.
			\end{enumerate}
		\end{proof}
		
		This means that we can equip $\RR$ with the \df{euclidean metric} $d_E\colon \RR \times \RR \to \RR$, $d_E(a, b) \dfeq |a - b|$, and define balls for this metric, $\ball[E]{a}{r} \dfeq \st{x \in \RR}{d_E(a, x) < r}$. Since they are given by an open predicate, they are open in $\RR$, so in this sense the intrinsic topology of $\RR$ is at least as strong as the euclidean one (clearly the argument generalizes to arbitrary metric spaces: the intrinsic topology is at least as strong as the metric one). It can be strictly stronger, though --- for example, in classical sets where we take all of them to be discrete. However, in topological models we do get for the terminal streak precisely $\RR$ with the euclidean topology (see Subsection~\ref{Subsection: topological_models_of_reals}).

	\section{Models of reals}\label{Section: models_of_reals}
	
		In the previous section we defined the reals as the terminal streak. In this section we show, that the usual constructions of reals (within their proper mathematical framework) satisfy this definition.
		
		Rather than just stating a model of reals and do the verification, we'll study how the idea of the construction itself fits into the framework of streaks, obtaining further reflections. Typically a model of reals is then obtained by applying the reflection on a dense streak (recall Subsection~\ref{Subsection: dense_streaks}).
	
		\subsection{Cauchy reals}\label{Subsection: Cauchy_reals}
		
			In this subsection we observe that the Cauchy reals --- \ie the equivalence classes of rational Cauchy sequences --- satisfy our definition of reals, at least when countable choice holds (in particular in classical mathematics and many versions of constructive mathematics). It is known that without countable choice this ``Cauchy completion'' of rationals behaves badly --- it might not itself be Cauchy complete~\cite{Lubarsky:2007:CCC:1224238.1224293}.
			
			As is the common strategy in this paper, we won't immediately and directly prove the desired theorem, but will instead develop a more general theory to get a better insight into the structure in question, in this case Cauchy completness. As usual, this means constructing a reflection.
			
			Before we start, a few general words on Cauchy sequences. The usual definition is that $a$ is a Cauchy sequence in (a subset of) $\RR$ when
			\[\xall{\epsilon}{\RR_{> 0}}\xsome{m}{\NN}\xall{i,j}{\NN_{\geq m}}{|a_i - a_j| < \epsilon}\]
			holds. If we want to construct reals as a Cauchy completion of rationals, we can't already use $\RR$ in the definition; one way to rephrase it is
			\[\xall{n}{\NN_{> 0}}\xsome{m}{\NN}\xall{i,j}{\NN_{\geq m}}{|a_i - a_j| < \frac{1}{n}}.\]
			Constructively we often need more: an explicit modulus of convergence, \ie a mapping which tells us how late terms of the sequence must we take to obtain the desired precision. There are different ways to express this; we will take the map, obtained from the above condition by the application of countable choice:
			\[\xsome{M}{\NN^\NN}\xall{n}{\NN_{> 0}}\xall{i,j}{\NN_{\geq M(n)}}{|a_i - a_j| < \frac{1}{n}}.\]
			This will be our definition of a Cauchy sequence. In the presence of countable choice it is equivalent to the more standard one above, but we'll try to prove as much as possible in the general setting, so we assume the stronger latter condition.
			
			There is still a minor detail to rephrase $|a_i - a_j| < \frac{1}{n}$ in the form which uses only the general (pre)streak operations. It is equivalent to
			\[n \cdot a_i < 1 + n \cdot a_j \land n \cdot a_j < 1 + n \cdot a_i,\]
			but the second conjunct becomes superfluous after we universally quantify over $i$ and $j$. Also, since we no longer divide by $n$, we don't need to explicitly exclude $0$ from its domain.
			\begin{definition}
				Let $X$ be a prestreak and $a\colon \NN \to X$ a sequence in it.
				\begin{itemize}
					\item
						A map $M\colon \NN \to \NN$ is called a \df{modulus of convergence} for $a$ when it satisfies $\xall{n}{\NN}\all{i,j}{\NN_{\geq M(n)}}{n \cdot a_i < 1 + n \cdot a_j}$.
					\item
						A sequence which possesses a modulus of convergence is called \df{Cauchy}.
				\end{itemize}
			\end{definition}
			
			Clearly a modulus of convergence can be arbitrarily increased and still remain a modulus of convergence. In particular, any Cauchy sequence has an increasing one: just replace $M$ with $n \mapsto \sup\st{M(i)}{i \in \NN_{\leq n}}$.
			
			\begin{lemma}\label{Lemma: distance_between_Cauchy_sequence_terms}
				Let both $M$, $N$ be moduli of convergence for $a \in X^\NN$. Then for all $m, n \in\NN$ and all $i \in \NN_{\geq M(m)}$, $j \in \NN_{\geq N(n)}$
				\[m n \cdot a_i < m n \cdot a_j + \sup\{m, n, 1\}.\]
				In particular, we have the special cases
				\[m n \cdot a_{M(m)} < m n \cdot a_{N(n)} + \sup\{m, n, 1\}\]
				and
				\[n \cdot a_{M(n)} < n \cdot a_{N(n)} + 1.\]
			\end{lemma}
			\begin{proof}
				The statements clearly hold if at least one of $m$, $n$ is zero. Assume hereafter that $m, n \geq 1$. The order on $\NN$ is decidable, so we have $M(m) \leq N(n) \lor N(n) \leq M(m)$.
				
				Assume first $M(m) \leq N(n)$; then $i, j \geq M(m)$. By the definition of a modulus of convergence we have $m \cdot a_i < 1 + m \cdot a_j$, so
				\[m n \cdot a_i < n + m n \cdot a_j \leq m n \cdot a_{N(n)} + \sup\{m, n, 1\}.\]
				
				Similarly for $M(m) \geq N(n)$; then $i, j \geq N(n)$ whence $n \cdot a_i < 1 + n \cdot a_j$, so
				\[m n \cdot a_i < m + m n \cdot a_j \leq m n \cdot a_{N(n)} + \sup\{m, n, 1\}.\]
			\end{proof}
			
			
			Denote the set of Cauchy sequences in a prestreak $X$ by $CS(X)$, \ie
			\[CS(X) \dfeq \st{a \in X^\NN}{\xsome{M}{\NN^\NN}\xall{n}{\NN}\all{i,j}{\NN_{\geq M(n)}}{n \cdot a_i < 1 + n \cdot a_j}}.\]
			There is an embedding $c_X\colon X \to CS(X)$ which maps an element to the corresponding constant sequence, \ie $c_X(x)_n \dfeq x$ for all $x \in X$ and $n \in \NN$. A constant sequence is of course Cauchy: every map $\NN \to \NN$ is its modulus of convergence. (The converse also holds: if every map $\NN \to \NN$ is a modulus of convergence of a certain sequence, then that sequence is constant.)
			
			We claim that if $X$ is an archimedean prestreak, then so is $CS(X)$.
			
			For $a, b \in CS(X)$ fix their moduli of convergence $M, N \in \NN^\NN$, and define
			\[a < b \dfeq \some{n}{\NN}{n \cdot a_{M(n)} + 2 < n \cdot b_{N(n)}}.\]
			The idea behind this definition is that we order Cauchy sequences according to their limits, and we have (so far only on the intuitive level, but we make this precise in Theorem~\ref{Theorem: limit_operator}) $|a_i - \lim a| \leq \frac{1}{n}$ for $i \in \NN_{\geq M(n)}$, and the above condition simply states $a_{M(n)} + \frac{1}{n} < b_{M(n)} - \frac{1}{n}$.
			
			We claim that $<$ is well defined (independent of the choices for moduli of convergence). Let $M'$, $N'$ also be moduli for $a$, $b$ respectively. Suppose we have $n \in \NN$ (necessarily $n \geq 1$) such that $n \cdot a_{M(n)} + 2 < n \cdot b_{N(n)}$. By the archimedean property there exists $m \in \NN$ (we can take $m \geq n$) such that $n + m \cdot (n \cdot a_{M(n)} + 2) < m \cdot n \cdot b_{N(n)}$. Then (using Lemma~\ref{Lemma: distance_between_Cauchy_sequence_terms} twice)
			\[n \cdot (2 m \cdot a_{M'(2m)} + 2) + 2 m < 2 m n \cdot a_{M(n)} + \sup\{2m, n, 1\} + 2 n + 2 m =\]
			\[= 2 \cdot \big(n + m \cdot (n \cdot a_{M(n)} + 2)\big) < 2 m n \cdot b_{N(n)} <\]
			\[< 2 m n \cdot b_{N'(2m)} + \sup\{2m, n, 1\} = n \cdot \big(2 m \cdot b_{N'(2m)}\big) + 2 m,\]
			and so $2m \cdot a_{M'(2m)} + 2 < 2m \cdot b_{N'(2m)}$.
			
			Clearly $<$ is given by an open predicate on $CS(X)$. Its negation $\leq$ is clearly closed, as
			\[a \leq b \iff \all{n}{\NN}{n \cdot a_{M(n)} \leq n \cdot b_{N(n)} + 2}.\]
			
			In the definition of $<$ we compare only one term of the first sequence with only one term of the second one, but the comparison is actually valid for all terms from somewhere onward.
			\begin{lemma}\label{Lemma: characterization_of_order_of_cauchy_sequences}
				Let $X$ be an archimedean prestreak and $a, b \in CS(X)$ with moduli of convergence $M$, $N$ respectively. The following statements are equivalent:
				\begin{enumerate}
					\item
						$a < b$, \ \ie \ $\some{n}{\NN}{n \cdot a_{M(n)} + 2 < n \cdot b_{N(n)}}$,
					\item
						$\xsome{n}{\NN}\xall{i}{\NN_{\geq M(n)}}\all{j}{\NN_{\geq N(n)}}{n \cdot a_i + 2 < n \cdot b_j}$,
					\item
						$\xsome{n, k}{\NN}\all{i, j}{\NN_{\geq k}}{n \cdot a_i + 2 < n \cdot b_j}$.
				\end{enumerate}
			\end{lemma}
			\begin{proof}
				\begin{itemize}
					\item\proven{$(1 \impl 2)$}
						Take $m \in \NN$ (necessarily $m > 0$), for which we have $m \cdot a_{M(m)} + 2 < m \cdot b_{N(m)}$. Since $X$ is archimedean, we have some $q, r \in \QQ$ with $m \cdot a_{M(m)} + 2 < q < r < m \cdot b_{N(m)}$. Let $n \in \NN_{\geq m}$ be large enough so that $\frac{1}{n} < \frac{r-q}{2m}$.
						
						Take any $i \in \NN_{\geq M(n)}$, $j \in \NN_{\geq N(n)}$. By Lemma~\ref{Lemma: distance_between_Cauchy_sequence_terms} we have $m n \cdot a_i < m n \cdot a_{M(m)} + n$ and $m n \cdot b_{N(m)} < m n \cdot b_j + n$. Thus
						\[m \cdot \big(n \cdot a_i + 2\big) + n < m n \cdot a_{M(m)} + n + 2 m + n =\]
						\[= n \cdot \big(m \cdot a_{M(m)} + 2\big) + 2 m < n q + 2 m < n r < m n \cdot b_{N(m)} < m n \cdot b_j + n,\]
						so $n \cdot a_i + 2 < n \cdot b_j$, as desired.
					\item\proven{$(2 \impl 3)$}
						Take $k \dfeq \sup\big\{M(n), N(n)\big\}$.
					\item\proven{$(3 \impl 1)$}
						Suppose we have the assumed $n, k \in \NN$. If we increase a modulus of convergence, it remains a modulus of convergence, so $M'(m) \dfeq M(m) + k$ and $N'(m) \dfeq N(m) + k$ are also moduli of $a$, $b$ respectively, and we obviously have $n \cdot a_{M'(n)} + 2 < n \cdot b_{N'(n)}$. As has been shown above, the statement $a < b$ is independent of the choice of moduli of convergence, so it holds.
				\end{itemize}
			\end{proof}
			
			We have $a < a \iff \some{n}{\NN}{n \cdot a_{M(n)} + 2 < n \cdot a_{M(n)}} \iff \some{n}{\NN}{2 < 0}$, clearly a false statement, so $<$ is asymmetric on $CS(X)$. To show that $<$ is also cotransitive on $CS(X)$, take $a, b, x \in CS(X)$ and denote their moduli of convergence by $M$, $N$, $O$ (without loss of generality assume that they are increasing). Suppose $a < b$, \ie $n \cdot a_{M(n)} + 2 < n \cdot b_{N(n)}$. By the archimedean property of $X$ there exist $q, r \in \QQ$ with $n \cdot a_{M(n)} + 2 < q < r < n \cdot b_{N(n)}$. Let $m \in \NN_{> n}$ be large enough so that $\frac{1}{m} < \frac{r-q}{5n}$.
			
			By cotransitivity $\frac{3 q + 2 r}{5} < n \cdot x_{O(m)} + 1 \lor n \cdot x_{O(m)} + 1 < \frac{2 q + 3 r}{5 n}$. Assume that the first disjunct holds. Then (with the help of Lemma~\ref{Lemma: distance_between_Cauchy_sequence_terms})
			\[n \cdot \big(m \cdot a_{M(m)} + 2\big) + m < m n \cdot a_{M(n)} + \sup\{m, n, 1\} + 2 n + m =\]
			\[= m \cdot \big(n \cdot a_{M(n)} + 2\big) + 2 n < m q + 2 n < m \cdot \tfrac{3 q + 2 r}{5} <  n m \cdot x_{O(m)} + m,\]
			so $m \cdot a_{M(m)} + 2 < m \cdot x_{O(m)}$ and therefore $a < x$. We can check in a very similar way that $n \cdot x_{O(m)} + 1 < \frac{2 q + 3 r}{5 n}$ implies $x < b$. In conclusion, $<$ is a strict order on $CS(X)$.
			
			For $a, b \in CS(X)$ we define $(a + b)_n \dfeq a_n + b_n$. The result is again in $CS(X)$: if $M, N$ are moduli of convergence of $a, b$ respectively, then $n \mapsto \sup\{M(2n), N(2n)\}$ is a modulus of convergence of the sum. Clearly $+$ is commutative and associative, $c_X(0)$ is the unit, and it satisfies the law, connecting it with $<$.
			
			To define the multiplication of Cauchy sequences, we first characterise, when an element of $CS(X)$ is positive.
			\begin{lemma}\label{Lemma: Cauchy_sequence_positive}
				Let $X$ be an archimedean prestreak. Then for every $a \in CS(X)$ the following is equivalent.
				\begin{enumerate}
					\item
						$a > c_X(0)$
					\item
						$\xsome{m,n}{\NN}\xall{k}{\NN_{\geq m}}{n \cdot a_k > 1}$
					\item
						$\xsome{n}{\NN}\xall{k}{\NN_{\geq n}}{n \cdot a_k > 1}$
				\end{enumerate}
			\end{lemma}
			\begin{proof}
				\begin{itemize}
					\item\proven{$(1 \impl 2)$}
						Let $M$ be a modulus of continuity for $a$; then by assumption we have $n \in \NN$ such that $2 < n \cdot a_{M(n)}$. Set $m \dfeq M(n)$. Then for any $k \in \NN_{\geq m}$ we have $n \cdot a_{M(n)} < 1 + n \cdot a_k$ whence $1 < n \cdot a_k$.
					\item\proven{$(2 \impl 1)$}
						Use the assumption to provide suitable $m, n \in \NN$. Without loss of generality we can find a modulus of convergence $M$ for $a$ such that $M(2n) \geq m$ (if necessary, replace it with $i \mapsto \sup\{M(i), m\}$). For $k = M(2n)$ we then obtain $n \cdot a_k > 1$, so $2n \cdot a_{M(2n)} > 2$, as desired.
					\item\proven{$(2 \Leftrightarrow 3)$}
						Obvious.
				\end{itemize}
			\end{proof}
			
			We now define the multiplication as $(a \cdot b)_n \dfeq a_n \cdot b_n$ for $a, b \in CS(X)_{> c_X(0)}$. To see that this is again a Cauchy sequence, choose moduli of convergence $M, N$ for $a, b$. By Lemma~\ref{Lemma: Cauchy_sequence_positive} we have $n \in \NN$ such that $n \cdot a_k > 1$ (and so $a_k > 0$) for all $k \in \NN_{\geq n}$, and similarly for $b$; in fact, since $n$ can obviously be increased, we can assume that we have the same $n$ for $a$ and $b$, and that furthermore $n \geq 1$ and (by the archimedean property of $X$) $a_{M(1)} + 1 < n$, $b_{N(1)} + 1 < n$. Hence by Lemma~\ref{Lemma: distance_between_Cauchy_sequence_terms} for any $i \in \NN_{\geq M(1)}$ we have $a_i < a_{M(1)} + 1 < n$, and similarly for $b$.
			
			Define $O\colon \NN \to \NN$ by $O(m) \dfeq \sup\big\{M(2 n m), N(2 n m), M(1), N(1)\big\}$. Then for any $m \in \NN$ and $i, j \in \NN_{\geq O(m)}$ we have
			\[2 n m \cdot a_i \cdot b_i < a_i \cdot (1 + 2 n m \cdot b_j) = a_i + 2 n m \cdot a_i \cdot b_j <\]
			\[< a_i + (1 + 2 n m \cdot a_j) \cdot b_j = a_i + b_j + 2 n m \cdot a_i \cdot b_j < 2 n \cdot (1 + m \cdot a_j \cdot b_j),\]
			so $m \cdot a_i \cdot b_i < 1 + m \cdot a_j \cdot b_j$, and therefore $O$ is a modulus of convergence of $a \cdot b$.
			
			We have to still see that the product of positive elements is positive. Set $m \dfeq n^2$ (note $m \geq n$); then for any $k \in \NN_{\geq m}$ we have $m \cdot a_k \cdot b_k \geq (n \cdot a_k) \cdot (n \cdot b_k) > 1 \cdot 1 = 1$, proving the claim (by Lemma~\ref{Lemma: Cauchy_sequence_positive}).
			
			Clearly so-defined multiplication is commutative, associative and distributive over addition. With similar methods as above, we show that the law, connecting $\cdot$ with $<$, holds as well. Thus $CS(X)$ is a prestreak.
			
			\begin{proposition}
				For any archimedean prestreak $X$ the map $c_X\colon X \to CS(X)$ is a morphism.
			\end{proposition}
			\begin{proof}
				Mostly obvious; the only thing actually needed to be checked is the preservation of $<$.
				
				Take any $x, y \in X$ with $x < y$. By the archimedean property of $X$ there exists $n \in \NN$ with $2 + n \cdot x < n \cdot y$. Taking any maps $\NN \to \NN$ as moduli of convergence of $c_X(x)$ and $c_X(y)$, we obtain $c_X(x) < c_X(y)$, as desired.
			\end{proof}
			
			Note that $(n \cdot a)_k = n \cdot a_k$ for $n, k \in \NN$, $a \in CS(X)$. Here is the verification of the archimedean property of $CS(X)$.
			\begin{proposition}
				For an archimedean prestreak $X$ the prestreak $CS(X)$ is also archimedean.
			\end{proposition}
			\begin{proof}
				We use Lemma~\ref{Lemma: characterization_of_archimedean_prestreaks}. Take any $a, b \in CS(X)$ with moduli of convergence $M$, $N$ respectively, and assume $a < b$, \ie there is some $n \in \NN$ (necessarily $n > 0$) for which $n \cdot a_{M(n)} + 2 < n \cdot b_{N(n)}$ holds.
				
				Since $X$ is archimedean, we can find $u, v \in \QQ$ such that $n \cdot a_{M(n)} + 2 < u < v < n \cdot b_{N(n)}$. Let $m \in \NN_{> 0}$ be large enough that $\frac{1}{m} < \frac{v-u}{4n}$, and set $r \dfeq \frac{u-1}{n} + \frac{2}{m}$, $s \dfeq \frac{v-1}{n} - \frac{2}{m}$; note that $r < s$.
				\begin{itemize}
					\item\proven{$q < a$}
						Write $q = i-j$ where $i, j \in \NN$. Then $q < a$ is equivalent to $i < j + a$. One can easily check that $M$ is a modulus of convergence also for $j + a$ while for the constant sequence $i$ we can take any modulus of convergence. Thus $i < j + a$ is equivalent to the existence of $m \in \NN$, for which $m \cdot i + 2 < m \cdot (j + a_{M(m)})$. Taking $m = n$, we get
						\[n \cdot (j + a_{M(n)}) > n \cdot j + n \cdot q + 2 = n \cdot i + 2,\]
						as desired.
					\item\proven{$b < t$}
						Goes the same as in the previous item.
					\item\proven{$a < r$}
						Take $i, j \in \NN$, $k \in \NN_{> 0}$ such that $u = \frac{i-j}{k}$. Then $a < r$ is equivalent to $k n m \cdot a + j m + k m < i m + 2 k n$. Note that $l \mapsto M(k n m l)$ is a modulus of convergence for $k n m \cdot a + j m + k m$, so to prove the statement, we need to find $l \in \NN$ with $k n m l \cdot a_{M(k n m l)} + j m l + k m l + 2 < i m l + 2 k n l$. Actually, any $l \geq 1$ works, as the following calculation (using Lemma~\ref{Lemma: distance_between_Cauchy_sequence_terms}) shows.
						\[k n m l \cdot a_{M(k n m l)} + j m l + k m l + 2 < k n m l \cdot a_{M(n)} + k m l + j m l + k m l + 2 <\]
						\[< k m l u + j m l + 2 = i m l + 2 \leq i m l + 2 k n l\]
					\item\proven{$s < b$}
						Goes the same as in the previous item.
				\end{itemize}
			\end{proof}
			
			Recall that we may compare elements from different archimedean prestreaks.
			\begin{lemma}\label{Lemma: distance_of_cauchy_sequence_terms_from_the_limit}
				Let $X$ be an archimedean prestreak and $a \in CS(X)$ a Cauchy sequence with a modulus of convergence $M$. Then for any $n \in \NN$ and $k \in \NN_{\geq M(n)}$
				\[n \cdot a_k \leq n \cdot a + 1 \qquad \text{and} \qquad n \cdot a \leq n \cdot a_k + 1.\]
			\end{lemma}
			\begin{proof}
				By Lemma~\ref{Lemma: invariance_of_order_across_archimedean_prestreaks} these statements are equivalent to $n \cdot c_X(a_k) \leq n \cdot a + 1$ and $n \cdot a \leq n \cdot c_X(a_k) + 1$.
				
				Assume $n \cdot a + 1 < n \cdot c_X(a_k)$. Note that $m \mapsto M(n m)$ as a possible modulus of convergence for $n \cdot a + 1$, so there exists $m \in \NN$ (necessarily $m > 0$) such that we have $m \cdot \big(n \cdot a_{M(n m)} + 1\big) + 2 < m n \cdot a_k$. However, by Lemma~\ref{Lemma: distance_between_Cauchy_sequence_terms} we have $m n \cdot a_k < m n \cdot a_{M(n m)} + m$, a contradiction. Thus $n \cdot c_X(a_k) \leq n \cdot a + 1$.
				
				The second statement is proved similarly.
			\end{proof}
			\begin{remark}
				Lemma~\ref{Lemma: distance_of_cauchy_sequence_terms_from_the_limit} would not hold if we required the strict inequality $<$ instead of $\leq$. As a counterexample, take $X = \QQ$, $a_0 \dfeq 0$, $a_n \dfeq \frac{1}{n}$ for $n \in \NN_{> 0}$, and $M \dfeq \id[\NN]$.
			\end{remark}
			
			While $CS(X)$ is an archimedean prestreak if $X$ is, it is not the case that $CS(X)$ is a streak if $X$ is (intuitively, different Cauchy sequences can have the same limit). As usual, to obtain a streak, we need to apply $Q$ at the end. Denote then the composition of $Q$ and $CS$ by $CC$, and the composition of $\theta_{CS(X)}$ and $c_X$ by $\gamma_X\colon X \to CC(X)$.
			
			\begin{definition}
				A streak $X$ is \df{Cauchy complete} when $\gamma_X$ is an isomorphism (\ie it has an inverse).
			\end{definition}
			
			We can construct the \df{limit operator} for any Cauchy complete streak $X$.
			\[\xymatrix@+1em{
				&  CS(X) \ar[d]^{\theta_{CS(X)}}  \\
				X \ar[ru]^{c_X} \ar[r]_{\gamma_X}^\ism  &  CC(X)
			}\]
			Define $\lim_X\colon CS(X) \to X$ as $\lim_X \dfeq \gamma_X^{-1} \circ \theta_{CS(X)}$. As a composition of two morphisms, $\lim_X$ is itself a morphism. Note that to define it, we did not have to resort to notions such as `metric' or `neighbourhood'. Indeed, the usual definition of a limit is in our setting a theorem (and of course the usual properties of a limit follow as well).
			\begin{theorem}\label{Theorem: limit_operator}
				Let $X$ be a Cauchy complete streak and $a, b \in CS(X)$ Cauchy sequences with moduli of convergence $M$, $N$.
				\begin{enumerate}
					\item
						We have $\lim_X \circ c_X = \id[X]$, \ie the limit of a constant sequence is any of its terms.
					\item
						For any $n \in \NN_{> 0}$ the terms of the sequence $a$ from $M(n)$ onward are at most $\frac{1}{n}$ away from $\lim_X(a)$. More formally, for all $n \in \NN$ and $k \in \NN_{\geq M(n)}$ we have $n \cdot a_k \leq n \cdot \lim_X(a) + 1$ and $n \cdot \lim_X(a) \leq n \cdot a_k + 1$.
					\item
						$\lim_X(a + b) = \lim_X(a) + \lim_X(b)$ and $\lim_X(a \cdot b) = \lim_X(a) \cdot \lim_X(b)$ (whenever these products are defined).
				\end{enumerate}
			\end{theorem}
			\begin{proof}
				\begin{enumerate}
					\item
						By definition of $\lim_X$ and $\gamma_X$
						\[\lim\nolimits_X \circ c_X = \gamma_X^{-1} \circ \theta_{CS(X)} \circ c_X = \big(\theta_{CS(X)} \circ c_X\big)^{-1} \circ \theta_{CS(X)} \circ c_X = \id[X].\]
					\item
						This is just Lemma~\ref{Lemma: distance_of_cauchy_sequence_terms_from_the_limit} with $\lim_X$ applies to one side of the inequalities (we can do this by Lemma~\ref{Lemma: invariance_of_order_across_archimedean_prestreaks}).
					\item
						As a morphism, $\lim_X$ preserves addition and multiplication.
				\end{enumerate}
			\end{proof}
			
			\begin{remark}
				Note that for any morphism $f\colon X \to Y$ between two Cauchy complete streaks and any $a \in CS(X)$ we have $\lim_Y(f \circ a) = f\big(\lim_X(a)\big)$. In categorical terms, $\lim$ is a natural transformation.\footnote{Denote the category of Cauchy complete streaks by $\Ccstr$. Then $\lim$ maps from the functor $CS\colon \Ccstr \to \Apstr$ to the inclusion functor $\Ccstr \hookrightarrow \Apstr$.}
			\end{remark}
			
			The idea of the constructions $CS$ and $CC$ is that $CS(X)$ is the set of Cauchy sequences in a streak $X$, ordered according to their limits, and so $CC(X)$ is the set of those limits. In other words, $CC(X)$ should be the Cauchy completion of $X$. However, as already mentioned, constructively such ``completion'' need not be idempotent --- the result need not be Cauchy complete. This is because if we have a sequence of equivalence classes, we might not be able to produce a sequence of their representatives. However, that is clearly not a problem if we also assume the axiom of countable choice. Indeed, this axiom (or some variant of it) is considered necessary to work with Cauchy sequences constructively.
			
			If countable choice does hold, we get the expected result.
			\begin{theorem}
				Assume countable choice, and let $X$ be any streak.
				\begin{enumerate}
					\item
						$CC(X)$ is Cauchy complete.
					\item
						$CC$ is a reflection from streaks to Cauchy complete streaks.
					\item
						If $X$ is dense, then $CC(X)$ is a terminal streak (so a model of $\RR$ by Definition~\ref{Definition: reals}).
				\end{enumerate}
			\end{theorem}
			\begin{proof}
				\begin{enumerate}
					\item
						Let $a\colon \NN \to CC(X)$ be a Cauchy sequence in $CC(X)$; choose its modulus of convergence $M\colon \NN \to \NN$ (without loss of generality assume it is increasing). By countable choice we can produce two sequences of sequences $b\colon \NN \to CS(X)$, $N\colon \NN \to \NN^\NN$ such that for all $n \in \NN$, $a_n = [b_n]$ and $N_n$ is a modulus of convergence for $b_n$ (without loss of generality assume that $N$ is increasing in both variables). Define a new sequence $s\colon \NN \to X$ by $s_n \dfeq b_{n,n}$ and a map $O\colon \NN \to \NN$ by $O(n) \dfeq \sup\big\{N\big(M(3 n)\big)(3 n), M(3 n)\big\}$.
						
						Using Lemma~\ref{Lemma: distance_of_cauchy_sequence_terms_from_the_limit} we have
						\[3n \cdot s_i = 3n \cdot b_{i,i} \leq 3n \cdot b_i + 1 < 3n \cdot b_j + 2 \leq 3n \cdot b_{j,j} + 3 = 3n \cdot s_j + 3,\]
						\ie $n \cdot s_i < n \cdot s_j + 1$, so $s$ is a Cauchy sequence with a modulus of convergence $O$.
						
						We claim $\gamma_{CC(X)}([s]) = a$, or equivalently, $c_{CS(X)}(s) \nap b$. Assume $c_{CS(X)}(s) < b$; by Lemma~\ref{Lemma: characterization_of_order_of_cauchy_sequences} we have some $n \in \NN$ such that $n \cdot s + 2 < n \cdot b_j$ for all $j \in \NN_{\geq M(n)}$. This is furthermore equivalent to the existence of $n, m \in \NN$ such that for all $j \in \NN_{\geq M(n)}$, $i \in \NN_{\geq N(j)(n m)}$ we have
						\[m \cdot \big(n \cdot s_{O(n m)} + 2\big) + 2 < m n \cdot b_{j, i}.\]
						Take $i = j = O(n m)$; we thus get
						\[m n \cdot b_{O(n m), O(n m)} + 2 m + 2 < m n \cdot b_{O(n m), O(n m)}\]
						which is of course a contradiction. Similarly we derive a contradiction from $c_{CS(X)}(s) < b$.
						
						This shows that $\gamma_{CC(X)}$ is surjective, so an isomorphism by Corollary~\ref{Corollary: surjective_streak_morphism_is_iso}.
					\item
						Take any streak $X$, a Cauchy complete streak $Y$ (meaning $\gamma_Y^{-1}$ exists) and a morphism $f\colon X \to Y$. Define $\overline{f}\colon CC(X) \to Y$ by $\overline{f}([a]) \dfeq \lim_Y(f \circ a)$. This is well defined since if $[a] = [b]$, then $[f \circ a] = [f \circ b]$ (exercise) whence $\lim_Y(f \circ a) = \lim_Y(f \circ b)$. The verification, that $\overline{f}$ is a morphism, is easy. Finally, we have
						$\overline{f}\big([c_X(x)]\big) = \lim_Y\big(f \circ c_X(x)\big) = \lim_Y\big(c_Y(f(x))\big) = f(x)$.
					\item
						Take any streak $Y$ and $y \in Y$. Set $k_0 \dfeq 0$ and $a_0 \dfeq 0$. For each $n \in \NN_{> 0}$ use Lemma~\ref{Lemma: archimedean_prestreak_is_a_union_of_rational_intervals} to find $k_n \in \ZZ$ such that $y \in \intoo[Y]{\frac{k_n-1}{4 n}}{\frac{k_n+1}{4 n}}$. Using density of $X$ we can find $a_n \in \intoo[X]{\frac{k_n-1}{4 n}}{\frac{k_n+1}{4 n}}$. By countable choice this defines sequences $k\colon \NN \to \ZZ$ and $a\colon \NN \to X$.
						
						Note that $a$ is a Cauchy sequence in $X$ with a modulus of continuity $\id[\NN]$. Indeed, for any $n \in \NN_{> 0}$ and $i, j \in \NN_{\geq n}$
						\[2 n \cdot a_i < 2 n \cdot \frac{k_i + 1}{4 i} \leq 2 n \cdot \frac{k_i - 1}{4 i} + 1 < 2 n \cdot y + 1 <\]
						\[< 2 n \cdot \frac{k_j + 1}{4 j} + 1 \leq 2 n \cdot \frac{k_j - 1}{4 j} + 2  < 2 n \cdot a_j + 2\]
						whence $n \cdot a_i < 1 + n \cdot a_j$, as desired. Obviously, this formula holds also for $n = 0$.
						
						Define now $f(y) \dfeq [a]$. We claim this is well defined. Let $k'$ and $a'$ be another sequences, satisfying the required properties. Suppose $[a] < [a']$, \ie there exists $n \in \NN$ (necessarily $n > 0$) such that $n \cdot a_n + 2 < n \cdot a'_n$. On the other hand we have
						\[4 n \cdot a'_n < k'_n + 1 < 4 n \cdot y + 2 < k_n + 3 < 4 n \cdot a_n + 4,\]
						\ie $n \cdot a'_n < n \cdot a_n + 1$, a contradiction. By symmetry $[a'] < [a]$ leads to contradiction as well; thus $[a] = [a']$.
						
						We are done if we check that $f$ is a streak morphism. By Theorem~\ref{Theorem: preservation_of_rationals_implies_streak_morphism} it is sufficient to verify that it preserves comparison with rationals on both sides. We prove only $q < y \implies q < f(y)$; the other implication $y < q \implies f(y) < q$ works similarly.
						
						Take then arbitrary $y \in Y$ and $q \in \QQ$ with $q < y$. Write $q = \frac{i-j}{m}$ where $i, j \in \NN$, $m \in \NN_{\geq 3}$. By definition we have
						\[q < f(y) \iff i < j + m \cdot f(y) \iff \some[1]{n}{\NN}{n \cdot i + 2 < n \cdot j + n m \cdot a_{n m}}.\]
						Pick $r \in \QQ$ with $q < r < y$ and let $n \in \NN_{> 0}$ be large enough so that $\frac{1}{n} < r - q$. We have
						\[n \cdot j + n m \cdot a_{n m} > n j + \frac{k_{n m} - 1}{4} = n j + \frac{k_{n m} + 1}{4} - \frac{1}{2} > n j + n m r - \frac{1}{2} >\]
						\[> n j + m \cdot (1 + n q) - \frac{1}{2} = n j + m + n \cdot (i-j) - \frac{1}{2} = n i + m - \frac{1}{2} > n i + 2.\]
				\end{enumerate}
			\end{proof}
			In particular (assuming countable choice), the \df{Cauchy reals} $CC(\QQ)$ are a model of reals.
			
			Despite needing countable choice for this final result, we can use the theory, developed in this subsection, even in a setting without it. A terminal streak $\RR$ might still exist (for example, the Dedekind reals in the next subsection do not require choice), in which case we have for any streak $X$ the embeddings $X \stackrel{\gamma_X}{\longrightarrow} CC(X) \stackrel{\trm[CC(X)]}{\longrightarrow} \RR$. In particular the Cauchy reals are always a subset of $\RR$. Also, the existence of these embeddings implies $\RR \ism CC(\RR)$, so a terminal streak is always Cauchy complete, even in the absence of countable choice. In particular we always have the limit operator $\lim_\RR\colon CS(\RR) \to \RR$, the existence of which is assured by the universal property of $\RR$, rather than its topology or metric.
			
		\subsection{Dedekind reals}\label{Subsection: Dedekind_reals}
		
			In this subsection we observe that the usual construction of \df{Dedekind reals} (where a real is represented by a pair of sets, one with lower and the other with its upper rational bound) satisfies our definition of $\RR$. However, due to our introduction of the additional topological structure, we need to restrict ourselves to open cuts (with closed complements) --- something that is known from ASD~\cite{Bauer_A_Taylor_P_2009:_the_dedekind_reals_in_abstract_stone_duality} and synthetic topology~\cite{Lesnik_D_2010:_synthetic_topology_and_constructive_metric_spaces}. Of course, in classical mathematics, and those constructive examples where it makes sense to take $\optp(X) = \pst(X)$, this amounts to no additional assumption, and we get the usual Dedekind cuts.
			
			As is our habit in this paper, we won't construct Dedekind cuts only out of rationals, but of general dense streaks.
			
			The additional assumption in this subsection is that we can actually construct the cuts, so we postulate that $\optp(\QQ)$, and more generally $\optp(X)$, where $X$ is a streak we want to construct the cuts from, is actually a set.
			
			For a subset $A \subseteq X$ and $t \in X$ let $t + A$ denote, as usual, the set $\st{t + a}{a \in A}$.
			
			\begin{definition}
				Let $X$ be a streak.
				\begin{itemize}
					\item
						A subset $L \subseteq X$ is called a \df{lower cut} when
						\begin{itemize}
							\item
								$L$ is inhabited: $\xsome{a}{X}{a \in L}$,
							\item
								$L$ is a lower set: $\all[1]{a, b}{X}{(a \leq b \land b \in L) \implies a \in L}$,
							\item
								$L$ is upwards rounded: $\xall{a}{L}\xsome{b}{L}{a < b}$,
							\item
								$t + L$ is open and $(t + L)^C = X \setminus (t + L)$ is closed in $X$ for all $t \in X$.
						\end{itemize}
					\item
						Analogously, $U \subseteq X$ is called an \df{upper cut} when
						\begin{itemize}
							\item
								$U$ is inhabited: $\xsome{a}{X}{a \in U}$,
							\item
								$U$ is an upper set: $\all[1]{a, b}{X}{(a \leq b \land a \in U) \implies b \in U}$,
							\item
								$U$ is downwards rounded: $\xall{a}{U}\xsome{b}{U}{a > b}$,
							\item
								$t + U$ is open and $(t + U)^C = X \setminus (t + U)$ is closed in $X$ for all $t \in X$.
						\end{itemize}
					\item
						A pair $(L, U)$ is called a (\df{two-sided}) \df{Dedekind cut} when $L$ is a lower cut, $U$ is an upper cuts, and the two fit together in the following way:
						\begin{itemize}
							\item
								they are disjoint: $L \cap U = \emptyset$,
							\item
								the pair $(L, U)$ is \df{located}: $\all[1]{a, b}{X}{a < b \implies (a \in L \lor b \in U)}$.
						\end{itemize}
				\end{itemize}				
			\end{definition}
			
			The conditions for cuts are standard, except the ones having to do with the topology which are new. They say that the cuts are open, their complements closed, and the same holds for all their translates. Note however that if $X$ is a ring streak (as is usually the case --- the Dedekind reals are generally defined as Dedekind cuts on rationals), then it is enough to postulate the openness/closedness just for cuts/their complements themselves, not for translates. To see this, take any $t \in X$ and define a map $f\colon X \to X$, $f(x) = x - t$. Then for any $A \subseteq X$ we have $t + A = f^{-1}(A)$ and $(t + A)^C = f^{-1}(A^C)$ and by our assumptions on the intrinsic topology all maps are continuous (preimages of open subsets are open, preimages of closed subsets are closed).
			
			Denote the set of Dedekind cuts by
			\[\dc(X) \dfeq \st[1]{(L, U) \in \optp(X) \times \optp(X)}{(L, U) \text{ is a Dedekind cut}}.\]
			This set is interesting when we can actually embed $X$ into it (this is in general not the case: for extreme examples, consider $\dc(\NN) = \dc(\ZZ) = \emptyset$), specifically via the map
			\[\delta_X(a) \dfeq \big(X_{< a}, X_{> b}\big)\]
			which captures the intuition that the lower cut contains lower bounds, and the upper cut the upper bounds. We want $\delta_X$ to map into $\dc(X)$, and to satisfy the roundedness condition (as well as inhabitedness of the lower cut), we need to assume that $X$ is a dense streak. Of course, we also want $\dc(X)$ to be a streak, and $\delta_X$ a morphism. We verify this presently.
			
			So, let $X$ be a dense streak. For $(L', U'), (L'', U'') \in \dc(X)$ we define, as usual,
			\[(L', U') < (L'', U'') \dfeq U' \between L'' = \some{x}{X}{x \in U' \land x \in L''}\]
			whence $(L', U') \leq (L'', U'') = \xall{x}{X}{\lnot\lnot(x \notin U'' \lor x \notin L')}$.
			
			Suppose $(L', U') < (L'', U'')$ and $(L'', U'') < (L', U')$, \ie we may find elements $a \in U' \cap L''$ and $b \in U'' \cap L'$. It follows from the definition of a Dedekind cut, that every element of $L'$ must be smaller than any element in $U'$, so $b < a$, but the same applies for $L''$, $U''$, so $a < b$, which contradicts the assymmetry of $<$ in $X$. Thus $<$ is asymmetric in $\dc(X)$ as well.
			
			Suppose we now have $(L, U), (L', U'), (L'', U'') \in \dc(X)$, and $(L', U') < (L'', U'')$, \ie there is $a \in U' \cap L''$. We may find $b \in L''$, $a < b$, but then also $b \in U'$. We have $a \in L \lor b \in U$, the first disjunct being tantamount to $(L', U') < (L, U)$, and the second to $(L, U) < (L'', U'')$. Thus $<$ is cotransitive.
			
			Recall that tightness of $<$ (or of $\apart$) is equivalent to the antisymmetry of $\leq$. One may verify that $(L', U') \leq (L'', U'')$ is equivalent to $L' \subseteq L''$, as well as to $U' \supseteq U''$. From here, the antisymmetry is obvious.
			
			The addition of Dedekind cuts is defined as follows:
			\[(L', U') + (L'', U'') \dfeq \Big(\st[1]{x + y}{x \in L' \land y \in L''}, \st[1]{x + y}{x \in U' \land y \in U''}\Big) =\]
			\[= \Big(\st[1]{a \in X}{\some{x, y}{X}{x \in L' \land y \in L'' \land a < x + y}},\]
			\[\st[1]{b \in X}{\some{x, y}{X}{x \in U' \land y \in U'' \land x + y < b}}\Big).\]
			The first definition is more straightforward, but the second makes the proof (which we skip), that the sum is again a Dedekind cut, more direct.
			
			It is easy to see that $+$ is commutative, associative, and has $\delta_X(0)$ for a unit.
			
			The condition $\delta_X(0) < (L, U)$ means that $L_{> 0}$ is inhabited (which is also equivalent to $0 \in L$). For $(L', U'), (L'', U'') \in \dc(X)_{> \delta_X(0)}$ we define the multiplication by
			\[(L', U') \cdot (L'', U'') \dfeq \Big(\lowcl\st[1]{x \cdot y}{x \in L'_{> 0} \land y \in L''_{> 0}}, \st[1]{x \cdot y}{x \in U' \land y \in U''}\Big) =\]
			\[= \Big(\st[1]{a \in X}{\some{x, y}{X}{x \in L'_{> 0} \land y \in L''_{> 0} \land a < x \cdot y}},\]
			\[\st[1]{b \in X}{\some{x, y}{X}{x \in U' \land y \in U'' \land x \cdot y < b}}\Big)\]
			where $\lowcl{S}$ denotes the downward closure of $S \subseteq X$. Similarly as for addition, we can verify that $\cdot$ is commutative, associative, and the unit is $\delta_X(1)$. Also, multiplication distributes over addition, and both operations satisfy the conditions connecting them with $<$.
			
			For $n \in \NN$ one can quickly check that $n \cdot (L, U) = \delta_X(0)$ if $n = 0$, whereas for $n > 0$ we have $n \cdot (L, U) = \big(\st{n \cdot a}{a \in L}, \st{n \cdot b}{b \in U}\big)$.
			
			Take now $(A, Z), (B, W), (C, V), (D, U) \in \dc(X)$ such that $(B, W) < (D, U)$, \ie we have $x \in W \cap D$. Then there is also some $y \in D_{> x}$, thus also $y \in W \cap D$. Additionally, pick some $a \in Z$, $c \in C$. Since $X$ is archimedean, there exists $n \in \NN$ (we may assume $n > 0$) such that $a + n \cdot x < c + n \cdot y$.
			
			We prove that $(A, Z) + n \cdot (B, W) < (C, V) + n \cdot (D, U)$ by showing that $a + n \cdot x \in (Z + n \cdot W) \cap (C + n \cdot D)$ (in fact, one can see that the whole interval $\intcc[X]{a + n \cdot x}{c + n \cdot y}$ is contained in this intersection). The part $a + n \cdot x \in Z + n \cdot W$ is clear. Similarly $c + n \cdot y \in C + n \cdot D$, but $C + n \cdot D$ is a lower set, so it contains $a + n \cdot x$ as well. We conclude that $\dc(X)$ is archimedean.
			
			For $\dc(X)$ to be a streak what is still missing are the topological conditions. We make an ad-hoc definition that a streak $X$ is ``\df{good}'' when it is dense, the above defined relation $<$ on $\dc(X)$ is open, $\leq$ on $\dc(X)$ is closed, and the components in the above defined sum and product are open, and their complements are closed, also after translation. Thus if $X$ is a ``good'' streak, then $\dc(X)$ is a streak.
			
			Obviously in settings where we don't care about topology (that is, all subsets are taken as open and closed), all dense streaks are ``good''; a reader who cares just for this particular case, may freely skip forward to Theorem~\ref{Theorem: dedekind_reals}. For the rest we now set to show that all dense streaks are ``good'' in general. We start with the countable ones.
			
			\begin{lemma}\label{Lemma: dedekind_cuts_of_countable_streaks}
				Any countable dense streak is ``good''.
			\end{lemma}
			\begin{proof}
				We have to check that for a dense streak $X$, Dedekind cuts $(L', U'), (L'', U'') \in \dc(X)$ and $a, b, t \in X$ the following predicates are open:
				\begin{itemize}
					\item
						$\some[1]{x}{X}{x \in U' \land x \in L''}$,
					\item
						$\xsome{x}{X}\some[1]{y}{X}{x \in L' \land y \in L'' \land a < t + x + y}$,
					\item
						$\xsome{x}{X}\some[1]{y}{X}{x \in U' \land y \in U'' \land t + x + y < b}$,
					\item
						$\xsome{x}{X}\xsome{y}{X}\xsome{q}{\QQ_{> 0}}\xsome{r}{\QQ_{> 0}}\xsome{s}{\QQ}{}$\\
						$\phantom{x}\qquad \big(x \in L' \land q < x \land y \in L'' \land r < y \land s < t \land a < s + q \cdot r\big)$,
					\item
						$\xsome{x}{X}\xsome{y}{X}\xsome{q}{\QQ_{> 0}}\xsome{r}{\QQ_{> 0}}\xsome{s}{\QQ}{}$\\
						$\phantom{x}\qquad \big(x \in U' \land x < q \land y \in U'' \land y < r \land t < s \land s + q \cdot r < b\big)$,
				\end{itemize}
				and the following ones are closed:
				\begin{itemize}
					\item
						$\xall{x}{X}{\lnot\lnot\big(x \in U'^C \lor x \in L''^C\big)}$,
					\item
						$\xall{x}{X}\xall{y}{X}{\lnot\lnot\big(x \in L'^C \lor y \in L''^C \lor a \geq t + x + y\big)}$,
					\item
						$\xall{x}{X}\xall{y}{X}{\lnot\lnot\big(x \in U'^C \lor y \in U''^C \lor t + x + y \geq b\big)}$,
					\item
						$\xall{x}{X}\xall{y}{X}\xall{q}{\QQ_{> 0}}\xall{r}{\QQ_{> 0}}\xall{s}{\QQ}{}$\\
						$\phantom{x}\qquad \lnot\lnot\big(x \in L'^C \lor q \geq x \lor y \in L''^C \lor r \geq y \lor s \geq t \lor a \geq s + q \cdot r\big)$,
					\item
						$\xall{x}{X}\xall{y}{X}\xall{q}{\QQ_{> 0}}\xall{r}{\QQ_{> 0}}\xall{s}{\QQ}{}$\\
						$\phantom{x}\qquad \lnot\lnot\big(x \in U'^C \lor x \geq q \lor y \in U''^C \lor y \geq r \lor t \geq s \lor s + q \cdot r \geq b\big)$.
				\end{itemize}
				For a countable $X$ this is clearly the case.
				
				Note that we somewhat complicated the predicates dealing with the multiplication. The reason is that we don't want to quantify over $X_{> 0}$ since even if $X$ is countable, we don't know whether $X_{> 0}$ is. However, the term $x \cdot y$ is in general not defined on the whole $X$, so we cannot use it if we want a predicate on the whole $X$. The above rewrite works though because we do know that $\QQ_{> 0}$ is countable.
			\end{proof}
			
			\begin{lemma}\label{Lemma: dedekind_cuts_of_separable_streaks}
				If a streak $X$ has a countable dense substreak $S \subseteq X$, then it is ``good''.
			\end{lemma}
			\begin{proof}
				Of course, if $X$ has a dense substreak, it is dense itself. Define the maps $f\colon \dc(S) \to \dc(X)$, $g\colon \dc(X) \to \dc(S)$ by
				\[f(L, U) \dfeq \big(\st{a \in X}{\xsome{x}{L}{a < x}}, \st{b \in X}{\xsome{y}{U}{y < b}}\big),\]
				\[g(L, U) \dfeq \big(L \cap S, U \cap S\big)\]
				That the maps $f$ and $g$ are well defined (they map Dedekind cuts to Dedekind cuts) is easy to check; we mention merely that inhabitedness of $L \cap S$, $U \cap S$ follows from inhabitedness of $L$, $U$ due to the density of $S$ and the archimedean property, and that the resulting cuts and their translates are open: for $t \in S$ the sets $t + L \cap S$, $t + U \cap S$ are preimages of $t + L$, $t + U$ via the inclusion map $S \hookrightarrow X$, and the conditions $\xsome{x}{L}{a < t + x}$, $\xsome{y}{U}{y < t + b}$ can be equivalently restated as $\some{x}{S}{x \in L \land a < t + x}$, $\some{y}{S}{y \in U \land y < t + b}$ which are open because $S$ is countable. Similarly we can see that the complements are closed.
				
				It follows easily from the conditions for Dedekind cuts that $f$ and $g$ are mutually inverse (for $f \circ g = \id[\dc(X)]$ use also the density of $S$), meaning that $\dc(S)$ and $\dc(X)$ are in bijective correspondence. One can also check that $f$ and $g$ preserve the order relations $<$, $\leq$ and the algebraic operations $+$, $\cdot$. Since isomorphic sets have isomorphic topologies (functors preserve isomorphisms), it follows that if $\dc(S)$ is a streak, then so is $\dc(X)$. But we know from Lemma~\ref{Lemma: dedekind_cuts_of_countable_streaks} that $S$ is ``good'', so $X$ is ``good'' also.
			\end{proof}
			
			\begin{lemma}\label{Lemma: dedekind_cuts_of_substreaks_of_ring_streaks}
				Let $X$ be a dense streak. Then if $\ring(X)$ is ``good'', so is $X$.
			\end{lemma}
			\begin{proof}
				Recall all the notation from Subsection~\ref{Subsection: ring_streaks}. Define $f\colon \dc(X) \to \dc(\ring(X))$, $g\colon \dc(\ring(X)) \to \dc(X)$ by
				\[f(L, U) \dfeq \big(\st{[(a, b)] \in \ring(X)}{a \in b + L}, \st{[(a, b)] \in \ring(X)}{a \in b + U}\big),\]
				\[g(L, U) \dfeq \big(\rho^{-1}(L), \rho^{-1}(U)\big).\]
				Note that $f$ is well defined, for if $[(a, b)] = [(a', b')]$, \ie $a + b' = a' + b$, and if $a \in b + L$, \ie there exists $x \in L$ such that $a = b + x$, then $a + a' = b + x + a' = a + b' + x$, \ie $a' = b' + x$ (and similarly for $U$). Since $b + L$ is open, $a \in b + L$ is an open predicate, so $\st{(a, b) \in FD(X)}{a \in b + L}$ is open in $FD(X)$ and therefore $\st{[(a, b)] \in \ring(X)}{a \in b + L}$ is open in the quotient $\ring(X)$ (similarly for $U$ and the closedness of complements).
				
				We leave the verification that $f(L, U)$, $g(L, U)$ are Dedekind cuts, that $f$ and $g$ are inverse and that they preserve the order and algebraic structure as an exercise. In any case, $\dc(X) \ism \dc(\ring(X))$, so if $\ring(X)$ is ``good'', so is $X$.
			\end{proof}
			
			\begin{theorem}
				For any dense streak $X$ the set of its Dedekind cuts $\dc(X)$ is a streak (for the order and algebra, defined above).
			\end{theorem}
			\begin{proof}
				We are saying that all dense streaks are ``good''. This holds for countable dense streaks by Lemma~\ref{Lemma: dedekind_cuts_of_countable_streaks}, and Lemma~\ref{Lemma: dedekind_cuts_of_separable_streaks} generalizes this to streaks with a countable dense substreak. By Lemma~\ref{Lemma: countable_dense_substreak} every multiplicative streak, in particular every ring streak, has a countable dense substreak. Thus $\ring(X)$ is ``good'', and by Lemma~\ref{Lemma: dedekind_cuts_of_substreaks_of_ring_streaks} $X$ is ``good'' as well.
			\end{proof}
			
			\begin{theorem}\label{Theorem: dedekind_reals}
				For any dense streak $X$ the streak $\dc(X)$ is terminal, thus a model of the reals, as per Definition~\ref{Definition: reals}.
			\end{theorem}
			\begin{proof}
				To prove that $\dc(X)$ is terminal, we have to, for an arbitrary streak $Y$, construct a morphism $f\colon Y \to \dc(X)$ (its uniqueness is guaranteed by the fact that $\Str$ is a preorder category). Let
				\[f(y) \dfeq \big(X_{< y}, X_{> y}\big) = \big(\st{x \in X}{x < y}, \st{x \in X}{x > y}\big).\]
				It is easy to check that this is a Dedekind cut if $X$ is dense, so $f$ is well defined. Checking that $f$ is a morphism is also immediate.
			\end{proof}
			
			In particular, the Dedekind reals $\dc(\QQ)$, constructed, as is usual, from rationals, are a model of reals, by our definition.
			
			The fact that $\dc(X)$ is a terminal streak gives us in particular an embedding $X \to \dc(X)$. Together with this, $\dc$ is (no surprise there) a reflection from dense streaks to (in the constructive sense) Dedekind complete streaks. However, viewing this as a reflection is in this case rather uninteresting, as the image of $\dc$ contains, up to isomorphism, only one element. That is, there is (up to isomorphism) only one Dedekind complete streak, and that's the terminal (the largest) one.
			
			We made an effort to always check the necessary openness and closedness conditions, but in practice this is often a non-issue. Not only do we in a lot of cases not care about the intrinsic topology (that is, we declare all subsets open and closed), even when we do, it can easily happen that \emph{all} cuts (and their translates) are automatically open (and their complements closed). Denote the set of ``Dedekind cuts without topological conditions'' (\ie the way the Dedekind reals are usually defined) by $\RR_D$. Then we have inclusions
			$CC(\QQ) \hookrightarrow \dc(\QQ) \hookrightarrow \RR_D$ (the first inclusion exists because $\dc(\QQ)$ is terminal). The fact that Cauchy reals embed into Dedekind reals is well known --- in fact, if countable choice holds, then they are the same, in which case also $\dc(\QQ) = \RR_D$. That is, in the presence of countable choice, ``every Dedekind cut is open''.
			
			When countable choice doesn't hold, then (as already mentioned) Cauchy reals need not be Cauchy complete. In that case it is interesting to ask, what the smallest Cauchy complete field is (the so-called \df{euclidean reals}~\cite{Escardo_M_Simpson_A_2001:_a_universal_characterization_of_the_closed_euclidean_interval}). We see that our theory provides an upper bound for it, namely $\dc(\QQ)$ (by the results of the previous subsection, a terminal streak is Cauchy complete). It would be interesting to see, whether this bound is exact, when we take the smallest possible intrinsic topology (such that decidable subsets are both open and closed, and we have the suitable closure properties for finite/countable unions/intersections of open/closed subsets).
			
			A few more words about the structure of Dedekind reals. Above we only had to check that they are a (terminal) streak, but we know (from Theorem~\ref{Theorem: structure_of_reals}) that it must also have lattice structure, total multiplication and inverses of nonzero elements. We can actually use this theory to derive formulae for all of this structure in $\dc(X)$.
			
			For example, given any Dedekind cuts $(L', U'), (L'', U'') \in \dc(X)$, we have
			\[\inf\big\{(L', U'), (L'', U'')\big\} = \iota_{\dc(X)}^{-1}\Big(\iota_{\dc(X)}\big(\inf\big\{(L', U'), (L'', U'')\big\}\big)\Big) =\]
			\[= \trm[\dc(X)^\land]\Big(\big[\big\{(L', U'), (L'', U'')\big\}\big]\Big) = \Big(X_{< [\{(L', U'), (L'', U'')\}]}, X_{> [\{(L', U'), (L'', U'')\}]}\Big).\]
			For $x \in X$ we have
			\[x < \big[\big\{(L', U'), (L'', U'')\big\}\big] \iff \some{q}{\QQ}{x < q \land q < \big[\big\{(L', U'), (L'', U'')\big\}\big]} \iff\]
			\[\iff \some{q}{\QQ}{x < q \land q < (L', U') \land q < (L'', U'')} \iff\]
			\[\iff x < (L', U') \land x < (L'', U'') \iff x \in L' \land x \in L'' \iff x \in L' \cap L''\]
			and similarly
			\[x > \big[\big\{(L', U'), (L'', U'')\big\}\big] \iff \some{q}{\QQ}{x > q \land q > \big[\big\{(L', U'), (L'', U'')\big\}\big]} \iff\]
			\[\iff \some{q}{\QQ}{x > q \land \big(q > (L', U') \lor q > (L'', U'')\big)} \iff\]
			\[\iff \some{q}{\QQ}{x > q \land q > (L', U')} \lor \some{q}{\QQ}{x > q \land q > (L'', U'')} \iff\]
			\[\iff x > (L', U') \lor x > (L'', U'') \iff x \in U' \lor x \in U'' \iff x \in U' \cup U'';\]
			thus $\inf\big\{(L', U'), (L'', U'')\big\} = \big(L' \cap L'', U' \cup U''\big)$.
			
			Of course, it is easy to guess straight from the definition of Dedekind cuts that infima are calculated this way (and suprema are given by $\sup\big\{(L', U'), (L'', U'')\big\} = \big(L' \cup L'', U' \cap U''\big)$). More interesting is the total multiplication on cuts which is in classical mathematics usually given by (nine) cases, depending on whether each of the factors is positive, negative or zero. Splitting the cases isn't an available method constructively. Following the theory from Subsection~\ref{Subsection: ring_streaks} we can provide a constructive formula for the multiplication of Dedekind cuts.
			
			While we can do this in general, the formula simplifies if $X$ is a ring streak (which is usually the case; after all, typically we are making cuts on the rationals), so we'll assume that. We leave it as an exercise to the reader that $-(L, U) = (-U, -L)$.
			
			Take any $(L', U'), (L'', U'') \in \dc(X)$. The lower cuts $L'$, $L''$ are inhabited by assumption, and since they are also lower sets, they must be inhabited by some negative integer. So take some $m, n \in \NN_{> 0}$ such that $-m \in L'$ and $-n \in L''$, meaning that $(L', U') + m > 0$ and $(L'', U'') + n > 0$. We then have
			\[(L', U') \cdot (L'', U'') = \big((L', U') + m\big) \cdot \big((L'', U'') + n\big) - m \cdot (L'', U'') - n \cdot (L', U') - m n =\]
			\[= (L' + m, U' + m) \cdot (L'' + n, U'' + n) - \big(m \cdot L'' + n \cdot L' + m n, m \cdot U'' + n \cdot U' + m n\big) =\]
			\[= \big(\lowcl(L' + m)_{> 0} \cdot (L'' + n)_{> 0}, (U' + m) \cdot (U'' + n)\big) + \big(-m \cdot U'' - n \cdot U' - m n, -m \cdot L'' - n \cdot L' - m n\big).\]
			The downarrow isn't actually required any more after we add the second summand, so in conclusion we obtain
			\[(L', U') \cdot (L'', U'') =\]
			\[= \Big((L' + m)_{> 0} \cdot (L'' + n)_{> 0} - m \cdot U'' - n \cdot U' - m n, (U' + m) \cdot (U'' + n) - m \cdot L'' - n \cdot L' - m n\Big).\]
		
		\subsection{Reals via the interval domain}\label{Subsection: interval_domain}
		
			The idea for the next construction of reals is that an individual point can be given as a collection of its neighbourhoods. Specifically, a real can be determined by listing all the intervals with rational endpoints which contain it.
			
			Such an interval can be given simply as a pair of rationals (its endpoints), the first component smaller than the second. The smaller the interval, the more information we have, where the real in the consideration lies. This ``information ordering'' is a special kind of partial order, called a \df{domain}~\cite{Gierz_G_Hoffmann_KH_Keimel_K_Lawson_JD_Mislove_MW_Scott_DS_2003:_continuous_lattices_and_domains} (though we won't use this explicitly in this paper), and the set of rational intervals is therefore called the \df{interval domain}. It turns out that the collections of rational interval neighbourhoods of reals are precisely the maximal ideals in the interval domain. Thus the set of these maximal ideals is another model of reals.
			
			We get a model of reals regardless whether we consider a pair of rationals to represent an open or a closed interval (of course, the definition of the orders and operations is slightly different between the two cases). However, the version with open intervals is essentially just a restatement of the Dedekind construction: a Dedekind real $(L, U)$ is represented by a maximal ideal $L \times U$, and to go in the other direction, just take the images of both projections.
			
			To make things different from the previous subsection, we present here the construction of reals via closed intervals. The price for this, however, is that openness of $<$ (and closedness of $\leq$) does not follow from the construction. Thus we assume in this subsection that all subsets are open and closed, \ie $\optp(X) = \pst(X) = \cltp(X)$ for all $X$. Moreover, while we assumed in the previous subsection only for $\optp(X)$ to be sets, in this subsection this amounts to the assumption that we have powersets in general, \ie $\pst(X)$ are sets.
			
			Under these assumptions we present the construction of the interval domain. As usual, we adopt the construction to general streaks, veering slightly away from the standard definitions.
			
			Let $X$ be a streak. We define $ID(X) \dfeq \st{(a, b) \in X \times X}{a \leq b}$ and an embedding $\psi_X\colon X \to ID(X)$, $\psi_X(x) \dfeq (x, x)$. Intuitively, a pair $(a, b) \in ID(X)$ represents the closed interval $\intcc[X]{a}{b}$, and a point $x \in X$ can be viewed as the degenerate interval $\intcc[X]{x}{x}$.
			
			We can define order and operations on $ID(X)$ which have several properties of those from streaks, though not all of them. We start by defining for $(a, b), (c, d) \in ID(X)$
			\[(a, b) < (c, d) \dfeq b < c.\]
			Clearly this relation is asymmetric: if $(a, b) < (c, d) < (a, b)$, then $b < c \leq d < a \leq b$, contradicting assymmetry of $<$ in $X$. However, it is not cotransitive; for example, we have $(0, 2) < (3, 5)$, but neither $(0, 2) < (1, 4)$ nor $(1, 4) < (3, 5)$.
			
			We can define $\leq$, $\apart$ and $\nap$ in the usual way, but they have only some of the usual properties. For example, $(a, b) \nap (c, d)$ is equivalent to stating that the intervals $\intcc[X]{a}{b}$ and $\intcc[X]{c}{d}$ intersect; this is a reflexive and symmetric relation, but not a transitive (and hence not an equivalence) one (in fact, its transitive hull is the total relation on $ID(X)$). Moreover, since $\nap$ is not the equality, we see that $<$ is also not tight.
			
			We define addition componentwise:
			\[(a, b) + (c, d) \dfeq (a + c, b + d).\]
			Clearly this operation is commutative, associative, and has $\psi_X(0)$ for the unit (it is, of course, also well defined: if $a \leq b$ and $c \leq d$, then $a + c \leq b + d$).
			
			For $(x, y) \in ID(X)$ it additionally holds $(a, b) + (x, y) < (c, d) + (x, y) \implies (a, b) < (c, d)$ since $b + y < c + x$ implies $b + x < c + x$ which implies $b < c$. However, the implication in the other direction does not hold in general: for example $(0, 1) < (2, 3)$, but not $(0, 1) + (0, 1) < (2, 3) + (0, 1)$. Still, it does hold if $(x, y)$ is in the image of $\psi_X$, \ie if $x = y$.
			
			By definition $\psi_X(0) < (a, b)$ when $0 < a$ (and therefore also $0 < b$). For $(a, b), (c, d) \in ID(X)_{> \psi_X(0)}$ we define
			\[(a, b) \cdot (c, d) \dfeq (a \cdot c, b \cdot d).\]
			Since multiplication is defined componentwise also, it is commutative, associative, and distributes over addition. Likewise, for $(x, y) \in ID(X)_{> \psi_X(0)}$ the implication $(a, b) \cdot (x, y) < (c, d) \cdot (x, y) \implies (a, b) < (c, d)$ holds, but not the implication in the other direction.
			
			For $n \in \NN$ we see that $n \cdot (a, b) = (n \cdot a, n \cdot b)$. Given $(a, a'), (b, b'), (c, c'), (d, d') \in ID(X)$ such that $(b, b') < (d, d')$ (that is, $b' < d$), we may use the archimedean property of $X$ to find $n \in \NN$, for which $a' + n \cdot b' < c + n \cdot d$, meaning $(a, a') + n \cdot (b, b') < (c, c') + n \cdot (d, d')$. That is, $ID(X)$ satisfies the archimedean property.
			
			Obviously $\psi_X$ preserves the order and operations (in this sense it could be called a morphism, except that its codomain is not entirely a (pre)streak).
			
			We define the usual \df{information ordering} on $ID(X)$ by
			\[(a, b) \inford (c, d) \dfeq a \leq c \land b \geq d.\]
			Normally, we would now define what an ideal in this ordering is, and when it is maximal, but we will take a different approach. Define directly
			\[IDR(X) \dfeq \Big\{{S \in \pst(ID(X))}\ \Big|\ \xall{a}{X}\xall{b}{X_{> a}}\some{(x, y)}{S}{a + y < b + x} \land\]
			\[\land\ {\all[2]{(a, b)}{ID(X)}{(a, b) \in S \iff \xall{(x, y)}{S}{(a, b) \nap (x, y)}}}\Big\}.\]
			The first condition says that (at least when $X$ is dense) there are arbitrary small intervals in $S$ (in particular, $S$ is inhabited). This actually follows from the second condition in classical mathematics, but constructively we have to assume it.
			
			Before we check that this works, let us see that the elements of $IDR(X)$ essentially are maximal ideals.
			
			As already mentioned, they are inhabited, and if $(a, b) \in S$, then any $(c, d) \in ID(X)_{\inford (a, b)}$ is in $S$ as well: if a smaller interval intersects all (intervals, represented by the) elements of $S$, then the larger one must also.
			
			Next, we'd need to check that an intersection of two intervals in $S$ is again in $S$. There is a minor problem here: the intersection of $(a, b), (c, d) \in S$ is $(\sup\{a, c\}, \inf\{b, d\})$ which in general makes sense only if $X$ is a lattice. In this case, since $(a, b) \nap (c, d)$, we have $\sup\{a, c\} \leq \inf\{b, d\}$, and one easily sees that if both $(a, b)$, $(c, d)$ intersect every element in $S$, then their intersection must also.
			
			This means that every $S \in IDR(X)$ is an ideal, insofar $X$ is a lattice streak. However, our definition of $IDR(X)$ does not refer to suprema and infima, so it works in greater generality (not that this is a major issue, given that we usually take $X = \QQ$ which is a lattice).
			
			Finally, the maximality (in the classical sense) of $S$ can be seen as follows: if we added any element $(z, w) \in ID(X)$ which is not yet in $S$, it would mean that there is $(a, b) \in S$ such that $\intcc[X]{a}{b}$ has an empty intersection with $\intcc[X]{z}{w}$. That would mean that the empty set is in $S$, but we cannot represent the empty set as a closed interval with the left bound no greater than the right bound (and even if we added a special symbol for the top element in $ID(X)$, an ideal containing it would have to be the whole of $ID(X)$, contradicting the definition of maximality).
			
			In conclusion, the condition for elements of $IDR(X)$ is the constructive definition of a maximal ideal in $ID(X)$ for general streaks $X$.
			
			We claim that $IDR(X)$ is a streak if $X$ is dense. The order and operations are inherited from $ID(X)$ in the following way. For $S, T \in IDR(X)$ define
			\[S < T \dfeq \xsome{(a, b)}{S}\xsome{(c, d)}{T}{(a, b) < (c, d)}.\]
			Asymmetry follows from asymmetry of $<$ on $ID(X)$, but cotransitivity is new. Suppose we have $S, T, Z \in IDR(X)$ and $S < T$, that is, some $(a, b) \in S$, $(c, d) \in T$ with $(a, b) < (c, d)$, \ie $b < c$. Since $X$ is dense, we can find $b', c' \in X$ such that $b < b' < c' < c$. By the definition of $IDR(X)$ there exists $(x, y) \in Z$ such that $b' + y < c' + x$. By cotransitivity in $X$ we have $b < x \lor x < b'$ and $c' < y \lor y < c$. If the first disjunct of the first disjunction holds, then $(a, b) < (x, y)$, so $S < Z$, and we are done. Similarly, if the second disjunct in the second disjunction holds, we have $(x, y) < (c, d)$, so $Z < T$. The only remaining possibility is that both $x < b'$ and $c' < y$ hold, but this cannot happen, as it leads to the contradiction $b' + y < c' + x < y + b'$.
			
			As for tightness, suppose we have $S, T \in IDR(X)$, $S \leq T$, $T \leq S$. This means that for all $(a, b) \in S$ and $(c, d) \in T$ we have $(a, b) \leq (c, d)$ and $(c, d) \leq (a, b)$, that is $(a, b) \nap (c, d)$. By the definition of $IDR(X)$ this means that every element in $S$ belongs also to $T$ and vice versa, \ie $S = T$.
			
			The addition in $IDR(X)$ is defined by
			\[S + T \dfeq \st{(a, b) + (c, d)}{(a, b) \in S \land (c, d) \in T}.\]
			Note that the zero element is given by $X_{\leq 0} \times X_{\geq 0} = \st{(a, b) \in ID(X)}{a \leq 0 \leq b}$. For $S \in IDR(X)$ we have $X_{\leq 0} \times X_{\geq 0} < S$ if and only if there exist $(a, b) \in S$ such that $a$ (and therefore also $b$) is positive. We define for $S, T \in IDR(X)_{> X_{\leq 0} \times X_{\geq 0}}$
			\[S \cdot T \dfeq \lowcl\st{(a, b) \cdot (c, d)}{(a, b) \in S_{> \psi_X(0)} \land (c, d) \in T_{> \psi_X(0)}}\]
			where $\lowcl$ denotes the downward closure of the set with regard to the information ordering in $ID(X)$. The unit for this multiplication is $X_{\leq 1} \times X_{\geq 1}$. We skip the verification that $+$ and $\cdot$ satisfy the usual requirements.
			
			It is easy to see that for $n \in \NN$ and $S \in IDR(X)$ we have $n \cdot S = \st{n \cdot (a, b)}{(a, b) \in S}$. Take $A, B, C, D \in IDR(X)$ such that $B < D$, so we have $(b, b') \in B$ and $(d, d') \in D$ with $(b, b') < (d, d')$. Take any $(a, a') \in A$ and $(c, c') \in C$ and use the archimedean property of $ID(X)$ to find $n \in \NN$ with $(a, a') + n \cdot (b, b') < (c, c') + n \cdot (d, d')$. Hence $A + n \cdot B < C + n \cdot D$, so $IDR(X)$ is archimedean.
			
			We conclude that $IDR(X)$ is a streak if $X$ is dense, but of course we want to see that it is a terminal one.
			\begin{theorem}
				Let $X$ be a dense streak. Then $IDR(X)$ is a terminal streak, \ie a model of reals according to Definition~\ref{Definition: reals}.
			\end{theorem}
			\begin{proof}
				Let $Y$ be any streak. We define a map $f\colon Y \to IDR(X)$ by
				\[f(y) \dfeq X_{\leq y} \times X_{\geq y}.\]
				It is easy to see that this $f$ is a (necessarily unique) morphism $Y \to IDR(X)$.
			\end{proof}
		
		\subsection{Reals as the formal space/locale/topological space}\label{Subsection: topological_models_of_reals}
		
			In previous subsections we have always assumed some additional conditions (such as countable choice in the case of Cauchy reals) to be able to construct the models of reals in question. One such condition was the existence of powersets in the previous subsection. In \df{predicative mathematics} this assumption is considered too strong; a colection of \emph{all} subsets of a given set is in general a proper class (called a \df{powerclass}) rather than a set.
			
			Since in predicative mathematics already $\pst(\one)$ is problematic (the powerclass of a singleton can be seen as a collection of all truth values, and this might not be a set), and the only topology in the classical sense (\ie arbitrary unions of opens are open) on a singleton is $\pst(\one)$, it follows that already the assumption in Subsection~\ref{Subsection: Dedekind_reals}, that topologies are sets, is too strong for a predicative mathematician.
			
			Since in this paper we assume only for countable unions of opens to be again open, it can happen that topologies are sets even if powerclasses aren't (such as in ASD). However, predicative mathematicians have developed tools~\cite{Coquand_T_Sambin_G_Smith_JM_Valentini_S_2003:_inductively_generated_formal_topologies} to deal even with class-sized topologies (closed under arbitrary unions), as long as they have a basis which is a set.
			
			Topologies, given purely via a (set-sized) base, are called \df{formal spaces}. A formal space is given as some set (which intuitively represents the set of basic opens, though its actual elements can be anything), together with a relation which tells us, to which extent these ``basic opens'' cover each other. We do not have ``the set of points of the underlying set of the topological space'' in general, hence this approach is also called \df{pointfree topology}.
			
			In this subsection we recall the construction of the formal space of reals and show that it satisfies our definition of reals. At the end we also observe that this result for formal spaces easily implies the same result in cases, when the topological space of reals is given as a locale, or as a classical topological space. In each of the three cases we see, that for $\RR$ to be a terminal streak in a given setting, it must have the usual euclidean topology.
			
			Note however that the category of formal spaces (or of locales, or of topological spaces) doesn't satisfy all of our assumptions on the setting (recall Section~\ref{Section: setting}) --- its logical structure is too weak. One way to deal with this is to ignore the problem and simply rely on a background set and class theory, which we do here. A more formal way would be to embed the (subcategory of small) formal spaces (or locales, or topological spaces) into a category with richer structure, such as sheaves. An interesting question however is --- is this even necessary? Can the definitions, theorems and proofs in this paper be phrased in such a way that even the weak logic of the categories in this subsection is sufficient? My suspicion is that the answer is positive, at least for the vast majority of the results, and so the assumptions on the setting could be weakened further, but this is something that still needs to be done.
			
			Anyway, back to formal spaces.
			\begin{definition}
				A \df{formal space} is a tuple $X = (X^*, \land_X, \top_X, \covby_X, \pospred{X})$ where $X^*$ is a set, $\top_X \in X^*$ an element of it, $\land_X\colon X^* \times X^* \to X^*$ a binary operation, $\covby_X \subseteq X^* \times \pst(X^*)$ a (class-sized) binary relation between elements of $X^*$ and subsets of $X^*$, and $\pospred{X} \subseteq X^*$ a unary relation on $X^*$, such that the following conditions are fulfilled for all $a, b \in X^*$, $V, W \subseteq X^*$.
				\begin{itemize}
					\item
						$a \in V \implies a \covby_X V$
					\item
						$a \covby_X \{b\} \land b \covby_X \{a\} \implies a = b$
					\item
						$a \covby_X V \land \xall{x}{V}{x \covby_X W} \implies a \covby_X W$
					\item
						$a \covby_X V \lor b \covby_X V \implies a \land_X b \covby_X V$
					\item
						$a \covby_X V \land a \covby_X W \implies a \covby_X \st{x \land_X y}{x \in V \land y \in W}$
					\item
						$a \covby_X \{\top_X\}$
					\item
						$\pospred[a]{X} \land a \covby_X V \implies \xsome{x}{V}{\pospred[x]{X}}$
					\item
						$\big(\pospred[a]{X} \implies a \covby_X V\big) \implies a \covby_X V$
				\end{itemize}
				We call $\covby_X$ the \df{covering relation} (intuitively, it tells, when a basic open is covered by the union of a family of basic opens), and $\pospred{X}$ the \df{positivity predicate} (intuitively, it tells, when a basic open is inhabited).
			\end{definition}
			
			It follows from these conditions that $(X^*, \land_X, \top_X)$ is a bounded (\ie $\top_X$ is the top element) $\land_X$-semilattice. Thus $\leq_X$, defined by $a \leq_X b \dfeq a \land_X b = a$, is a partial order on $X^*$. Note that $a \leq_X b$ is equivalent to $a \covby_X \{b\}$.
			
			We may pass with $\land_X$, $\covby_X$, $\pospred{X}$ to $\pst(X^*)$ (we'll use the same symbols; it should be clear, in which sense they are meant) by defining:
			\[U \land_X V \dfeq \st{x \land_X y}{x \in U \land y \in V},\]
			\[U \covby_X V \dfeq \xall{x}{U}{x \covby_X V},\]
			\[\pospred[U]{X} \dfeq \xsome{x}{U}{\pospred[x]{X}}.\]
			This enables us to state some of the above conditions for formal spaces in a simpler way:
			\begin{itemize}
				\item
					$a \leq_X b \land b \leq_X a \implies a = b$ \quad (antisymmetry of $\leq_X$),
				\item
					$U \covby_X V \land V \covby_X W \implies U \covby_X W$ \quad (transitivity of $\covby_X$),
				\item
					$U \covby_X V \lor W \covby_X V \implies U \land_X W \covby_X V$,
				\item
					$U \covby_X V \land U \covby_X W \implies U \covby_X V \land_X W$,
				\item
					$\pospred[U]{X} \land U \covby_X V \implies \pospred[V]{X}$
			\end{itemize}
			\etc
			
			Collections of basic opens $U \subseteq X^*$ can be seen as representing the "unions" of these basic opens, that is, arbitrary opens. It then makes sense to identify $U, V \subseteq X^*$ when they "represent the same open", that is, when they cover each other:
			\[U \equ_X V \dfeq U \covby_X V \land V \covby_X U.\]
			This is easily seen to be an equivalence relation, so we may define the (generally class-sized) \df{topology} of a formal space by
			\[\optp(X) \dfeq \pst(X^*)/_{\equ_X}.\]
			
			
			The relation $\covby_X$ on $\pst(X^*)$ induces a partial order $\covby_X$ on $\optp(X^*)$. The smallest element herein is $[\emptyset]$, the largest $[X^*]$, binary meets are given by $[U] \land_X [V] = [U \land_X V]$ and arbitrary joins by $\bigvee_{i \in I} [U_i] = [\bigcup_{i \in I} U_i]$ (by a slight abuse of notation, we use the same symbols for the structure of $\optp(X)$ as for $X$ itself).
			
			A \df{continuous map} or a \df{morphism of formal spaces} $f\colon X \to Y$ is given by a (generally class-sized) map $\optp(f)\colon \optp(Y) \to \optp(X)$ which preserves all finite meets and arbitrary joins (intuitively, $\optp(f)$ the preimage map). Often such a map is represented by its ``restriction to basic opens'' $f^*\colon Y^* \to \pst(X^*)$ with properties
			\begin{itemize}
				\item
					$f^*(a \land_Y b) \equ_X f^*(a) \land_X f^*(b)$, $f^*(\top_Y) \equ_X \{\top_X\}$,
				\item
					$a \covby_Y V \implies f^*(a) \covby_X \bigcup_{b \in V} f^*(b)$,
				\item
					$\pospred[f^*(a)]{X} \implies \pospred[a]{Y}$
			\end{itemize}
			for all $a, b \in Y^*$, $V \subseteq Y^*$.
			
			The formal spaces, together with their morphisms, form a category $\FormSp$. Technically elements of $\optp(X)$ should be certain subobjects (equivalence classes of monomorphisms) of $X$ in this category, but it is a lot more convenient to give $\optp(X)$ as above. However, we then have to make the connection between elements of $\optp(S)$ and open subobjects of $S$ explicit. The following is essentialy a restatement of forming open sublocales in the way, which refers just to basic opens.
			
			Given $[U] \in \optp(S)$, define
			\[[U]^* \dfeq \st{a \in X^*}{a \covby_X U} + \{\top_{[U]}\},\]
			with
			\begin{itemize}
				\item
					$\covby_{[U]}$ and $\pospred{[U]}$ in $[U]^*$ defined as restrictions of these relations from $X^*$, and additionally
				\item
					$a \covby_{[U]} V$ for all $a \in [U]^*$, $\top_{[U]} \in V \subseteq [U]^*$,
				\item
					$\pospred[\top_{[U]}]{[U]} \dfeq \pospred[U]{X}$,
				\item
					$\land_{[U]}$ in $[U]^*$ is the restriction of this operation from $X^*$, with $\top_{[U]}$ acting as the top element.
			\end{itemize}
			
			It is easily seen that this is a well-defined formal space. To realize it as a subobject of $X$, we need to present a monomorphism $\optp(u)\colon \optp(X) \to \optp([U])$ which we define by $\optp(u)([V]) \dfeq [U \land_X V]$ (we leave it as an exercise to show that this is a well-defined morphism). Rather than checking directly that this defines a mono, we show that it has a right inverse (thus making it a split epi on the level of topologies, which, as it is known, implies that, as a morphism of formal spaces, it is monic) --- intuitively, this means that open subobjects have the subspace topology. We define $\dot{u}\colon U^* \to X^*$ by $\dot{u}(a) \dfeq a$ for $a \in X^* \cap [U]^*$ and $\dot{u}(\top_{[U]}) \dfeq \top_X$. The extension of this map to $\optp([U]^*) \to \optp(X^*)$ is a right inverse to $\optp(u)$.
			
			If $[U]^* = [V]^*$, then $[U]^* \cap X^* = [V]^* \cap X^*$, so $U \covby_X V$ and $V \covby_X U$, thus $[U] = [V]$. This means that the subobjects, which have the form as above, are in bijective correspondence with elements of $\optp(X^*)$, meaning that $\optp(X^*)$, as originally defined, is a suitable way to present topology. Any mono (isomorphic to one) which is obtained in this way will be called \df{open}.
			
			We won't bother explicating the closed subobjects, beyond noting that they are defined as complements of opens, and thus satisfy all the requirements that we'll need.
			
			We need to say a little bit about products in $\FormSp$, though. Actually we only need to recall that for
			\[\big((X \times Y)^*, \land_{X \times Y}, \top_{X \times Y}, \covby_{X \times Y}, \pospred{X \times Y}\big) =\]
			\[= \big(X^*, \land_X, \top_X, \covby_X, \pospred{X}\big) \times \big(Y^*, \land_Y, \top_Y, \covby_Y, \pospred{Y}\big)\]
			we have
			\begin{itemize}
				\item
					$(X \times Y)^* = X^* \times Y^*$,
				\item
					$(a, b) \land_{X \times Y} (c, d) = (a \land_X c, b \land_Y d)$,
				\item
					$\top_{X \times Y} = (\top_X, \top_Y)$,
				\item
					$\pospred[(a, b)]{X \times Y} = \pospred[a]{X} \land \pospred[b]{Y}$,
			\end{itemize}
			and if $a \covby_X V$ and $b \covby_Y W$, then $(a, b) \covby_{X \times Y} V \times W$.
			
			Consider now the standard number sets in $\FormSp$. Let $\NN$, $\ZZ$, $\QQ$ denote the usual (set-theoretic) number sets. Constructing their formal versions $\formal{N}$, $\formal{Z}$, $\formal{Q}$ amounts to equipping them with the discrete topology --- that is, basic opens are singletons, which we represent by their unique element. For example, $\formal{Q} = \big(\QQ + \{\bot_\formal{Q}, \top_\formal{Q}\}, \land_\formal{Q}, \top_\formal{Q}, \covby_\formal{Q}, \pospred{\formal{Q}}\big)$ where for $q, r \in \QQ$, $V \subseteq \formal{Q}^* \dfeq \QQ + \{\bot_\formal{Q}, \top_\formal{Q}\}$
			\[q \land r \dfeq \begin{cases} q & q = r\\ \bot_\formal{Q} & q \neq r \end{cases} \qquad\qquad q \covby_\formal{Q} V \dfeq q \in V \lor \top_\formal{Q} \in V\]
			and for $q \in \formal{Q}^*$, $V \subseteq \formal{Q}^*$
			\[\bot_\formal{Q} \land_\formal{Q} q = q \land_\formal{Q} \bot_\formal{Q} \dfeq \bot_\formal{Q},
				\qquad \top_\formal{Q} \land_\formal{Q} q = q \land_\formal{Q} \top_\formal{Q} \dfeq q,
				\qquad \pospred[q]{\formal{Q}} \dfeq (q \neq \bot_\formal{Q}),\]
			\[\bot_\formal{Q} \covby_\formal{Q} V \text{ always},
				\qquad \top_\formal{Q} \covby_\formal{Q} V \dfeq \top_\formal{Q} \in V \lor \QQ \subseteq V\]
			(similarly for $\formal{N}$, $\formal{Z}$). One can verify that we obtained the natural numbers, the integers and the rational numbers objects.
			
			These three were easy to get, since the sets $\NN$, $\ZZ$, $\QQ$ have decidable equality, even constructively. We now recall the construction of the formal space of reals $\formal{R}$.
			
			Let $\overline{\QQ} \dfeq \QQ \cup \{-\infty, \infty\}$. This is a bounded lattice for the usual order. We declare $\formal{R}^* \dfeq \overline{\QQ} \times \overline{\QQ}$. Intuitively, $(a, b) \in \formal{R}^*$ represents the interval $\intoo{a}{b}$. In the sense, that we take open rational (possibly infinite) intervals as the basis, the formal space of reals has the euclidean topology.
			
			Using the intuition that the basic opens are intervals, the rest of the structure is defined as follows for $a, b, c, d \in \overline{\QQ}$:
			\begin{itemize}
				\item
					$(a, b) \land_\formal{R} (c, d) \dfeq \big(\sup\{a, c\}, \inf\{b, d\}\big)$, $\top_\formal{R} \dfeq (-\infty, \infty)$,
				\item
					$\pospred[(a, b)]{\formal{R}} \dfeq a < b$,
				\item
					$\covby_\formal{R}$ is the smallest covering relation, satisfying all of the following:
					\begin{itemize}
						\item
							if $a \geq b$, then $(a, b) \covby_\formal{R} \emptyset$,
						\item
							if $c < b$, then $(a, d) \covby_\formal{R} \big\{(a, b), (c, d)\big\}$,
						\item
							$(a, b) \covby_\formal{R} \overline{\QQ}_{> a} \times \overline{\QQ}_{< b}$.
					\end{itemize}
			\end{itemize}
			
			We now define the streak structure on $\formal{R}$. Since we'll later use Theorem~\ref{Theorem: preservation_of_rationals_implies_streak_morphism}, it is more convenient for us to define the strict order via comparisons with rationals.
			
			Let the formal space $\formalgtQ = \big(\formalgtQ^*, \land_{\formalgtQ}, \top_{\formalgtQ}, \covby_{\formalgtQ}, \pospred{\formalgtQ}\big)$ be given as follows: $\formalgtQ^* \dfeq \st{(q, a, b) \in \QQ \times \formal{R}^*}{q < a \land q < b} + \{\top_{\formalgtQ}\}$, with the structure inherited from $\formal{Q} \times \formal{R}$ in the way that makes $\formalgtQ$ an open subobject of $\formal{Q} \times \formal{R}$. Specifically, the open embedding $\embgtQ\colon \formalgtQ \to \formal{Q} \times \formal{R}$ is given by $\embgtQ^*\colon \formal{Q}^* \times \formal{R}^* \to \pst(\formalgtQ^*)$,
			\begin{itemize}
				\item
					$\embgtQ^*(q, a, b) \dfeq \big\{\big(\sup\{q, a\}, \sup\{q, b\}\big)\big\}$ for $q \in \QQ$,
				\item
					$\embgtQ^*(\bot_\formal{Q}, a, b) \dfeq \emptyset$,
				\item
					$\embgtQ^*(\top_\formal{Q}, a, b) \dfeq \st{\big(\sup\{q, a\}, \sup\{q, b\}\big)}{q \in \QQ}$.
			\end{itemize}
			As for any open subobject in $\FormSp$, we have the ``inclusion'' map $\dot{\embgtQ}\colon \formalgtQ^* \to \formal{Q}^* \times \formal{R}^*$ which induces a right inverse to $\embgtQ$ on the level of topologies.
			
			Completely analogously we define $\formalltQ$ and $\embltQ$.
			
			Addition $+\colon \formal{R} \times \formal{R} \to \formal{R}$ is given by $+^*\colon \formal{R}^* \to \pst(\formal{R}^* \times \formal{R}^*)$,
			\[+^*(a, b) \dfeq \st{(c, d, e, f) \in \formal{R}^* \times \formal{R}^*}{a \leq c + e \land d + f \leq b}.\]
			The unit for addition $0\colon \one \to \formal{R}$ is given by $0^*\colon \formal{R}^* \to \one$, $0^*(a, b) \dfeq \st{\unit \in \one}{a < 0 < b}$.
			
			We have $\formal{R}_{> 0}^* = \st{(a, b) \in \formal{R}^*}{a > 0 \land b > 0}$, and the multiplication $\cdot\colon \formal{R}_{> 0} \times \formal{R}_{> 0} \to \formal{R}_{> 0}$ is given by\footnote{The fact, that we need to define multiplication only on the positive part, makes the definition quite simple.} $\cdot^*\colon \formal{R}_{> 0}^* \to \pst(\formal{R}_{> 0}^* \times \formal{R}_{> 0}^*)$,
			\[\cdot^*(a, b) \dfeq \st{(c, d, e, f) \in \formal{R}^* \times \formal{R}^*}{a \leq c \cdot e \land d \cdot f \leq b},\]
			with the unit $1$ given similarly as $0$.
			
			One can verify that this makes $\formal{R}$ a streak in $\FormSp$ (but we do not do it here).
			
			We claim that $\formal{R}$ is a terminal streak. Let $S$ be an arbitrary streak in $\FormSp$. In particular this means that we have open strict orders, comparing ``elements'' of $S$ and rationals both ways, that is $\embgtQ[S]\colon \formalgtQ[S] \to \formal{Q} \times S$ and $\embltQ[S]\colon \formalltQ[S] \to S \times \formal{Q}$.
			
			We need to construct a streak morphism $f\colon S \to \formal{R}$. Since $f$ needs to preserve comparison with rationals, the idea is that $f^*(a, b)$ should be $\intoo[S]{a}{b}$ in some sense. The following definition for $a, b \in \QQ$ realizes this.
			\[f^*(a, \infty) \dfeq \st{s \in S^*}{(a, s) \in \formalgtQ[S]} \qquad f^*(-\infty, b) \dfeq \st{s \in S^*}{(s, b) \in \formalltQ[S]}\]
			We expand this to all basic opens of $\formal{R}$ by insisting that $f^*$ preserves finite meets, as well as dealing with infinite boundaries of empty basic opens.
			\[f^*(-\infty, \infty) = f^*(\top_\formal{R}) \dfeq \{\top_S\}\]
			\[f^*(a, b) = f^*\big((a, \infty) \land_\formal{R} (-\infty, b)\big) \dfeq f^*(a, \infty) \land_S f^*(-\infty, b)\]
			\[f^*(-\infty, -\infty) = f^*(\infty, \infty) = f^*(\infty, -\infty) \dfeq \emptyset\]
			
			We claim that this $f$ preserves comparison with rationals. Specifically for the comparison with rationals on the left this means we have a (necessarily unique) map $g$ which makes the following diagram commute.
			\[\xymatrix@+2em{
				\formalgtQ[S] \ar[r]^{g} \ar[d]_{\embgtQ[S]}  &  \formalgtQ \ar[d]^{\embgtQ}  \\
				\formal{Q} \times S \ar[r]_{\id[\formal{Q}] \times f}  &  \formal{Q} \times \formal{R}
			}\]
			On the level of topologies this amounts to
			\[\xymatrix@+2em{
				\optp(\formalgtQ[S]) \ar@/_2em/[d]_{\dot{\optp}(\embltQ[S])}  &  \optp(\formalgtQ) \ar[l]_{\optp(g)} \ar@/^2em/[d]^{\dot{\optp}(\embltQ)}  \\
				\optp(\formal{Q} \times S) \ar[u]_{\optp(\embgtQ[S])}  &  \optp(\formal{Q} \times \formal{R}) \ar[l]^{\optp(\id[\formal{Q}] \times f)} \ar[u]^{\optp(\embgtQ)}
			}\]
			The map $\optp(g)$ can be expressed with others --- if this diagram commutes, then necessarily
			\[\optp(g) = \optp(g) \circ \optp(\embgtQ) \circ \dot{\optp}(\embgtQ) = \optp(\embgtQ[S]) \circ \optp(\id[\formal{Q}] \times f) \circ \dot{\optp}(\embgtQ).\]
			This is therefore the only possible candidate for $g$; we check that it actually works. Since these maps preserve finite meets and arbitrary joins, it is enough to verify the commutativity of the diagram for elements of the form $(q, a, \infty)$, $(q, -\infty, b)$ where $q, a, b \in \QQ$. On one hand we have
			\[\optp(\embgtQ[S]) \circ \optp(\id[\formal{Q}] \times f)\big([\{(q, a, \infty)\}]\big) = \optp(\embgtQ[S])\big([(\id[\formal{Q}] \times f)^*(q, a, \infty)]\big) =\]
			\[= \optp(\embgtQ[S])\big([\{q\} \times f^*(a, \infty)]\big) = \optp(\embgtQ[S])\big([\{q\} \times \st{s \in S^*}{(a, s) \in \formalgtQ[S]}]\big) =\]
			\[= \optp(\embgtQ[S])\big([\st{(q, s) \in \formal{Q}^* \times S^*}{(a, s) \in \formalgtQ[S]}]\big) =\]
			\[= [\st{(q, s \land_S t) \in \formal{Q}^* \times S^*}{(a, s) \in \formalgtQ[S] \land (q, t) \in \formalgtQ[S]}],\]
			and on the other
			\[\optp(g) \circ \optp(\embgtQ)\big([\{(q, a, \infty)\}]\big) =\]
			\[= \optp(\embgtQ[S]) \circ \optp(\id[\formal{Q}] \times f) \circ \dot{\optp}(\embgtQ)\big([\{(q, \sup\{q, a\}, \infty)\}]) =\]
			\[= \optp(\embgtQ[S]) \circ \optp(\id[\formal{Q}] \times f)\big([\{(q, \sup\{q, a\}, \infty)\}]) =\]
			\[= \optp(\embgtQ[S])\big([\st{(q, s) \in \formal{Q}^* \times S^*}{(\sup\{q, a\}, s) \in \formalgtQ[S]}]\big) =\]
			\[= [\st{(q, s \land_S t) \in \formal{Q}^* \times S^*}{(\sup\{q, a\}, s) \in \formalgtQ[S] \land (q, t) \in \formalgtQ[S]}].\]
			These two equivalence classes are the same (\ie their representatives cover each other). In one direction this amounts to the observation that (due to transitivity) if $(\sup\{q, a\}, s) \in \formalgtQ[S]$, then also $(a, s) \in \formalgtQ[S]$. For the other direction, given $(q, s \land_S t) \in \formal{Q}^* \times S^*$ with $(a, s) \in \formalgtQ[S]$ and $(q, t) \in \formalgtQ[S]$, we can rewrite it as $(q, (s \land_S t) \land_S t)$ with $(\sup\{q, a\}, s \land_S t) \in \formalgtQ[S]$ and $(q, t) \in \formalgtQ[S]$.
			
			In conclusion, $f$ preserves the comparison with rationals on the left. Exactly the same argument shows it also preserves the comparison with rationals on the right.
			
			\begin{theorem}
				The formal space of reals $\formal{R}$ is a terminal streak, and so a model of reals by Definition~\ref{Definition: reals}.
			\end{theorem}
			\begin{proof}
				By the discussion above, using Theorem~\ref{Theorem: preservation_of_rationals_implies_streak_morphism} for the conclusion.
			\end{proof}
			
			We have seen that reals as a formal space with the euclidean topology are a terminal streak in the category of formal spaces. We can easily translate this result to two other settings.
			
			\df{Locales} are essentially an impredicative version of pointfree topology, where a space is presented by its topology (we won't go into more details in this paper; for an exposition on locales, see~\cite{Gierz_G_Hoffmann_KH_Keimel_K_Lawson_JD_Mislove_MW_Scott_DS_2003:_continuous_lattices_and_domains}\cite{Johnstone_PT_2002:_stone_spaces}), \ie the frame (complete distributive lattice) of all opens (as opposed to just the basic ones). Thus the study of locales can be seen as a special case of study of formal spaces, namely when topologies (in particular powersets, which are discrete topologies) are sets, rather than proper classes (hence the impredicativity). Consequently we can immediately infer, that the \df{locale of reals} is a terminal streak in the category of locales.
			
			We can stretch this result also to classical topological spaces. Despite being closely related, clearly topological spaces and locales are not the same: for example, any two spaces with trivial topology are the same, as locales. However, the categories of \df{sober topological spaces} and \df{spatial locales} are equivalent (that is, up to homeomorphism, sober topological spaces and continuous maps between them are in one-to-one correspondence to spacial locales and continuous maps between them).
			
			Classically the only apartness relation is the inequality $\neq$, which must be open in any streak, or equivalently, the equality $=$ is closed. This means that every streak in the category of topological spaces is Hausdorff, hence sober, hence has a (up to homeomorphism unique) counterpart among (spatial) locales. Since the locale of reals is terminal among locale streaks, and moreover spatial, the topological space of reals is terminal among topological streaks.
			
			A few words about the intuition, why the terminal streak $\RR$ must have the euclidean topology. There are more topological streaks than ``set-streaks'' in the following sense: the underlying set of a topological streak is a set-streak, but conversely, there may be many topologies on a given set-streak, such that $<$ is an open relation and the operations are continuous. For example, beside $\RR$ with the euclidean topology, $\RR$ with the discrete topology is also a topological streak. The identity map is the unique streak morphism $(\RR, \text{discrete}) \to (\RR, \text{euclidean})$, but there is no morphism in the other direction: the identity is continuous only when the topology of the domain is at least as strong as the topology on the codomain. Clearly then the terminal topological streak must have the weakest possible topology, in which $<$ is open and the operations are continuous --- and that's the euclidean topology.
		
		\subsection{Recap of various mathematical models}
		
			In the previous subsections of this section we focused on individual constructions of $\RR$. In this subsection we focus on the setting, explaining how various types of settings fulfil the conditions, outlined in Section~\ref{Section: setting}, and which of the above constructions of reals are suitable for them.
			
			\begin{itemize}
				\item\textbf{Classical mathematics}\\
					``Sets'' are the usual sets and ``classes'' are classes of some class theory, or they can be taken to be sets as well. All subsets are decidable, and so must be open and closed, \ie the only choice of the intrinsic topology is the discrete one.
					
					Classical mathematics validates (countable) choice and the powerset axiom. Hence Cauchy reals (Subsection~\ref{Subsection: Cauchy_reals}), Dedekind reals (Subsection~\ref{Subsection: Dedekind_reals}) and reals via interval domain (Subsection~\ref{Subsection: interval_domain}) are all possible constructions, yielding isomorphic models of reals.
				\item\textbf{Bishop-style constructivism}~\cite{Bishop_E_Bridges_D_1985:_constructive_analysis}\\
					Same as the previous item, really --- while the lack of the law of excluded middle in principle gives us some freedom, what intrinsic topology to choose, usually one does not make any prescription (that is, the discrete one is chosen). Typically the reals are represented as Cauchy (Subsection~\ref{Subsection: Cauchy_reals}), but when one does use Dedekind cuts, no topological conditions are imposed on them.
				\item\textbf{Intuitionism}~\cite{Bridges_DS_Richman_F_1987:_varieties_of_constructive_mathematics}\\
					As in the previous item.
				\item\textbf{Realizability}~\cite{Bauer_A_2000:_the_realizability_approach_to_computable_analysis_and_topology}\\
					What ``sets'' and ``classes'' are, depends on the specific choice of a model. In any case, at least \df{modest sets} are ``sets''. ``Open'' means \df{semidecidable}. For ``closed'' one can either take complements of open subsets, or all subsets with open complements, it makes no difference for our theory. Countable choice holds, and the reals are typically represented as Cauchy (Subsection~\ref{Subsection: Cauchy_reals}).
				\item\textbf{Predicative settings}~\cite{Myhill_J_1975:_constructive_set_theory}\\
					In predicative mathematics one does not assume the powerset axiom. There is a usual set/class distinction and the collection of all subsets of a given set is a class. It is still assumed that arbitrary unions of opens are open, and consequently that topologies are classes, but generally not sets. The model of reals used is the formal space of reals (Subsection~\ref{Subsection: topological_models_of_reals}).
%
				\item\textbf{Constructive settings with intrinsic topology}\\
					Examples are \df{synthetic topology}~\cite{Lesnik_D_2010:_synthetic_topology_and_constructive_metric_spaces} and \df{Abstract Stone Duality} (ASD)~\cite{Bauer_A_Taylor_P_2009:_the_dedekind_reals_in_abstract_stone_duality}. In both cases, the intrinsic topology is given via a \df{Sierpi\'{n}ski object} $\opn$ by defining $\optp(X) \dfeq \opn^X$ (\ie topologies are sets in this setting).
					
					Generally we don't have countable choice, and the Cauchy reals don't necessarily work. Instead, reals are constructed as open Dedekind cuts (as in Subsection~\ref{Subsection: Dedekind_reals}). Under mild assumptions these reals have euclidean topology.
			\end{itemize}
			
			There is a variety of constructivism, for which our theory obviously doesn't work: \df{ultrafinitism}~\cite{Bridges_DS_Richman_F_1987:_varieties_of_constructive_mathematics}. The reason for this is that ultrafinitism does not assume the existence of a set of all natural numbers, whereas we did.

	\section{Additional examples}\label{Section: additional}
	
		In this section we discuss some notions which have the word ``real'' in their name, but do not directly fit into our theory, in the sense that they are not terminal streaks. This is because in these cases one purposely breaks some standard property of reals to get another useful one (in the case of \df{smooth reals} one sacrifices partial order to get an intrinsic smooth structure, and in the cases of \df{lower} and \df{upper reals} one requires comparison with rationals only on one side, in order to get a variant of order completness).
	
		\subsection{Smooth reals}\label{Subsection: smooth_reals}
		
			In this subsection we consider the \df{smooth reals} from \df{synthetic differential geometry} (SDG). SDG (or a closely related synonym \df{smooth analysis}) is an approach where, rather that taking a set, equip it with topology and smooth structure and call it a ``smooth manifold'' (and similarly for maps), the background logic and axioms are changed in a way which in a certain sense equips every object with its ``intrinsic'' smooth structure and every map (that can be constructed in the setting) is automatically smooth. This makes working with smooth manifolds and smooth maps similar to working with the usual sets and maps (though the backgroung logic is necessarily constructive), generally quite a simplification.\footnote{The simplification does not come from trivialising the theory, or some such. There are models of SDG into which the category of smooth manifolds is embedded in a nice way. In this sense any synthetic theorem provides also a corresponding theorem in classical differential geometry.} For example, the tangent bundle of an object $M$ is given simply as an exponential $M^D$; nothing further (such as explicit topology or smooth structure) is required.
			
			Here $D$ denotes the set of \df{nilsquare infinitesimals},
			\[D \dfeq \st{x \in \smR}{x^2 = 0},\]
			where $\smR$ denotes the set of \df{smooth reals}. The most basic axiomatisation of SDG declares that $\smR$ is a commutative unital ring and a module over $\QQ$, and the following \df{Kock-Lawvere} axiom holds: For every map $f\colon D \to \smR$ there exist unique $a, b \in \smR$ such that $f(x) = a + b \cdot x$ for all $x \in D$.
			
			Note that the Kock-Lawvere axiom in particular implies $\{0\} \subsetneq D$ --- \ie $0$ is not the only infinitesimal (= infinitely small element).
			
			For more on how this makes differential geometry going (in particular, how derivatives of maps are defined), see~\cite{bell2008primer}.
			
			So, do the smooth reals have anything to do with streaks, in particular the terminal ones? The first obstacle to answering this question is that SDG is still a relatively young theory and there is no fixed axiomatisation of it yet (at the time of writing this paper). The above is the bare minimum, enough to define notions such as derivatives and tangent bundles, but not enough for a deeper theory.
			
			Moerdijk and Reyes provide a study of various possible (models and) axioms of SDG in~\cite{moerdijk1991models}. We summarize their axioms here.
			\begin{itemize}
				\item
					Axioms (A1)--(A5): $\smR$ is a commutative unital local ring which has square roots of all positive elements, as well as inverses of all positive and negative elements, and is equipped with order relations $<$, $\leq$, satisfying the usual properties. Also, we have $0 \leq x$ for all nilpotent $x \in \smR$.
				\item
					Axiom (A6): \emph{generalized} Kock-Lawvere axiom holds (there are polynomial formulae for real-valued maps, defined on more general infinitesimal objects than $D$).
				\item
					Axioms (A7)--(A9): axioms which make integration work.
				\item
					Axioms (A10)--(A15): axioms, expressing connection between $\smR$ and natural numbers; in particular, (A11) states that $\smR$ is archimedean with regard to smooth natural numbers.
				\item
					Axioms (A16)--(A17): properties of covers of $\intcc[\smR]{0}{1}$, in particular its compactness.
				\item
					Axioms (A18)--(A19): some standard functions exist.
				\item
					Axioms (A20)--(A21): existence and properties of invertible infinitesimals.
			\end{itemize}
			Not all of these axioms hold in all of the models that Moerdijk and Reyes study, but (A1)--(A5) do. These are enough to conclude that $\smR$ is a (field) prestreak (for some reasonable choice of intrinsic topologies, such as the smallest ones, for which $<$ on $\smR$ is open and $\leq$ on $\smR$ is closed).
			
			The axiom (A11) (which also holds in all the models) technically states that $\smR$ is archimedean --- but with regard to the \df{smooth natural numbers}, which may or may not be the same as the usual natural numbers, depending on the model. For the most standard models they are, though, and even in general we could just consider $\arch(\smR)$ (recall Subsection~\ref{Subsection: archimedean_coreflection}) which is an archimedean (field) prestreak which still contains all the infinitesimals, required for the theory. Thus we will just assume that $\smR$ is also archimedean.
			
			However, while $\smR$ can be assumed to be an archimedean prestreak, it cannot possibly be a streak. The crucial point of SDG is that we have nontrivial infinitesimals, in particular $\{0\} \subsetneq D \subseteq \intcc[\smR]{0}{0}$, so the preorder $\leq$ on $\smR$ is not a partial order.
			
			Thus $\smR$ cannot be a terminal streak, so it does not satisfy our definition of the reals. This isn't surprising --- all maps $\smR \to \smR$ are smooth practically by definition, but we have nondifferentiable (in the limit of differential quotient sense) maps $\RR \to \RR$, such as the absolute value.
			
			We may still make some use of our theory, though. Since models of SDG are topoi with natural numbers, we can still construct the terminal streak $\RR$, say via Dedekind cuts (recall Subsection~\ref{Subsection: Dedekind_reals}). By terminality of $\RR$ we have
			\[\smR \stackrel{\theta_\smR}{\longrightarrow} Q(\smR) \stackrel{\trm[Q(\smR)]}{\longrightarrow} \RR,\]
			the composition of which is a familiar mapping: taking the standard part.
		
		\subsection{Lower and upper reals}
		
			In Subsection~\ref{Subsection: Dedekind_reals} we presented the construction of reals as two-sided Dedekind cuts. It is known that in classical mathematics we don't need to specify both the lower and the upper cut of a Dedekind real --- either one can be reconstructed from the other. Thus the set of lower cuts, called the \df{lower reals} (we'll denote it by $\lr$), the set of upper cuts, called the \df{upper reals} (denoted by $\ur$) and the set of two-sided Dedekind cuts $\RR$ are all in bijective correspondence\footnote{More precisely, this holds, after we restrict to finite cuts. The whole dense streak can also be seen as a lower cut (representing $\infty$) or an upper cut (representing $-\infty$) which are of course not real numbers. See also the discussion at the end of this subsection.}, and therefore, by the results of Subsection~\ref{Subsection: Dedekind_reals}, a model of reals, according to our definition.
			
			This is not the case constructively; in general all these three sets are distinct. One can of course see $\RR$ as a subset of both $\lr$, $\ur$ (by ``forgetting'' one of the cuts), but this inclusion is not an equality, nor are $\lr$ and $\ur$ comparable. Since $\RR$ is a model of reals, it follows that neither the lower reals nor the upper reals are reals by our definition.
			
			This isn't a big deal since constructively one usually uses $\RR$ anyway, being the only field out of the three (one can't subtract in $\lr$, $\ur$, nor can one multiply in general, though it is possible to multiply nonnegative elements).
			
			Still, the one-sided reals are useful even constructively, as (unlike $\RR$ itself) they satisfy a form of Dedekind completness, and there is a way how to incorporate these two sets into our context. Recall from Proposition~\ref{Proposition: alt_streak} that a streak $X$ can be equivalently given by specifying the relations $< \subseteq \QQ \times X$ and $< \subseteq X \times \QQ$, rather than $< \subseteq X \times X$. The idea is to split the definition of a streak into two parts, each of which refers to only one-sided comparison with rationals. Then the lower and the upper reals should be terminal among such ``halfstreaks''.
			
			\begin{definition}\label{Definition: lower_streak}
				Let a set $X$ be equipped with a relation $< \subseteq \QQ \times X$ and operations $+\colon X \times X \to X$ and $\cdot\colon X_{> 0} \times X_{> 0} \to X_{> 0}$ (where $X_{> 0}$ stands for the set of all elements in $X$ which are bigger than the \emph{rational} $0$). Suppose the following conditions hold:
				\begin{itemize}
					\item
						boundedness (from below): $\xall{a}{X}\xsome{q}{\QQ}{q < a}$,
					\item
						cotransitivity: for all $a \in X$ and $q, r \in \QQ$
						\[q < a \implies q < r \lor r < a,\]
					\item
						for all $a, b \in X$
						\[\all{q}{\QQ}{q < a \iff q < b} \implies a = b,\]
					\item
						the relation $<$ is open and its negation $\leq$ closed,
					\item
						$+$ makes $X$ into a commutative monoid (meaning we in particular have an additive unit which we also denote by $0$),
					\item
						for all $q, r \in \QQ$ and $a, b \in X$
						\[q < a \land r < b \implies q + r < a + b,\]
					\item
						for all $q \in \QQ$ and $a, b \in X$
						\[q < a + b \implies \some{r, s}{\QQ}{q < r + s \land r < a \land s < b},\]
					\item
						$\cdot$ makes $X_{> 0}$ into a commutative monoid (meaning we also have a multiplicative unit which we denote by $1$) and distributes over $+$,
					\item
						for all $q, r \in \QQ_{> 0}$ and $a, b \in X_{> 0}$
						\[q < a \land r < b \implies q \cdot r < a \cdot b,\]
					\item
						for all $q \in \QQ_{> 0}$ and $a, b \in X_{> 0}$
						\[q < a \cdot b \implies \some{r, s}{\QQ_{> 0}}{q < r \cdot s \land r < a \land s < b},\]
					\item
						``asymmetry'': $\lnot(\underbrace{1}_{\in \QQ} < \underbrace{1}_{\in X})$.
				\end{itemize}
				Then $(X, <, +, 0, \cdot, 1)$ is called a \df{lower streak}.
			\end{definition}
			
			Note that the conditions for addition and multiplication can be readily generalized.
			\begin{lemma}
				Let $X$ be a lower streak. Then the following holds for all $n \in \NN_{> 0}$:\footnote{Actually the statements hold for $n = 0$ as well; for this we need Proposition~\ref{Proposition: unambiguity_of_order_in_lower_streaks} below (and recall that the sum of zero summands is $0$ and the product of zero factors is $1$).}
				\begin{enumerate}
					\item
						for all $q_0, \ldots, q_{n-1} \in \QQ$ and $a_0, \ldots, a_{n-1} \in X$,
						\[\all{i}{\NN_{< n}}{q_i < a_i} \implies \sum_{i \in \NN_{< n}} q_i < \sum_{i \in \NN_{< n}} a_i,\]
					\item
						for all $q \in \QQ$ and $a_0, \ldots, a_{n-1} \in X$,
						\[q < \sum_{i \in \NN_{< n}} a_i \implies \some[2]{r_0, \ldots, r_{n-1}}{\QQ}{q < \sum_{i \in \NN_{< n}} r_i \land \xall{i}{\NN_{< n}}{r_i < a_i}},\]
					\item
						for all $q_0, \ldots, q_{n-1} \in \QQ_{> 0}$ and $a_0, \ldots, a_{n-1} \in X_{> 0}$,
						\[\all{i}{\NN_{< n}}{q_i < a_i} \implies \prod_{i \in \NN_{< n}} q_i < \prod_{i \in \NN_{< n}} a_i,\]
					\item
						for all $q \in \QQ_{> 0}$ and $a_0, \ldots, a_{n-1} \in X_{> 0}$,
						\[q < \prod_{i \in \NN_{< n}} a_i \implies \some[2]{r_0, \ldots, r_{n-1}}{\QQ}{q < \prod_{i \in \NN_{< n}} r_i \land \xall{i}{\NN_{< n}}{r_i < a_i}}.\]
				\end{enumerate}
				In particular we have $q < a \implies \some{r}{\QQ}{q < r \land r < a}$ for all $q \in \QQ$ and $a \in X$.
			\end{lemma}
			\begin{proof}
				We only check the stated special case; the rest is obvious induction on $n \in \NN_{> 0}$.
				
				Suppose $q < a$; since $a = a + 0$, there are $r, s \in \QQ$ such that $q < r + s$, $r < a$ and $s < 0$. Then $r + s < a + 0 = a$, so $r + s$ is the looked-for rational.
			\end{proof}
			
			As a monoid for addition, any lower streak possesses mutiplication with natural numbers, defined inductively in the usual way: $\underbrace{0}_{\in \NN} \cdot a \dfeq \underbrace{0}_X$, $(n+1) \cdot a \dfeq n \cdot a + a$. In particular, natural numbers $\NN$ embed into any lower streak $X$ (via $n \mapsto n \cdot 1$), and in this sense we write $\NN \subseteq X$.
			
			Also standard, we define the non-strict order by negating $<$:
			\[a \leq q \dfeq \lnot(q < a)\]
			for $a \in X$, $q \in \QQ$. However, we can also define $\leq$ between elements of a lower streaks $X$ themselves by declaring
			\[a \leq b \dfeq \all{q}{\QQ}{q < a \implies q < b}\]
			for $a, b \in X$. Obviously this $\leq$ is a preorder and one of the conditions in the definition of a lower streak is precisely the antisymmetry of $\leq$\footnote{Dropping this condition from the definition would yield a ``lower archimedean prestreak''.}; thus $\leq$ is a partial order on $X$.
			
			All possible transitivity conditions hold.
			\begin{proposition}
				Let $X$ be a lower streak. For all $q, r \in \QQ$ and $a, b, c \in X$ the following holds.
				\begin{enumerate}
					\item
						$q \leq r \land r < a \implies q < a$ (in particular $q < r \land r < a \implies q < a$)
					\item
						$q < a \land a \leq b \implies q < b$
					\item
						$q < a \land a \leq r \implies q < r$
					\item
						$a \leq q \land q \leq r \implies a \leq r$
					\item
						$a \leq b \land b \leq q \implies a \leq q$
					\item
						$a \leq b \land b \leq c \implies a \leq c$
				\end{enumerate}
			\end{proposition}
			\begin{proof}
				\begin{enumerate}
					\item
						If $r < a$, then by cotransitivity $r < q \lor q < a$, but the first disjunct is in contradiction with $q \leq r$, so the second one must hold.\footnote{In fact, since $<$ is decidable on $\QQ$, this transitivity condition not just follows from, but is even equivalent to cotransitivity in lower streaks.}
					\item
						By definition of $\leq$ on $X$.
					\item
						Same argument as in the first item.
					\item
						If $r < a$ held, then together with $q \leq r$ it would yield $q < a$ by an already known transitivity, in contradiction with $a \leq q$.
					\item
						If $q < a$ held, it would follow $q < b$ from $a \leq b$, in contradiction with $b \leq q$.
					\item
						Take any $q \in \QQ$ with $q < a$. From $a \leq b$ it follows $q < b$ and then from $b \leq c$ it follows $q < c$, as desired.
				\end{enumerate}
			\end{proof}
			
			It is inconvenient having to always specify when $0$, $1$ (or natural numbers in general) represent an element of a lower streak or a rational, so we show that it doesn't matter.
			
			\begin{proposition}\label{Proposition: unambiguity_of_order_in_lower_streaks}
				Let $X$ be a lower streak. The following holds for all $q \in \QQ$, $n \in \NN$ and $a \in X$:
				\begin{enumerate}
					\item
						$q < a \iff q + \underbrace{n}_{\in \QQ} < a + \underbrace{n}_{\in X}$,
					\item
						$q < \underbrace{n}_{\in \QQ} \iff q < \underbrace{n}_{\in X}$, therefore also $\underbrace{n}_{\in \QQ} \leq q \iff \underbrace{n}_{\in X} \leq q$,
					\item
						$a \leq \underbrace{n}_{\in \QQ} \iff a \leq \underbrace{n}_{\in X}$,
					\item
						$\underbrace{n}_{\in X} \leq \underbrace{n}_{\QQ}$.
				\end{enumerate}
			\end{proposition}
			\begin{proof}
				We check some special cases first, eventually building up to the general statements. We use $X$ or $\QQ$ in the indices, to denote in which set a number is meant to be.
				\begin{itemize}
					\item\proven{$\all{q}{\QQ}{q < 0_\QQ \implies q < 0_X}$}
						Suppose $q < 0_\QQ$. There is some $r \in \QQ$ with $r < 0_X$. Let $n \in \NN_{> 0}$ be such that $n q < r$. From $r < 0_X = n \cdot 0_X$ it follows that there are $s_0, \ldots, s_{n-1} \in \QQ$ such that $r < s_0 + \ldots + s_{n-1}$ and $s_i < 0_X$ for all $i \in \NN_{< n}$. For each such $i$ consider $s_i < q \lor q < 0_X$. If for any $i$ the second disjunct holds, we are done. If all first disjuncts held, then summing them would yield $r < n q$, a contradiction.
					\item\proven{$\all{q}{\QQ}{q < 1_\QQ \implies q < 1_X}$}
						Take $q \in \QQ$ with $q < 1_\QQ$. By definition $0_\QQ < 1_X$, so $0_\QQ < q \lor q < 1_X$. If the second disjunct holds, we are done. Assume now the first one. From here, the idea is the same as in the previous item, only we move one degree higher in operations (multiplication instead of addition, exponentiation instead of multiplication).
						
						Find all of the following: $r \in \QQ$ with $0_\QQ < r < 1_X$, $n \in \NN_{> 0}$ with $q^n < r$ and $s_0, \ldots, s_{n-1} \in \QQ$ with $r < s_0 \cdot \ldots \cdot s_{n-1}$ and $0 < s_i < 1_X$ for all $i \in \NN_{< n}$. For each such $i$ consider $s_i < q \lor q < 1_X$. If for any $i$ the second disjunct holds, we are done. If all first disjunts held, then multiplying them them would yield $r < q^n$, a contradiction.
					\item\proven{$\xall{n}{\NN}\xall{a}{X}\all{q}{\QQ}{q < a \implies q + n_\QQ < a + n_X}$}
						By induction on $n$. Obviously the statement holds for $n = 0$. Assuming the induction hypothesis for $n$, $q < a$ implies $q + n_\QQ < a + n_X$, so there is some $r \in \QQ$ with $q + n_\QQ < r < a + n_X$. Then $q + n_\QQ + 1_\QQ - r < 1_\QQ$, so by the previous item also $q + n_\QQ + 1_\QQ - r < 1_X$. Adding this to $r < a + n_X$, we obtain the desired result $q + n_\QQ + 1_\QQ < a + n_X + 1_X$.
					\item\proven{$\xall{n}{\NN}\all{q}{\QQ}{q < n_\QQ \implies q < n_X}$}
						We already know that this statement holds for $n = 0$, so if $q < n_\QQ$, then $q - n_\QQ < 0_\QQ$, so $q - n_\QQ < 0_X$. Using the previous item, we get $q < n_X$.
					\item\proven{$\all[1]{n}{\NN}{\lnot(n_\QQ < n_X) \land \all{q}{\QQ}{q < n_X \implies q < n_\QQ}}$}
						First note that we only have to prove the first conjunct $\lnot(n_\QQ < n_X)$ (or $n_X \leq n_\QQ$, if you will), as the second one follows from it by transitivity. However, we consider both statements together, as we need this for the following induction on $n$.
						
						Regarding the base of induction, if we had $0_\QQ < 0_X$, then by an item above also $1_\QQ < 1_X$, in contradiction with the definition of lower streaks.
						
						Assume now that the statement holds for $n$. If we had $(n+1)_\QQ < (n+1)_X$, then there would exist $q, r \in \QQ$ with $(n+1)_\QQ < q + r$, $q < n_X$, $r < 1_X$. By induction hypothesis $q < n_\QQ$. Since we have $1_X \leq 1_\QQ$, also $r < 1_\QQ$ by transitivity. Thus $(n+1)_\QQ < q + r < n_\QQ + 1_\QQ$, a contradiction.
					\item\proven{$\xall{n}{\NN}\xall{a}{X}\all{q}{\QQ}{q + n_\QQ < a + n_X \implies q < a}$}
						By induction on $n$. Clearly the statement holds for $n = 0$.
						
						Assume it holds for $n$ and suppose $q + n_\QQ + 1_\QQ < a + n_X + 1_X$. There exist $r, s \in \QQ$ such that $q + n_\QQ + 1_\QQ < r + s$, $r < a + n_X$ and $s < 1_X$. Recall what we've already proved to conclude $s < 1_\QQ$.
						
						By cotransitivity we have $r < q + n_\QQ \lor q + n_\QQ < a + n_X$. However, the first disjunct implies the contradiction $r + s < q + n_\QQ + 1_\QQ$. Thus the second disjunct must hold, and then we have $q < a$ by the induction hypothesis, as desired.
					\item\proven{$\xall{n}{\NN}\all{a}{X}{n_\QQ \geq a \implies n_X \geq a}$}
						Rewrite this statement into the form
						\[\xall{n}{\NN}\xall{a}{X}\all{q}{\QQ_{< a}}{a \leq n_\QQ \implies q < n_X}.\]
						Given $n \in \NN$, $a \in X$ with $a \leq n_\QQ$ and $q \in \QQ$ with $q < a$, it follows $q < n_\QQ$ by transitivity and then $q < n_X$ by one of the items above.
					\item\proven{$\xall{n}{\NN}\all{a}{X}{n_X \geq a \implies n_\QQ \geq a}$}
						Rewrite this statement into the form
						\[\lnot\xsome{n}{\NN}\some{a}{X}{n_\QQ < a \land \xall{q}{\QQ_{< a}}{q < n_X}}.\]
						Suppose we had $n \in \NN$ and $a \in X$ with $n_\QQ < a$ and $\xall{q}{\QQ_{< a}}{q < n_X}$. Clearly this implies $n_\QQ < n_X$, in contradiction with what we've proved in one of the previous items.
				\end{itemize}
			\end{proof}
			
			A special case of the comparison $q < a$ (where $q \in \QQ$ and $a \in X$) is when $q$ is a natural number. We claim that the comparison of elements of a lower streak with natural numbers already uniquely determines the comparison with rationals in general.
			\begin{proposition}
				Let $X$ be a lower streak. Then for any $a \in X$, any $q \in \QQ$ and any decomposition $q = \frac{i-j}{k}$ where $i, j \in \NN$, $k \in \NN_{> 0}$ we have
				\[q < a \iff i < k \cdot a + j.\]
			\end{proposition}
			\begin{proof}
				We use results of the previous proposition for the following.
				\[q < a \iff \tfrac{i-j}{k} < a \iff i-j < k \cdot a \iff i < k \cdot a + j\]
			\end{proof}
			
			\begin{proposition}
				For a lower streak $X$ we have $q < a \land n > 0 \iff n \cdot q < n \cdot a$ for all $q \in \QQ$, $n \in \NN$, $a \in X$ (compare Proposition~\ref{Proposition: multiplication_with_natural_numbers}).
			\end{proposition}
			\begin{proof}
				Same idea as in the proof of Proposition~\ref{Proposition: multiplication_with_natural_numbers}, while using the properties of order relations on lower streaks, shown above.
			\end{proof}
			
			As usual, we want our objects in question --- lower streaks --- to form a category. Thus we need to define a notion of morphisms. We'll be very minimalistic in our definition, however.
			\begin{definition}
				Let $X$, $Y$ be lower streaks. A \df{lower streak morphism} $f\colon X \to Y$ is a map with the property
				\[q < a \iff q < f(a)\]
				for all $q \in \QQ$, $a, b \in X$.
			\end{definition}
			The reason why we only require preservation (in both direction) of $<$ is that the preservation of all the rest of the lower streak structure follows from that.
			\begin{lemma}
				Let $f\colon X \to Y$ be a lower streak morphism. Then the following holds for all $n \in \NN$, $q \in \QQ$ and $a, b \in X$:
				\begin{enumerate}
					\item
						$a \leq q \iff f(a) \leq q$,
					\item
						$a \leq b \iff f(a) \leq f(b)$,
					\item
						$f(a + b) = f(a) + f(b)$,
					\item
						$f(0) = 0$,
					\item
						if $a, b > 0$, then also $f(a), f(b) > 0$ and $f(a \cdot b) = f(a) \cdot f(b)$,
					\item
						$f(1) = 1$,
					\item
						$f(n \cdot a) = n \cdot f(a)$, in particular $f(n) = n$.
				\end{enumerate}
			\end{lemma}
			\begin{proof}
				\begin{enumerate}
					\item
						If $q < a$ and $q < f(a)$ are equivalent (by assumption), then so are their negations.
					\item
						Suppose $a \leq b$ and $q < f(b)$. Then $q < a$, therefore $q < b$ and so $q < f(b)$. The other direction works the same.
					\item
						Take any $q \in \QQ$ with $q < f(a + b)$. Then $q < a + b$ and there exist $r, s \in \QQ$ with $q < r + s$, $r < a$ and $s < b$. Hence $r < f(a)$, $s < f(b)$, so $q < r + s < f(a) + f(b)$. We conclude $f(a + b) \leq f(a) + f(b)$. The inequality $f(a) + f(b) \leq f(a + b)$ is proved the same way. The equality then follows from antisymmetry of $\leq$ on $X$.
					\item
						For any $q \in \QQ$ we have $q < f(0) \iff q < \underbrace{0}_{\in X} \iff q < \underbrace{0}_{\in \QQ} \iff q < \underbrace{0}_{\in Y}$. Thus $f(0) = 0$.
					\item
						Take $q \in \QQ$ with $q < f(a \cdot b)$ (or $q < f(a) \cdot f(b)$ for the other direction). If $q \leq 0$, we are clearly done while the case $0 < q$ works the same as for addition.
					\item
						Works the same as for $0$.
					\item
						Obvious induction.
				\end{enumerate}
			\end{proof}
			
			Like in the case of streaks, lower streak morphisms are injective and unique.
			
			\begin{proposition}\label{Proposition: lower_streaks_a_preorder_category}
				\
				\begin{enumerate}
					\item
						Lower streak morphisms are injective.
					\item
						For any lower streaks $X$, $Y$ there exists at most one lower streak morphism $X \to Y$. That is, lower streaks form a preorder category.
				\end{enumerate}
			\end{proposition}
			\begin{proof}
				For $f, g\colon X \to Y$ lower streak morphisms and $a, b \in X$ we have
				\[a = b \iff \all{q}{\QQ}{q < a \iff q < b} \iff\]
				\[\iff \all{q}{\QQ}{q < f(a) \iff q < g(b)} \iff f(a) = g(b).\]
				For $f = g$ we get injectivity; for $a = b$ we get uniqueness.
			\end{proof}
			
			We mentioned right at the beginning that $\NN$ embeds into any lower streak $X$ via $n \mapsto n \cdot 1$. Strictly speaking, we haven't checked that this map is actually injective (but we didn't use it anywhere either), but now that is clear (so writing $\NN \subseteq X$ is reasonable). Proposition~\ref{Proposition: unambiguity_of_order_in_lower_streaks} implies that this map is a lower streak morphism and by Proposition~\ref{Proposition: lower_streaks_a_preorder_category} it is then injective.
			
			We are now ready for the main point of this subsection, namely how the ``one-sided reals'' fit into the picture of streaks. Recall that the conditions for a streak ensured that rationals were dense in any streak in a suitable sense, and therefore the terminal (largest) streak could be deemed a completion of rationals. Similarly we can consider the \emph{terminal lower streak} to be a ``lower completion'' of $\QQ$.
			\begin{definition}
				The \df{lower reals} $\lr$ are the terminal lower streak.
			\end{definition}
			We claim that the set of all lower cuts (when it exists) is a terminal lower streak.
			
			Recall from Subsection~\ref{Subsection: Dedekind_reals} the definition of the lower cut (like in that subsection we assume that $\optp(\QQ)$ exists). Let $\lc(\QQ)$ denote the set of all lower cuts on $\QQ$ (for simplicity and brevity's sake we'll consider only the cuts on rationals here, unlike in Subsection~\ref{Subsection: Dedekind_reals}, where cuts on more general dense streaks were considered).
			
			We equip $\lc(\QQ)$ with the lower streak structure as follows. For $q \in \QQ$ and cuts $L, L', L'' \in \lc(\QQ)$ let
			\[q < L \dfeq q \in L,\]
			\[L' + L'' \dfeq \st{r + s}{r \in L' \land s \in L''} = \st{q \in \QQ}{\xsome{r}{L'}\xsome{s}{L''}{q < r + s}}.\]
			Note that the two definitions of addition are equivalent and that the unit for addition is $\QQ_{< 0}$. If $L'$ and $L''$ are positive (meaning $0 \in L'$, $0 \in L''$, and therefore $L'_{> 0}$ and $L''_{> 0}$ are inhabited), we further define
			\[L' \cdot L'' \dfeq \st{q \in \QQ}{\xsome{r}{L'_{> 0}}\xsome{s}{L''_{> 0}}{q < r \cdot s}}.\]
			The unit for multiplication is $\QQ_{< 1}$.
			
			We leave the verification, that all this is well defined and that $\lc(\QQ)$ is a lower streak, to the reader.
			
			\begin{theorem}
				$\lc(\QQ)$ is a terminal lower streak.
			\end{theorem}
			\begin{proof}
				For any morphism $f\colon X \to \lc(\QQ)$ we must have $q < x \iff q < f(x) \iff q \in f(x)$ for all $q \in \QQ$ and $x \in X$. Thus
				\[f(x) \dfeq \st{q \in \QQ}{q < x}\]
				is the unique lower streak morphism from $X$ to $\lc(\QQ)$.
			\end{proof}
			
			Unsurprisingly, analogous results hold for \emph{upper} reals. There is just one minor difference: in an upper streak $X$ we have a relation $< \subseteq X \times \QQ$, so we can't directly say when an element of $X$ is positive. However, note that if $X$ is a streak, we have $0 < a \iff \xsome{q}{\QQ_{> 0}}{q \leq a}$ for all $a \in X$. Thus the following definition.
			
			\begin{definition}\label{Definition: upper_streak}
				Let a set $X$ be equipped with a relation $< \subseteq X \times \QQ$ and operations $+\colon X \times X \to X$ and $\cdot\colon X_{> 0} \times X_{> 0} \to X_{> 0}$, where
				\[X_{> 0} \dfeq \st{x \in X}{\xsome{q}{\QQ_{> 0}}{\lnot(x < q)}}.\]
				Suppose the following conditions hold:
				\begin{itemize}
					\item
						boundedness (from above): $\xall{a}{X}\xsome{q}{\QQ}{a < q}$,
					\item
						cotransitivity: for all $a \in X$ and $q, r \in \QQ$
						\[a < q \implies a < r \lor r < q,\]
					\item
						for all $a, b \in X$
						\[\all{q}{\QQ}{a < q \iff b < q} \implies a = b,\]
					\item
						the relation $<$ is open and its negation $\leq$ closed,
					\item
						$+$ makes $X$ into a commutative monoid (meaning we in particular have an additive unit which we also denote by $0$),
					\item
						for all $q, r \in \QQ$ and $a, b \in X$
						\[a < q \land b < r \implies a + b < q + r,\]
					\item
						for all $q \in \QQ$ and $a, b \in X$
						\[a + b < q \implies \some{r, s}{\QQ}{r + s < q \land a < r \land b < s},\]
					\item
						$\cdot$ makes $X_{> 0}$ into a commutative monoid (meaning we also have a multiplicative unit which we denote by $1$) and distributes over $+$,
					\item
						for all $q, r \in \QQ_{> 0}$ and $a, b \in X_{> 0}$
						\[a < q \land b < r \implies a \cdot b < q \cdot r,\]
					\item
						for all $q \in \QQ_{> 0}$ and $a, b \in X_{> 0}$
						\[a \cdot b < q \implies \some{r, s}{\QQ_{> 0}}{r \cdot s < q \land a < r \land b < s},\]
					\item
						``asymmetry'': $\lnot(\underbrace{1}_{\in X} < \underbrace{1}_{\in \QQ})$.
				\end{itemize}
				Then $(X, <, +, 0, \cdot, 1)$ is called an \df{upper streak}.
			\end{definition}
			
			The rest works out the same as for the lower streaks, and we have all of the following.
			
			\begin{definition}
				Let $X$, $Y$ be upper streaks. An \df{upper streak morphism} $f\colon X \to Y$ is a map with the property
				\[a < q \iff f(a) < q\]
				for all $q \in \QQ$, $a, b \in X$.
			\end{definition}
			
			\begin{definition}
				The \df{upper reals} $\ur$ are the terminal upper streak.
			\end{definition}
			
			\begin{theorem}
				$\uc(\QQ)$ (the set of all upper cuts) is a terminal upper streak.
			\end{theorem}
			
			To get the results in this subsection, we split streaks into two more general structures. We now consider how to put them back together again, to get streaks back.
			
			\begin{theorem}\label{Theorem: connection_between_streaks_and_their_onesided_variants}
				Being a streak is equivalent to being both a lower and an upper streak, the orders of which are joined by asymmetry and a cotransitivity condition. More precisely, the following holds.
				\begin{enumerate}
					\item\label{Theorem: connection_between_streaks_and_their_onesided_variants: streak_implies_lower_and_upper}
						Any streak is also a lower and an upper streak (for the usual operations and order).
					\item
						Suppose $X$ is equipped with addition, multiplication and both one-sided comparisons with rationals in a way which makes it a lower and an upper streak. Suppose additionally that the asymmetry condition
						\[\lnot(q < a \land a < q)\]
						and the contransitivity condition
						\[q < r \implies q < a \lor a < r\]
						hold for all $q, r \in \QQ$, $a \in X$. Then $X$ is a streak (for the usual definition of the strict order $a < b \dfeq \xsome{q}{\QQ}{a < q < b}$).
					\item
						In a setting where ``open'' implies ``not-not-stable''\footnote{There are many examples where this happens, such as classical mathematics, realizability models validating countable choice and Markov's principle (where ``open'' means \df{semidecidable}), sheaf models of synthetic topology where the Sierpi\'{n}ski object (which induces the intrinsic topology) is given (in the usual way) as the topology functor\ldots}, the previous item holds even without postulating the cotransitivity condition.
				\end{enumerate}
			\end{theorem}
			\begin{proof}
				\begin{enumerate}
					\item
						To a large part, this is just comparing Defintions~\ref{Definition: lower_streak}~and~\ref{Definition: upper_streak} with Proposition~\ref{Proposition: alt_streak}. The only things that need to be checked are the statements
						\[q < a + b \implies \some{r, s}{\QQ}{q < r + s \land r < a \land s < b},\]
						\[q < a \cdot b \implies \some{r, s}{\QQ_{> 0}}{q < r \cdot s \land r < a \land s < b},\]
						\[a + b < q \implies \some{r, s}{\QQ}{r + s < q \land a < r \land b < s},\]
						\[a \cdot b < q \implies \some{r, s}{\QQ_{> 0}}{r \cdot s < q \land a < r \land b < s}.\]
						We check only the first one; the others work similarly.
						
						Suppose $q < a + b$ in a streak $X$. Take some $t \in \QQ$ with $q < t < a + b$ and let $k \in \NN_{> 0}$ be large enough so that $\frac{1}{k} < \frac{t-q}{4}$. We then have $i, j \in \ZZ$ such that $a \in \intoo[X]{\frac{i-1}{k}}{\frac{i+1}{k}}$, $b \in \intoo[X]{\frac{j-1}{k}}{\frac{j+1}{k}}$. Set $r \dfeq \frac{i-1}{k}$, $s \dfeq \frac{j-1}{k}$. Thus $r < a$, $s < b$ and
						\[r + s = \tfrac{i+1}{k} + \tfrac{j+1}{k} - \tfrac{4}{k} > a + b - (t-q) > t - (t-q) = q.\]
					\item
						We note that everything is fine with the definition of $X_{> 0}$. Given any $x \in X$, bigger than the rational $0$, there exists $q \in \QQ$ with $0 < q < x$. By asymmetry $q \leq x$. Conversely, suppose there is $q \in \QQ$ with $0 < q \leq x$. By cotransitivity $0 < x \lor x < q$, but the second disjunct is in contradiction with $q \leq x$, so the first one must hold. We see that the definitions of $X_{> 0}$ from lower streaks and upper streaks match. In particular, there is no problem with suggesting that the same multiplication makes $X$ into a lower and an upper streak.
						
						Otherwise, all streak conditions from Proposition~\ref{Proposition: alt_streak} are clearly fulfilled.
					\item
						We claim that cotransitivity follows from other conditions. Take any $a \in X$ and $q, r \in \QQ$ with $q < r$. We want to prove $q < a \lor a < r$ which is an open statement, so by assumption not-not-stable, and it is therefore sufficient to derive a contradiction from its negation.
						
						Suppose then that its negation $a \leq q \land r \leq a$ holds. By transitivity $a \leq q$ and $q < r$ (hence $q \leq r$) imply $a \leq r$. Write $r = \frac{i-j}{k}$ where $i, j \in \NN$, $k \in \NN_{> 0}$. Then $a \leq r$ is equivalent to $k \cdot a + j \leq i$ and $r \leq a$ is equivalent to $i \leq k \cdot a + j$. Thus $k \cdot a + j = i$. We may without loss of generality assume that $q$ can be expressed with the same denominator; let $q = \frac{m-n}{k}$ with $m, n \in \NN$. In a similar way we obtain $k \cdot a + n = m$.
						
						The inequality $q < r$ means $m + j < i + n$. However, that is in contradiction with
						\[i + n = k \cdot a + j + n = m + j.\]
				\end{enumerate}
			\end{proof}
			
			\begin{remark}
				The four statements, presented in the proof of Theorem~\ref{Theorem: connection_between_streaks_and_their_onesided_variants}(\ref{Theorem: connection_between_streaks_and_their_onesided_variants: streak_implies_lower_and_upper}), are genuinly needed for the theory of lower and upper streaks to work, even though they didn't explicitly appear when studying streaks. This proof shows why: when we have comparison with rationals on both sides together, they are already implied by the weaker statements
				\[q < a \implies \xsome{r}{\QQ}{q < r < a} \qquad \text{and} \qquad a < q \implies \xsome{r}{\QQ}{a < r < q}.\]
			\end{remark}
			
			A few words about infinity. Since lower cuts are inhabited by definition, they are strictly larger than $-\infty$, and similarly upper cuts are strictly smaller than $\infty$. Thus the two-sided Dedekind cuts all represent finite numbers. However, this is no longer the case when we consider just one-sided cuts. One of the elements of $\lc(\QQ)$ is $\QQ$ itself, representing $\infty$, and similarly $\QQ \in \uc(\QQ)$, representing $-\infty$. Sometimes we want to restrict to finite lower/upper reals\footnote{For example, the property as basic as $a + x = b + x \implies a = b$ does not hold when we allow infinity: $\infty = \infty + 1$, but not $0 = 1$.}, which is done by simply additionally postulating for cuts that their complements are also inhabited. On the other hand, we sometimes want (lower/upper) reals with \emph{both} infinities included --- in this case we drop any inhabitedness condition.
			
			Wahtever the variant, the lower/upper reals are useful because they satisfy a form of order completness. Recall that in classical mathematics the reals are \df{Dedekind complete} in the following sense: every inhabited subset of $\RR$, bounded from above, has a supremum (and since $\RR$ has subtraction, one can easily prove the analogous statement for infima). Nothing of the sort holds constructively: while $\RR$ is necessarily a lattice (as we have seen), it is not closed under taking more general suprema. Even the simple statement that every increasing binary sequence\footnote{A \df{binary sequence} is a map $\NN \to \{0, 1\}$.}  has a supremum is equivalent to \lpo[].
			
			On the other hand, if we have an inhabited subset $A \subseteq \lc(\QQ)$, then $\bigcup A$ satisfies all the conditions for a lower cut, with the possible exception of the topological ones. Thus in case of $\pst = \optp = \cltp$ the lower reals are closed under taking \emph{arbitrary} inhabited suprema. Even in general, we at least have inhabited countable suprema (more generally, the inhabited \df{overt} ones). Similarly upper cuts are closed under taking such infima.
			
			To summarize, $\RR$ is useful because of its field strucure, $\lr$ is useful because it is closed under more general suprema and $\ur$ is useful because it is closed under more general infima. Classically we can pack all of this into one set, but constructively we cannot, and so must choose the right tool for each particular job.

	
	\newpage
	\bibliographystyle{plain}
	\bibliography{Reals_Bibliography}

\begin{thebibliography}{10}

\bibitem{Bauer_A_2000:_the_realizability_approach_to_computable_analysis_and_t%
opology}
A.~Bauer.
\newblock {\em The Realizability Approach to Computable Analysis and Topology}.
\newblock PhD thesis, School of Computer Science, Carnegie Mellon University,
  Pittsburgh, 2000.

\bibitem{Bauer_A_Taylor_P_2009:_the_dedekind_reals_in_abstract_stone_duality}
A.~Bauer and P.~Taylor.
\newblock The {D}edekind reals in abstract {S}tone duality.
\newblock {\em Mathematical Structures in Comp.\ Sci.}, 19(4):757--838, 2009.

\bibitem{bell2008primer}
J.L. Bell.
\newblock {\em A Primer of Infinitesimal Analysis}.
\newblock Cambridge University Press, 2008.

\bibitem{Bishop_E_Bridges_D_1985:_constructive_analysis}
E.~Bishop and D.~Bridges.
\newblock {\em Constructive Analysis}.
\newblock Springer-Verlag, New York, 1985.

\bibitem{Bridges_DS_Richman_F_1987:_varieties_of_constructive_mathematics}
D.S. Bridges and F.~Richman.
\newblock {\em Varieties of Constructive Mathematics}, volume~97 of {\em
  Lecture Note Ser.}
\newblock London Math.~Soc., London, 1987.

\bibitem{Coquand_T_Sambin_G_Smith_JM_Valentini_S_2003:_inductively_generated_f%
ormal_topologies}
T.~Coquand, G.~Sambin, J.M. Smith, and S.~Valentini.
\newblock Inductively generated formal topologies.
\newblock {\em Annals of Pure and Applied Logic}, 124:71--106, 2003.

\bibitem{Escardo_M_Simpson_A_2001:_a_universal_characterization_of_the_closed_%
euclidean_interval}
M.H. Escardo and A.K. Simpson.
\newblock A universal characterization of the closed euclidean interval.
\newblock In {\em Logic in Computer Science, 2001. Proceedings. 16th Annual
  IEEE Symposium on}, pages 115--125, 2001.

\bibitem{Gierz_G_Hoffmann_KH_Keimel_K_Lawson_JD_Mislove_MW_Scott_DS_2003:_cont%
inuous_lattices_and_domains}
G.~Gierz, K.H. Hoffmann, K.~Keimel, J.D. Lawson, M.W. Mislove, and D.S. Scott.
\newblock {\em Continuous Lattices and Domains}.
\newblock Encyclopedia of Mathematics and Its Applications. Cambridge
  University Press, 2003.

\bibitem{Johnstone_PT_2002:_stone_spaces}
P.T. Johnstone.
\newblock {\em Stone Spaces}.
\newblock Cambridge University Press, 2002.
\newblock First published in 1982.

\bibitem{lane1998categories}
S.M. Lane.
\newblock {\em Categories for the Working Mathematician}.
\newblock Graduate Texts in Mathematics. Springer, 1998.

\bibitem{Lesnik_D_2010:_synthetic_topology_and_constructive_metric_spaces}
D.~Le\v{s}nik.
\newblock {\em Synthetic Topology and Constructive Metric Spaces}.
\newblock PhD thesis, University of Ljubljana, 2010.

\bibitem{Lubarsky:2007:CCC:1224238.1224293}
Robert~S. Lubarsky.
\newblock On the cauchy completeness of the constructive cauchy reals.
\newblock {\em Electron. Notes Theor. Comput. Sci.}, 167:225--254, January
  2007.

\bibitem{moerdijk1991models}
I.~Moerdijk and G.E. Reyes.
\newblock {\em Models for smooth infinitesimal analysis}.
\newblock Springer-Verlag, 1991.

\bibitem{Myhill_J_1975:_constructive_set_theory}
J.~Myhill.
\newblock Constructive set theory.
\newblock {\em Journal of Symbolic Logic}, 40(3):347--382, 1975.

\bibitem{Richman_F_2008:_real_numbers_and_other_completions}
F.~Richman.
\newblock Real numbers and other completions.
\newblock {\em Mathematical Logic Quarterly}, 54(1):98--108, 2008.

\bibitem{Troelstra_AS_Dalen_D_1988:_constructivism_in_mathematics_volume_2}
A.S. Troelstra and D.~van Dalen.
\newblock {\em Constructivism in Mathematics, Volume 2}, volume 123 of {\em
  Studies in Logic and the Foundations of Mathematics}.
\newblock North-Holland, Amsterdam, 1988.

\end{thebibliography}

\end{document}